\documentclass[11pt,oneside]{amsart}
\usepackage[margin=2.5cm]{geometry}
\usepackage{amsmath,amssymb,amsfonts,amsthm,mathrsfs}
\usepackage[nobysame]{amsrefs} 
\usepackage{epic}
\usepackage{amsopn,amscd,graphicx}
\usepackage{color,transparent} 
\usepackage[usenames, dvipsnames]{xcolor}
\usepackage
[colorlinks, breaklinks,
bookmarks = false,
linkcolor = NavyBlue,
urlcolor = ForestGreen,
citecolor = ForestGreen,
hyperfootnotes = false
]
{hyperref}
\usepackage{scalerel}
\usepackage{stackengine,wasysym}
\usepackage{palatino, mathpazo}
\usepackage{multirow}
\usepackage[all,cmtip]{xy}
\usepackage{stmaryrd}
\usepackage{tikz-cd}
\usepackage{yhmath}
\usepackage{enumerate}
\usepackage[dvipsnames]{xcolor}
 1
 1
 1
\usepackage{comment}

\newtheorem{theorem}{Theorem}[section]
\newtheorem{corollary}[theorem]{Corollary}
\newtheorem{proposition}[theorem]{Proposition}

\newtheorem{lemma}[theorem]{Lemma}

\theoremstyle{definition}
\newtheorem{definition}[theorem]{Definition}
\newtheorem{example}[theorem]{Example}
\newtheorem{remark}[theorem]{Remark}

\newtheoremstyle{named}{}{}{\itshape}{}{\bfseries}{.}{.5em}{\thmnote{#3}#1}
\theoremstyle{named}

\numberwithin{equation}{subsection}

\DeclareMathOperator{\Dom}{Dom}
\DeclareMathOperator{\supp}{supp}
\DeclareMathOperator{\Ker}{Ker}

\newcommand{\R}{\mathbb{R}}
\newcommand{\C}{\mathbb{C}}
\newcommand{\N}{\mathbb{N}}
\newcommand{\Z}{\mathbb{Z}}
\newcommand{\ol}{\overline}
\newcommand{\domain}{D}
\newcommand{\divisor}[1]{\left[\operatorname{Div}\!\left(#1\right)\right]}
\newcommand{\pqboundaryforms}[3]{\Omega^{#1,#2}(\overline{#3})}
\newcommand{\nnbforms}[2]{\pqboundaryforms{#1}{#1}{#2}}

\def\cC{\mathscr{C}}

\let\Re\relax
\DeclareMathOperator{\Re}{Re}
\let\Im\relax
\DeclareMathOperator{\Im}{Im}
\def\cCc{\mathscr{C}^\infty_c}

\title[Induced Fubini-Study metrics and zeros of random
CR functions]{Induced Fubini-Study metrics on strictly pseudoconvex CR manifolds
and zeros of random
CR functions}

\begin{document}
\author[Hendrik Herrmann]{Hendrik Herrmann}
\address{Bergische Universität Wuppertal, Fakultät 4 - Mathematik und Naturwissenschaften, 
Gaußstraße 20, 42119 Wuppertal, Germany}
\thanks{Hendrik Herrmann and George Marinescu are  partially supported
by the ANR-DFG project QuaSiDy (Project ID 490843120).
George Marinescu is partially supported by the 
DFG Priority Program 2265 `Random Geometric Systems' (Project-ID 422743078).
Chin-Yu Hsiao is partially supported by Taiwan Ministry of Science and Technology projects
108-2115-M-001-012-MY5, 109-2923-M-001-010-MY4.
Wei-Chuan Shen is supported by the SFB/TRR 191 "Symplectic Structures in Geometry, 
Algebra and Dynamics", funded by the DFG (Projektnummer 281071066 – TRR 191).}
\email{hherrmann@uni-wuppertal.de or post@hendrik-herrmann.de}

\author[Chin-Yu Hsiao]{Chin-Yu Hsiao}
\address{Institute of Mathematics, Academia Sinica, Astronomy-Mathematics Building, 
No. 1, Sec. 4, Roosevelt Road, Taipei 10617, Taiwan}
\email{chsiao@math.sinica.edu.tw or chinyu.hsiao@gmail.com}

\author[George Marinescu]{George Marinescu}
\address{Universit{\"a}t zu K{\"o}ln,  Mathematisches Institut,
Weyertal 86-90, 50931 K{\"o}ln, Germany
\newline\mbox{\quad}\,\,Institute of Mathematics `Simion Stoilow', 
Romanian Academy, Bucharest, Romania}
\email{gmarines@math.uni-koeln.de}

\author[Wei-Chuan Shen]{Wei-Chuan Shen}
\address{Universit{\"a}t zu K{\"o}ln,  Mathematisches Institut,
    Weyertal 86-90,   50931 K{\"o}ln, Germany}
\email{wshen@math.uni-koeln.de}
\date{\today}
\begin{abstract}
Let $X$ be a compact strictly pseudoconvex embeddable 
Cauchy-Riemann manifold and let $T_P$ be the Toeplitz operator on $X$ 
associated with a first-order pseudodifferential operator $P$. 
In our previous work we established the asymptotic expansion
for $k$ large of the kernel of the operators
$\chi(k^{-1}T_P)$, where $\chi$ is a smooth cut-off function supported in the positive real line.  
By using these asymptotics, we show in this paper that 
\(X\) can be projectively embedded
by maps with components of the form 
$\chi(k^{-1}\lambda)f_\lambda$, where
$\lambda$ is an eigenvalue of $T_P$ and
$f_\lambda$ is a corresponding eigenfunction.
We establish the asymptotics of the pull-back
of the Fubini-Study metric by these maps and
we obtain the distribution of the zero divisors of 
random Cauchy-Riemann functions.
We then establish a version of the Lelong-Poincar\'e 
formula for domains with boundary and obtain
the distribution of the zero divisors of random holomorphic functions 
on strictly pseudoconvex domains.  
\end{abstract}

\maketitle
\tableofcontents
\section{Introduction and Main Results}
{\tiny{{\color{white}{\subsection{ }}}}}
In our previous work \cite{HHMS23} we introduced a semi-classical approach
to the spectral theory of Toeplitz operators of Boutet de Monvel 
and Guillemin \cite{BG81} and gave several applications to
the CR geometry of strictly pseudoconvex Cauchy-Riemann manifolds.
In this paper we give further applications of the semi-classical
approach: projective embedding, asymptotics of the induced
Fubini-Study metrics and distribution of zeros of
random Cauchy--Riemann and holomorphic functions. 

A famous result due to Kodaira~\cite{Ko54} states that
a compact complex manifold carrying a
positive holomorphic line bundle $L$ is holomorphically embedded
into the complex projective space by the Kodaira maps $\Phi_k$ 
given by the high tensor powers $L^k$. A theorem of Tian~\cite{Ti90} 
shows that the Kodaira maps $\Phi_k$ are asymptotically isometric 
with respect to the Fubini-Study metric and the K\"ahler forms 
given by the curvatures of $L^k$ as $k\to+\infty$. 
These results yield the asymptotic distribution
of zero sets for random holomorphic sections 
\cite{CM15, DMS12, DS06, DLM23, NV98, BCM20, SZ99}, 
see also \cite{BCHM18} for a survey. 

We will establish analogues of these results in the
setting of Cauchy--Riemann manifolds (CR manifolds for short). 
A classical result due to Boutet de Monvel~\cite{Bou75} 
states that any compact strictly pseudoconvex CR manifold 
of dimension greater or equal than five can be CR embedded 
into the complex euclidean space. 
In dimension three there are non-embeddable
compact strictly pseudoconvex CR manifolds 
(see e.\,g.\ Burns~\cite{Bu:77}, where the boundary 
of the non-fillable example of strictly pseudoconcave manifold 
by Grauert~\cite{Gr94}, Andreotti-Siu~\cite{AS70} and Rossi~\cite{R65}
is shown to be non-embeddable). Lempert~\cite{Lem92} 
showed the embeddabillity in any dimension 
by assuming the existence of a CR Reeb vector field.
The embeddability of Sasakian manifolds was established in \cite{MY07}.
Using the semiclassical form introduced in \cite{HHMS23}
of the spectral theory of Toeplitz operators \cite{BG81},
we obtained in \cite{HHMS23} embedding results for 
CR manifolds and we related their CR Reeb dynamics 
to those of spheres in the complex euclidean space.

In this work we study embeddings of compact strictly pseudoconvex CR 
manifolds into the complex projective space linking the Fubini-Study 
metric with CR geometric data of the underlying manifold. 
Then, in analogy to the complex case, we obtain results on
the asymptotic distributions of zero sets for random CR functions
which turns out to be directly related to the study of zero sets for
random holomorphic functions on strictly pseudoconvex domains. 
In the following, we explain the precise set-up and formulate the main results.

We consider   an orientable 
compact strictly pseudoconvex Cauchy--Riemann (CR for short) manifold $(X,T^{1,0}X)$
of dimension $2n+1$, $n\geq1$, with volume form \(dV\)
such that the Kohn Laplacian on $X$ has closed range in 
$L^2(X)$, where \(L^2(X)\) denotes the space of square integrable functions on \(X\). 
   Denote by \(\Pi\colon L^2(X)\to L^2(X)\) the Szeg\H{o}-projector and by \(H_b^0(X)=\Pi L^2(X)\) the space of \(L^2\) CR functions. Choose a contact form \(\xi\) on \(X\), that is, \(\xi\in\Omega^1(X)\) is a smooth real non-vanishing one form such that \(\ker \xi = \Re(T^{1,0}X)\). We can and we will choose \(\xi\) in a way that the Levi form associated to \(\xi\) is positive definite (see  Definition~\ref{D:Leviform}). Let  $P\in L^1_{\mathrm{cl}}(X)$ be a  first order formally self-adjoint classical pseudodifferential operator which is transversally elliptic meaning that  the principal symbol $\sigma_P$ of $P$ restricted to the symplectic cone
$\Sigma:=\{(x,t\xi(x)):x\in X,~t>0\}\subset T^*X$ is everywhere positive. One can then define the so called characteristic contact form \(\alpha_P\) with respect to \(P\) which is given by
\begin{eqnarray}\label{eq:CharacteristikContactFormIntro}
	\alpha_P:=\frac{\xi}{\sigma_P(\xi)}.
\end{eqnarray}
It turns out that \(\alpha_P\) does not depend on the choice of \(\xi\).
 We consider the Toeplitz operator $T_P:=\Pi P\Pi:\cC^\infty(X)\to\cC^\infty(X)$ 
associated to $P$. One has that $T_P$ has a self-adjoint $L^2$-extension.
More precisely, $T_P: \Dom T_P\subset L^2(X)\to L^2(X)$
is self-adjoint, with $\Dom T_P=\{u\in L^2(X): \Pi P\Pi u\in L^2(X)\}$, where $\Pi P\Pi u$ is defined in the sense of distributions. 
Since \cite{BG81}*{Proposition 2.14} applies here, 
the spectrum of $T_P:\operatorname{Dom}T_P\to L^2(X)$ consists only of 
a discrete set of real eigenvalues, which is bounded from below and has only $+\infty$ 
as accumulation point. Furthermore, given \(\lambda\neq 0\) we find that \(\ker (T_P-\lambda\operatorname{Id})\) is a finite dimensional subspace of \(L^2(X)\) contained in \(H_b^0(X)\cap \mathscr{C}^\infty(X)\).  So let \(0<\lambda_1\leq\lambda_2\leq\ldots\) be the positive eigenvalues of \(T_P\) counting multiplicity and let \(\{f_j\}_{j\in\N}\subset H_b^0(X)\) be a respective orthonormal system of eigenvectors. In particular, we have \(T_Pf_j=\lambda_jf_j\), \((f_j,f_j)=1\) and \((f_j,f_\ell)=0\) for all \(j,\ell\in \N\), \(j\neq \ell\). We consider a function $\chi\in\cC^\infty_c(\R)$ with 
$\text{supp}\chi\subset(0,+\infty)$ and define for $k>0$ the function
\begin{equation}\label{eq:chik}
	\chi_k:\R\to\C\,,\quad \chi_k(t)=\chi\left(k^{-1}t\right).
\end{equation}
Then $\chi_k\in\mathscr{C}^\infty_c(\R)$, $\text{supp}\chi_k=k\text{supp}\chi\subset(0,+\infty)$.
In~\cite{HHMS23} the functional calculus \(\chi_k(T_P)\) was studied and it was shown that its Schwartz kernel given by
\begin{eqnarray}\label{eq:defFunctionalCalculusTP}
\chi_k(T_P)(x,y)=\sum_{j=1}^\infty \chi_k(\lambda_j)f_j(x)\overline{f_j(y)}
\end{eqnarray}
has an asymptotic expansion for \(k\to+\infty\) (see Theorem~\ref{thm:ExpansionMain} for more details). Moreover, first applications were obtained with respect to weighted spectral measures and almost spherical embeddings. In this paper we study further applications concerning projective embeddings and equidistribution results for CR-currents and zeros of random holomorphic functions.  We assume \(\operatorname{supp}(\chi)\subset (\delta_1,\delta_2)\) for some \(0<\delta_1<\delta_2\), \(\chi\not\equiv 0\) and put \(N_k=\#\{j\mid \lambda_j\leq\delta_2k\}\). We consider the map
\begin{eqnarray}
	F_k\colon X\to \C^{N_k},\,\,\, F_k(x)=(\chi_k(\lambda_1)f_1,\ldots,\chi_k(\lambda_{N_k})f_{N_k}).
\end{eqnarray}  
Furthermore, given a non-negative function \(\eta\in \mathscr{C}^\infty_c(\R_+,[0,\infty))\), \(\eta\not\equiv 0\), \(\R_+:=(0,\infty)\), and \(j\in\N_0\) we define
\begin{eqnarray}
	\tau_j(\eta)=\int_{\R}t^{n+j}\eta(t)dt.
\end{eqnarray}
We note that \(t\mapsto \tau_0(\eta)^{-1} t^n\eta(t)\) defines a probability distribution on \(\R_+\) such that its momenta are given by \(\tau_j(\eta)/\tau_0(\eta)\) for \(j\in\N\). In particular, we put 
\begin{eqnarray}\label{eq:mvANDvarIntro}
\operatorname{mv}(\eta):=\frac{\tau_1(\eta)}{\tau_0(\eta)},\,\,\,\,\operatorname{var}(\eta):=\frac{\tau_2(\eta)}{\tau_0(\eta)}-(\operatorname{mv}(\eta))^2 
\end{eqnarray}
to denote its mean value and variance respectively.
We now state our first main result.
\begin{theorem}\label{thm:ProjectiveEmbeddingIntro}
	Let $(X,T^{1,0}X)$ be an orientable 
	compact strictly pseudoconvex Cauchy--Riemann manifold
	of dimension $2n+1$, $n\geq1$,
	such that the Kohn Laplacian on $X$ has closed range in 
	$L^2(X)$. Let $\xi$ be a contact form on $X$ such that its associated Levi form
	$\mathcal{L}$ is positive definite. Consider a formally self-adjoint first order pseudodifferential operator 
	$P\in L^1_\mathrm{cl}(X)$ with
	$\sigma_P(\xi)>0$ on $X$ and denote by \(T_P\) its Toelplitz operator. Let \(0<\lambda_1\leq\lambda_2\leq\ldots\) be the positive eigenvalues of \(T_P\) counting multiplicity and let \(\{f_j\}_{j\in\N}\subset H_b^0(X)\cap \mathscr{C}^\infty(X)\) be a respective orthonormal system of eigenvectors. 
	Given \(k>0\), \(0<\delta_1<\delta_2\) and $\chi\in\cC^\infty_c((\delta_1,\delta_2))$, $\chi\not\equiv 0$, consider the the map
	 \[F_k\colon X\to \C^{N_k},\,\,\, F_k(x)=(\chi_k(\lambda_1)f_1,\ldots,\chi_k(\lambda_{N_k})f_{N_k}),\]
	where \(N_k=\#\{j\mid \lambda_j\leq\delta_2k\}\) and $\chi_{k}(\lambda):=\chi\left(k^{-1}\lambda\right)$. 
	There exists \(k_0>0\) such that the map 
	\[[F_k]\colon X\to \C\mathbb{P}^{N_k-1},\,\,\,\,[F_k](x):=[F_k(x)],\]  is a well defined CR embedding for all \(k\geq k_0\). Furthermore, \([F_k]^*ds^2_{FS}\) has an asymptotic expansion for \(k\to+\infty\) in \(\mathscr{C}^\infty\)-topology. Here \(ds^2_{FS}\) denotes the Hermitian line element with respect to the Fubini-Study metric. More precisely, there exist smooth sesquilinear forms \(r_j\colon \C TX\times \C TX\to \C\), \(j\in\N_0\), such that for any \(M,\ell\in\N\) there is a constant \(C_{M,\ell}>0\) with
	\begin{eqnarray}\label{eq:AsymptoticsProjectiveEmbeddingIntro}
		\left\|[F_k]^*ds^2_{FS} -\sum_{j=0}^Mk^{2-j}r_j\right\|_{\mathscr{C}^\ell(X)}\leq C_{M,\ell}k^{-M+1},\,\,\,\text{ for all }k\geq k_0.
	\end{eqnarray}
	In addition,  we have \(r_0=\operatorname{var}(|\chi|^2)\alpha_P\otimes \alpha_P\) and \(r_1(Z,W)=-i\operatorname{mv}(|\chi|^2)d\alpha_P(Z,\overline{W})\) for all \(Z,W\in T_x^{1,0}X\), \(x\in X\). Here \(\alpha_P\) denotes the characteristic contact form with respect to \(P\) as in~\eqref{eq:CharacteristikContactFormIntro} and  \(\operatorname{mv}(|\chi|^2)\), \(\operatorname{var}(|\chi|^2)\)  are given by~\eqref{eq:mvANDvarIntro}. 
\end{theorem}
Let \(g_{FS}:=\Re(ds^2_{FS})\) and \(\omega_{FS}:=-\Im(ds^2_{FS})\)  denote the Riemannian metric and the Kähler form on  \(\C\mathbb{P}^N\) induced by the Fubini-Study metric. The asymptotic expansion of \([F_k]^*ds^2_{FS}\) in Theorem~\ref{thm:ProjectiveEmbeddingIntro} then leads to  results concerning the pullback of \(g_{FS}\) and \(\omega_{FS}\) on \(\C\mathbb{P}^{N_k}\) under \([F_k]\). See for example Lemma~\ref{lem:PropertiesOfgkForX} together with Lemma~\ref{lem:hfHFRelatedFubinStudy} for a refined statement on the pullback of \(g_{FS}\) under  \([F_k]\).
\begin{remark}\label{rmk:ClosedRangeEmbedabble}
	We note that by Boutet de Monvel’s proof~\cite{Bou74,Bou75} and Kohn’s argument~\cite{Koh86}, for any compact strictly pseudoconvex
	CR manifold \(X\), the closed range condition is equivalent to the CR embeddability of \(X\) into the complex euclidean space. We also recall that the closed range condition holds for any compact  strictly pseudoconvex CR manifold when \(n\geq 2\) (cf.\ \cite{Bou75}), or when \(n=1\) and \(X\) admits a transversal CR \(\R\)-action  (cf.\ \cite{Lem92, MY07}).
\end{remark}
\begin{remark}
	We note that the statement that the map \([F_k]\) in Theorem~\ref{thm:ProjectiveEmbeddingIntro} is an embedding when  \(k>0\) is large enough cannot be directly obtained from the almost spherical embedding result in~\cite{HHMS23}. Furthermore, with respect to the results on equidistribution below it was necessary to understand  how fast \([F_k]\) separates points when \(k\) becomes large. So the method in the proof of Theorem~\ref{thm:ProjectiveEmbeddingIntro} differs from the strategy used in~\cite{HHMS23} to obtain the almost spherical embedding result there. See Remark~\ref{rmk:StrategyOfProofEmbedding} for further explanations. 
\end{remark}

An important class of strictly pseudoconvex CR manifolds are Sasakian manifolds. A strictly pseudoconvex CR manifold \((X,T^{1,0}X)\) together with a real transversal CR vector field \(\mathcal{T}\) is called a Sasakian manifold. Recall that a real vector field is transversal if it is transversal to \(\Re(T^{1,0}X)\) and CR if its flow preserves \(T^{1,0}X\). From Theorem~\ref{thm:ProjectiveEmbeddingIntro} we can deduce the following embedding result for Sasakian manifolds.
\begin{corollary}\label{cor:ProjectiveEmbeddingSasakianIntro}
Let \((X,T^{1,0}X)\) be a compact Sasakian manifold with transversal CR vector field \(\mathcal{T}\) and consider the transversal CR \(\R\)-action on \(X\) induced by the flow of \(\mathcal{T}\).  Let \(dV\) be an \(\R\)-invariant volume form on \(X\) and choose a contact form \(\xi\) for \((X,T^{1,0}X)\) such that its associated Levi form \(\mathcal{L}\) is positive definite.   Assume that the contraction of \(\xi\) with \(\mathcal{T}\) is positive and put \(P:=-i\mathcal{T}\). 
 Then \((X,T^{1,0}X)\), \(P\), \(dV\) satisfy the assumptions 
in Theorem~\ref{thm:ProjectiveEmbeddingIntro}. 
Hence, given \(\chi\in\cC_c^\infty((\delta_1,\delta_2))\), \(\chi\not\equiv 0\),  
for some \(0<\delta_1<\delta_2\) there exists \(k_0>0\) such that 
the associated map \([F_k]:X\to \C\mathbb{P}^{N_k-1}\)
in Theorem~\ref{thm:ProjectiveEmbeddingIntro}
defines a CR embedding of \(X\) into \(\C\mathbb{P}^{N_k-1}\)
for all \(k\geq k_0\) satisfying the asymptotics
\eqref{eq:AsymptoticsProjectiveEmbeddingIntro}.
Furthermore, the map \([F_k]\) is equivariant with respect to the 
CR \(\R\)-action on \(X\) and the weighted holomorphic 
\(\R\)-action on \(\C\mathbb{P}^{N_k-1}\) given by
\[\R\times \C\mathbb{P}^{N_k-1}\ni(t,[z_1,\ldots,z_{N_k}])
\mapsto [e^{i\lambda_1t}z_1,\ldots,e^{i\lambda_{N_k}t}z_{N_k}]\in
\C\mathbb{P}^{N_k-1}.\] 
with \(\lambda_1,\lambda_2,\ldots\) as in Theorem~\ref{thm:ProjectiveEmbeddingIntro}.
\end{corollary}

\begin{remark}
Equivariant projective CR embeddings for Sasakian manifolds and related objects were studied in~\cite{HHL22} and~\cite{HLM21} before. It turns out (see \cite[Remark 1.6]{HHL22}) that the existence of an equivariant projective CR embedding can be directly deduced  from equivariant euclidean CR embedding results as in \cite{Lem92, MY07,OV07,vCoe11}. Hence, the non-trivial statement in  Corollary~\ref{cor:ProjectiveEmbeddingSasakianIntro} is that the projectivization \([F_k]\)  of the map \(F_k:X\to \C^{N_k}\) becomes an equivariant CR embedding for \(k\) large enough and that  \([F_k]^*ds_{\operatorname{FS}}\) admits an asymptotic expansion for \(k\to\infty\). 
	  Furthermore, we would like to point out that the related results in~\cite{HHL22} and~\cite{HLM21} rely on the study of CR sections of positive CR line bundles. Using Theorem~\ref{thm:ProjectiveEmbeddingIntro} we obtain  the equivariant CR embedding without considering positive CR line bundles.
\end{remark}

The projective embeddability of \(X\) as in 
Theorem~\ref{thm:ProjectiveEmbeddingIntro} 
is related to the study of equilibrium measures in CR geometry 
and asymptotic distribution of zero sets for random holomorphic functions. 
Let \(f\in H_b^0(X)\cap \mathscr{C}^\infty(X)\) be a smooth CR function
such that zero is a regular value of the map \(f\colon X\to \C\). 
It turns out that under the assumptions on \(f\) we
have that \(df/f\) is integrable on \(X\) and the zero divisor of \(f\)
is given by \(\divisor{f}=(2\pi i)^{-1}d(df/f)\) in the sense of distributions, 
that is, 
\begin{eqnarray}
	\left(\divisor{f}, \psi\right) =\frac{1}{2\pi i}\int_X \frac{df}{f}\wedge d\psi
\end{eqnarray}
for all \(\psi\in \Omega^{2n-1}(X)\) (cf.\ \cite{HS92}, 
see also Theorem~\ref{thm:PoincareLelongForCR} and
Remark~\ref{rmk:DivisorsOfVectorvaluedMaps}). 
We consider \(\divisor{f}\) as distribution valued random variable.
 For each \(k>0\) we consider the space 
 \(A_k\subset H_b^0(X)\cap \mathscr{C}^\infty(X)\) 
 given by the linear span of the elements 
 \(f_1,\ldots,f_{N_k}\) with \(f_j,N_k\) as in Theorem~\ref{thm:ProjectiveEmbeddingIntro}. 
 Given \(\chi\in\mathscr{C}^\infty_c(\R_+)\) as in Theorem~\ref{thm:ProjectiveEmbeddingIntro}
 we consider the probability measure \(\mu_k\) on 
 \(A_k\) induced by the standard complex Gaussian measure 
 \(\mu^G\) on \(\C^{N_k}\) and the map 
 \begin{eqnarray}\label{eq:defAkMukCRIntro}
\C^{N_k} \ni a \mapsto \sum_{j=1}^{N_k}a_j\chi(k^{-1}\lambda_j)f_j \in A_k.
 \end{eqnarray} 
In other words, we consider random CR functions of the form
\[f=a_1\chi(k^{-1}\lambda_1)f_1+\ldots
+a_{N_k}\chi(k^{-1}\lambda_{N_k})f_{N_k},\]
where \(a_j\) are i.i.d standard complex Gaussian random variables, and study  \(\divisor{f}\), seen as a distribution valued random variable, especially in the case when \(k\) becomes large. Therefore, given \(k>0\) we consider the expectation value  \(\mathbb{E}_k(\divisor{f})\) of \(\divisor{f}\), which is - in case of existence - defined by
\begin{eqnarray}\label{eq:ExpectationValueIntroductionCR}
	\left(\mathbb{E}_k\left(\divisor{f}\right),\psi\right):=\int_{A_k}\left(\divisor{f}, \psi\right) d\mu_k(f),\,\,\,\,\psi\in \Omega^{2n-1}(X).
\end{eqnarray}
\begin{theorem}\label{thm:ExpectationValueCRDistributionIntro}
Under the same assumptions as in  Theorem~\ref{thm:ProjectiveEmbeddingIntro} and with the notations above there exists \(k_0>0\) such that for each \(k\geq k_0\) we have that \(\divisor{f}\) is a well-defined  current for all \(f\) outside a  subset in \(A_k\) of \(\mu_k\)-measure zero and the expectation value \(\mathbb{E}_k(\divisor{f})\) given by \eqref{eq:ExpectationValueIntroductionCR}   exists 
with \(\mathbb{E}_k\left(\divisor{f}\right)=d\beta_k\) where \(\beta_k \) is the smooth one form given by 
\[\beta_k(x)=\frac{d_x\eta_k(T_P)(x,y)|_{x=y}}{2\pi i\eta_k(T_P)(x,x)}\]
for all \(x\in X\) and \(k\geq k_0\). Here \(\eta_k(T_P)\) is given by~\eqref{eq:defFunctionalCalculusTP} with \(\eta=|\chi|^2\) and \(\divisor{f}\) denotes the zero divisor of \(f\). Furthermore, the probability space \((A_k,\mu_k)\) is given by~\eqref{eq:defAkMukCRIntro}.
\end{theorem}
From the results in~\cite{HHMS23} (see Lemma~\ref{lem:ExpansionOfBk} in this paper) we obtain as a direct consequence that 
\[k^{-1}\mathbb{E}_k\left(\divisor{f}\right)=\frac{\text{mv}(|\chi|^2)}{2\pi}d\alpha_P+O(k^{-1})\,\,\text{ in }\mathscr{C}^\infty\text{-topology}\] 
 which leads to a convergence result for the normalized expected measures \(k^{-1}\mathbb{E}_k\left(\divisor{f}\right)\) weakly in the sense of distributions when \(k\) goes to \(+\infty\). It turns out that this behavior holds in an even stronger sense. Therefore, consider \(k\in\N\) and put
 \(A_\infty=\Pi_{k=1}^\infty A_k\), \(d\mu_{\infty}=\Pi_{k=1}^\infty d\mu_k\). We have the following.
\begin{theorem}\label{thm:ConvergenceStrongCFIntro}
	Under the same assumptions as in Theorem~\ref{thm:ProjectiveEmbeddingIntro}, with the notations above and \(k_0>0\) as in Theorem~\ref{thm:ExpectationValueCRDistributionIntro} we have for \(\mu_\infty\)-almost every \(f=(f_k)_{k\in \N}\in A_{\infty}\) that \((\divisor{f_k}, \psi)\) is well defined for all \(\psi\in \Omega^{2n-1}(X)\) and \(k\geq k_0\). Furthermore, given  \(\psi\in \Omega^{2n-1}(X)\), we have for \(\mu_\infty\)-almost every \(f=(f_k)_{k\in \N}\in A_{\infty}\) that 
	\begin{eqnarray}\label{eq:ConvergenceStrongCFIntro}
		\lim_{k\to\infty} \left(k^{-1}\divisor{f_k}, \psi\right) = \frac{\operatorname{mv}(|\chi|^2)}{2\pi}\int_{X}d\alpha_P\wedge\psi.
	\end{eqnarray}
\end{theorem}
\begin{remark}\label{rmk:ConvergenceStrongCFIntro}
	One can further prove the following: The probability that for \(f_k\in A_k\)  the term \(\left(k^{-1}\divisor{f_k}, \psi\right)\) differs from the right-hand side of~\eqref{eq:ConvergenceStrongCFIntro} more than \(1/\sqrt{k}\) is an \(O(1/\sqrt{k})\) uniformly in \(\psi\in \Omega^{2n-1}(X)\) with \(\|\psi\|_{\mathscr{C}^1(X,\Lambda_\C^{2n-1}T^*X)}\leq 1\). See Corollary~\ref{cor:ZerodistributionCRMfdSequence} for more details.
\end{remark}
\begin{example}
	Let \(L\to M\) be a positive holomorphic line bundle with hermitian metric \(h\) over a complex compact manifold \(M\) of complex dimension \(n\geq 1\). It turns out that the circle bundle \(X:=\{v\in L^*\mid |v|_{h^*}=1\}\), where \(L^*\) denotes the dual of \(L\) and \(h^*\) the Hermitian metric induced by \(h\), is a compact strictly pseudoconvex CR manifold satisfying the assumptions in Theorem~\ref{thm:ProjectiveEmbeddingIntro}. Furthermore, the fiberwise \(\C^*\)-action on \(L^*\) induces a transversal CR \(S^1\)-action on \(X\) with infinitesimal generator denoted by \(\mathcal{T}\). It follows that \(\mathcal{T}\) defines a Reeb vector field on \(X\) with respect to some contact form \(\xi\) which is related to the Chern connection for \((L,h)\to M\). In particular, we have that the pullback of curvature of \(h\) is given to \(X\) is given by \(d\xi\) (up to some constant factor). For any \(m\in\Z\)  put 
\[H^{0}_{b,m}(X):=\{f\in H^{0}_b(X)\cap\mathscr{C}^\infty(X)\mid \mathcal{T}f=imf\}.\] 
Choosing the Reeb volume form \(\xi\wedge(d\xi)^n\) on \(X\) and 
\(P=-i\mathcal{T}\) we find that the corresponding  Toeplitz operator \(T_P\) 
satisfies the assumptions in Theorem~\ref{thm:ConvergenceStrongCFIntro} with 
\(\sigma_P(\xi)=1\), that is, \(\alpha_P=\xi\). Since \(\mathcal{T}\) is a CR vector field we obtain 
\(T_P=-i\mathcal{T}\) on \(H^0_{b}(X)\cap \mathscr{C}^\infty(X)\). 
Then,  from Fourier decomposition, we find that the set of positive eigenvalues of
\(T_P\) is given by \(\{m\in\N\mid H^{0}_{b,m}(X)\neq \{0\}\}\) with corresponding eigenspaces given by \(H^{0}_{b,m}(X)\). 
Choosing \(\chi\in \mathscr{C}^\infty_c((0,1))\), 
\(\chi\not\equiv 0\), Theorem~\ref{thm:ConvergenceStrongCFIntro} leads
to a statement on the zero sets for random CR functions \(f_k\)  of the form
\[f_k=\sum_{m=1}^k\chi_k(m)u_m\]
when \(k\) becomes large. Here, \(u_m\in H^{0}_{b,m}(X)\)
are standard Gaussian random functions in \(H^{0}_{b,m}(X)\) 
where the Gaussian measure is induced by the \(L^2\)-inner product for functions.
	
Let \(H^0(M,L^m)\) denote the space of holomorphic section
on \(M\) with values in the \(m\)-th tensor power of \(L\). 
Then \(H^{0}_{b,m}(X)\) can be identified with  \(H^0(M,L^m)\) 
for any \(m\in\Z\). The equidistribution results 
\cite{CM15, DMS12, DS06, DLM23, NV98, BCM20, SZ99} 
for the zero sets 
of holomorphic sections in \(H^0(M,L^m)\), 
see also \cite{BCHM18} for a survey, lead to equidistribution results for 
CR functions in \(H^{0}_{b,m}(X)\) when \(m\) becomes large. 
In particular, it turns out that the zero sets of functions in 
\(H^{0}_{b,m}(X)\) distribute in accordance with 
\(d\xi=d\alpha_P\) up to some constant factor when \(m\) goes to \(+\infty\).
Hence, there is a relation between 
Theorem~\ref{thm:ConvergenceStrongCFIntro} and 
classical results on the asymptotic distribution for zero sets 
of holomorphic sections in this specific set-up.
	
	Replacing \(M\) by a compact complex cyclic orbifold in the set-up above,  equidistribution results for CR functions in the direct sum \(\bigoplus_{j=1}^dH^{0}_{b,m_j}(X)\) for some fixed \(d\) depending on the orbifold structure, when \(m_j\), \(1\leq j\leq d\), become large in a way related to the orbifold singularities of \(M\), were obtained by Hsiao-Shao~\cite{Hsiao_Shao_19}. Also here, the zero sets of those functions distribute asymptotically in accordance with \(d\xi=d\alpha_P\) up to some factor.        
\end{example}	

Standard examples for CR manifolds are boundaries of smoothly bounded domains in  complex manifolds. Therefore, Theorem~\ref{thm:ConvergenceStrongCFIntro} should be considered in the context of domains with boundary.  
   
Let \(Y\) be an \((n+1)\)-dimensional complex Hermitian manifold, \(n\geq1\). Consider a relatively compact strictly pseudoconvex domain \(\domain\subset\subset Y\) with smooth boundary \(X:=b \domain\). Putting \(T^{1,0}X:=\C TX\cap T^{1,0}Y\), where \(T^{1,0}Y\) denotes the complex structure on \(Y\), we find that \((X,T^{1,0}X)\) is a compact strictly pseudoconvex CR manifold which is CR embeddable into the complex euclidean space (see~\cite{HsM17}). 
 It follows that $X$ satisfies the assumptions of Theorem~\ref{thm:ProjectiveEmbeddingIntro} (see Remark~\ref{rmk:ClosedRangeEmbedabble}).
From results due to Grauert~\cite{Gr58} and Kohn--Rossi~\cite{KR65}, we have that \(X\) is connected and that any smooth CR function \(f\) on \(X\) extends holomorphically to \(\domain\). In particular, there exists a uniquely determined function \(\mathcal{F}\in\mathcal{O}(\domain)\cap \mathscr{C}^\infty(\overline{\domain})\) such that \(f=\mathcal{F}\) on \(X\). Hence we have a linear mapping \(L\colon H_b^0(X)\cap \mathscr{C}^\infty(X)\to \mathcal{O}(\domain)\cap \mathscr{C}^\infty(\overline{\domain})\) given by \(L(f)=\mathcal{F}\). Assume that \(P,T_P,\xi\) are given as in  Theorem~\ref{thm:ProjectiveEmbeddingIntro}. 
Let  \(0<\lambda_1\leq \lambda_2\leq \ldots\) denote the the positive eigenvalues of \(T_P\) and denote by \(\{f_j\}_{j\in\N}\subset H_b^0(X)\cap \mathscr{C}^\infty(X)\) a respective orthonormal system of eigenfunctions.  
Given a function \(\eta\in \mathscr{C}^\infty_c\left(\R_+,[0,\infty)\right)\) 
we have that the integral kernel of \(\eta_k(T_P)\) can be written as  \(\eta_k(T_P)(x,y)=\sum_{j=1}^\infty\eta_k(\lambda_j)f_j(x)\overline{f_j(y)}\) with \(\eta_k(t)=\eta(k^{-1}t)\) for \(t\in\R_+\), \(k>0\). We define
\begin{eqnarray}
	B^\eta_k\colon\overline{\domain}\to \R,\,\,\,\,\,B^\eta_k(z)=\sum_{j=1}^\infty\eta_k(\lambda_j)L(f_j)(z)\overline{L(f_j)(z)}.
\end{eqnarray} 
Then \(B^\eta_k\) is smooth and for any \(c>0\) we have that \(k^{-1}\partial \overline{\partial}\log(c+B^\eta_k)\) defines a non negative \((1,1)\)-form on \(\domain\) which is smooth up to the boundary.  Denote by \(\nnbforms{n}{\domain}:=\mathscr{C}^\infty(\overline{\domain},\Lambda^{n,n}\C T^*Y)\) the space of \((n,n)\)-forms on \(D\) which are smooth up to the boundary.  We have the following result for \(B^\eta_k\).
\begin{theorem}\label{thm:EquidistributionDomainsIntro}
Let \(Y\) be an \((n+1)\)-dimensional connected complex Hermitian manifold, \(n\geq1\). Consider a relatively compact strictly pseudoconvex domain \(\domain\subset\subset Y\) with smooth boundary \(X:=b \domain\).  Let \(\xi\) be a contact form  for \((X,T^{1,0}X)\) where \(T^{1,0}X=\C TX\cap T^{1,0}Y\) such that its induced Levi form is positive definite. Consider a Toeplitz operator \(T_P=\Pi P \Pi\) on \(X\) as in the set-up of Theorem~\ref{thm:ProjectiveEmbeddingIntro}. Choose  \(\eta\in \mathscr{C}_c^\infty(\R_+,[0,\infty))\), \(\eta\not \equiv 0\). With the notations above, for all \(\psi\in \Omega^{n,n}(\overline{D})\), we have
	\begin{equation}
		\lim_{k\to\infty}\int_\domain k^{-1}i\partial \overline{\partial}\log(1+B^\eta_k)\wedge \psi =\operatorname{mv}(\eta)\int_{b\domain}\alpha_P\wedge \iota^*\psi,
	\end{equation}
	Here, \(\iota\colon b\domain \to \overline{\domain}\) denotes the inclusion map and \(\alpha_P\), \(\operatorname{mv}(\eta)\) are given by~\eqref{eq:CharacteristikContactFormIntro} and~\eqref{eq:mvANDvarIntro}. More precisely, there exists a constant \(C>0\) such that
	\[\left| \int_\domain k^{-1}i\partial \overline{\partial}\log(1+B^\eta_k)\wedge \psi - \operatorname{mv}(\eta)\int_{b\domain}\alpha_P\wedge \iota^*\psi\right|\leq Ck^{-1}(\log(k)+1)\|\psi\|_{\mathscr{C}^2(\overline{\domain},\Lambda^{n,n}\C T^*Y)} \]
	for all \(k\geq 1\) and  \(\psi\in \Omega^{n,n}(\overline{D})\).
\end{theorem}
Theorem~\ref{thm:EquidistributionDomainsIntro} is interesting in the context of zero sets for random holomorphic functions as follows. Let \(\chi\in \mathscr{C}^\infty_c(\R_+)\), \(\chi\not \equiv 0\), be a smooth function with \(\text{supp}(\chi)\subset (\delta_1,\delta_2)\) for some \(0<\delta_1<\delta_2\). Put \(N_k=\#\{j\in\N\colon \lambda_j\leq \delta_2k\}\). 
Given i.i.d.~standard complex Gaussian  random variables \(a_0,a_1,\ldots,a_{N_k}\), we consider random holomorphic functions on \(\domain\) of the form
\begin{equation}\label{Eq:randomholomsecIntro}
	u_k=a_0+\sum_{j=1}^{N_k}a_j\chi(k^{-1}\lambda_j)L(f_j).
\end{equation}
We are especially interested in the behavior of the zeros of \(u_k\) when \(k\) becomes large. We say that \(u_k\) is regular if for any \(p\in b\domain\) with \(u_k(p)=0\) we have \((\partial u_k)_p\neq 0\) (cf.\ Definition~\ref{def:RegularWithRespectToBoundary}). For regular \(u_k\) we have that the zero divisor is given by \(\divisor{u_k}=(\pi)^{-1}i\partial\overline{\partial}\log|u_k|\)  in distributions that is
\begin{eqnarray*}
	\left(\divisor{u_k}, \psi\right)
	&=&\frac{i}{\pi}\left(-\int_{b\domain}\iota^*(\frac{1}{2u_k}\partial u_k\wedge \psi)-\int_{b\domain}\iota^*(\log(|u_k|)\wedge\overline{\partial}\psi)+\int_\domain\log(|u_k|)\wedge\partial \overline{\partial} \psi\right)
\end{eqnarray*}
for all \(\psi\in \nnbforms{n}{\domain}\). Given \(\psi\) with compact support in \(\domain\)
this formula becomes the well known Lelong-Poincar\'e formula (see~\cite{Le57}). 
For the general case we refer to Theorem~\ref{thm:PoincareLelongFormula}.  
In order to state our result we need to introduce some probability spaces first. 
For \(k>0\) put \(A_k=\text{span}\left(\{1\}\cup\{L(f_j)\mid \lambda_j\leq 
\delta_2k\}\right)\subset\mathcal{O}(D)\cap\cC^\infty(\overline{D})\). 
On \(A_k\) we consider the probability measure \(\mu_k\) induced 
by the standard complex Gaussian measure \(\mu^G\) on \(\C^{N_k+1}\) and the map
\begin{eqnarray}\label{eq:defAkMukComplexIntro}
	\C^{N_k+1} \ni a \mapsto a_0+\sum_{j=1}^{N_k}a_j\chi_k(\lambda_j)L(f_j) \in A_k.
\end{eqnarray} 
We use the same notation \((A_k,\mu_k)\) for the probability spaces as before
since the construction just differs by the term \(a_0\) which is introduced to avoid
simultaneous vanishing of the \(u_k\)'s away from the boundary \(bD\). 
Similar as in the CR set-up above, we study \(\divisor{f}\)
as a distribution valued random variable on \(A_k\) especially when \(k\) becomes large. 
Given \(k>0\) we consider the expectation value  \(\mathbb{E}_k(\divisor{f})\) of \(\divisor{f}\), 
which is - in case of existence - defined by
\begin{eqnarray}\label{eq:ExpectationValueIntroDomain}
\left(\mathbb{E}_k\left(\divisor{f}\right),\psi\right):=
\int_{A_k} \left(\divisor{f},\psi\right)d\mu_k(f),\,\,\,\,\psi\in \nnbforms{n}{\domain}.
\end{eqnarray}
Recall that \(B_{k}^{|\chi|^2}(z)=\sum_{j=1}^\infty |\chi_k(\lambda_j)|^2|L(f_j)(z)|^2\), \(z\in \overline{\domain}\).
\begin{theorem}\label{thm:ExpactationRndZerosIntro}
Under the same assumptions as in  Theorem~\ref{thm:EquidistributionDomainsIntro} 
and with the notations above we have
for each \(k>0\)  that \(\divisor{f}\) is a well-defined current for all \(f\) outside a  subset in
\(A_k\) of \(\mu_k\)-measure zero and the expectation value \(\mathbb{E}_k\left(\divisor{f}\right)\)
defined by~\eqref{eq:ExpectationValueIntroDomain} exists with
\[\mathbb{E}_k\left(\divisor{f}\right)=\frac{i}{2\pi}\partial\overline{\partial}\log(1+B^{|\chi|^2}_{k}).\]
Here, \(\chi\in \mathscr{C}^\infty_c(\R_+)\), \(\chi\not\equiv 0\), 
and the respective probability space \((A_k,\mu_k)\) is given by~\eqref{eq:defAkMukComplexIntro}.
\end{theorem}
As a direct consequence of Theorem~\ref{thm:EquidistributionDomainsIntro} and
Theorem~\ref{thm:ExpactationRndZerosIntro} we obtain
\[\lim_{k\to\infty}\left( k^{-1}\mathbb{E}_k\left(\divisor{f}\right),\psi\right)= 
\frac{\operatorname{mv}(|\chi|^2)}{2\pi}\int_{b\domain}\alpha_P\wedge \iota^*\psi \] 
for all \(\psi\in \nnbforms{n}{\domain}\). This shows that in average the most of the zeros 
of the random holomorphic functions \(u_k\) in~\eqref{Eq:randomholomsecIntro} concentrate near 
the boundary when \(k\) becomes large. 
See also \cite[Example 2.10]{DLM23} for the concentration of zeros
of random holomorphic $(n,0)$-forms near the boundary.

This phenomenon even holds in a stronger sense as we will explain now. 
Consider \(k\in\N\) and put \(A_\infty=\Pi_{k=1}^\infty A_k\), 
\(d\mu_{\infty}=\Pi_{k=1}^\infty d\mu_k\). We have the following.
\begin{theorem}\label{thm:SequenceRndZerosIntro}
Under the same assumptions as in 
Theorem~\ref{thm:EquidistributionDomainsIntro} 
and with the notations above we have for \(\mu_\infty\)-almost every 
\(u=(u_k)_{k\in \N}\in A_{\infty}\) that \((\divisor{u_k}, \psi)\)
is well defined for all \(\psi\in \nnbforms{n}{\domain}\) and \(k\in\N\). 
Furthermore, given  \(\psi\in \nnbforms{n}{\domain}\), we have for \(\mu_\infty\)-almost every \(u=(u_k)_{k\in \N}\in A_{\infty}\) that
	\begin{eqnarray}\label{eq:SequenceRndZerosIntro}
	\lim_{k\to\infty} \left(k^{-1}\divisor{u_k},\psi\right) = \frac{\operatorname{mv}(|\chi|^2)}{2\pi}\int_{b\domain}\alpha_P\wedge \iota^*\psi.
	\end{eqnarray}
\end{theorem}
Roughly speaking, it follows from Theorem~\ref{thm:SequenceRndZerosIntro} 
that the zero set of a random holomorphic function \(u_k\) in~\eqref{Eq:randomholomsecIntro} 
accumulate at the boundary as $k\to+\infty$. Since \(\int_{b\domain} \alpha_P \wedge \iota^*\psi=0\) for all \(\psi\in \nnbforms{n}{\domain}\) with 
\(\alpha_P\wedge \iota^*\psi=0\) we conclude in addition that the zero set of \(u_k\) 
 converges tangentially to the boundary when \(k\) becomes large. However, it turns out (see Lemma~\ref{lem:UnregularFunctionsAreZeroSetCR}) that the intersection of the zero set of \(u_k\) and the boundary is transversal almost surely for \(k>0\). 
For other results on equidistribution for strictly pseudoconvex domains 
see~\cite{Hsiao_Shao_19}. We note that a similar statement
as in Remark~\ref{rmk:ConvergenceStrongCFIntro} also holds with respect to 
Theorem~\ref{thm:SequenceRndZerosIntro}. 
See Theorem~\ref{thm:SequenceRndZeros} for more details.
\begin{example}
Let \(D=\{z\in \C^{n+1}\mid |z|<1\}\) be the unit ball in \(\C^{n+1}\), \(n\geq 1\). 
Then \(X:=bD\) is the unit sphere and hence a compact strictly pseudoconvex CR manifold. 
It follows that \(D\) satisfies the assumptions in Theorem~\ref{thm:SequenceRndZerosIntro}. 
We find that \(\xi:=\iota^*\omega_0\) defines a contact form on \(X\) where
	\[\omega_0=\frac{1}{2i}\sum_{j=1}^{n+1}\overline{z}_jdz_j-z_jd\overline{z}_j,\,\,\,z\in \C^{n+1}\] 
	and \(\iota\colon X\to \C^{n+1}\) denotes the inclusion map. Then the vector field
	\[z\mapsto i\sum_{j=1}^{n+1}z_j\frac{\partial}{\partial z_j}-
	\overline{z}_j\frac{\partial}{\partial \overline{z}_j},\,\,\,z\in \C^{n+1}\]
	restricts to the respective Reeb vector field  on \(X\) denoted by \(\mathcal{T}\). 
	Putting \(P:=-i\mathcal{T}\) and choosing the Reeb volume form \(\xi\wedge (d\xi)^{n}\) 
	on \(X\) we find that \(T_P=\Pi P\Pi\) satisfies the assumptions in Theorem~\ref{thm:ProjectiveEmbeddingIntro} 
	with \(\sigma_P(\xi)=1\), that is, \(\alpha_P=\xi\). It turns out that the set of eigenvalues  \(T_P\) is given by \(\N_0\). More precisely, for any \(d\in\N_0\) we have that the space of CR eigenfunctions of \(T_P\) for the eigenvalue \(d\) is given by restrictions of homogeneous polynomials of degree \(d\) to \(X\). Choosing a cut-off function \(\chi\in \mathscr{C}^\infty_{c}((0,1))\), \(\chi\not\equiv 0\), Theorem~\ref{thm:SequenceRndZerosIntro} leads to a result on the the asymptotic distribution on \(\overline{D}\) for zero sets of weighted sums (weighted with respect to \(\chi\)) of (orthogonal) homogeneous Gaussian random polynomials with  degree less or equal than \(d\) when \(d\) becomes large. 
	
	The asymptotic distribution on \(\C^{n+1}\), \(n\geq 1\), of zero sets for orthogonal Gaussian random  polynomials with  degree less or equal than \(d\) when \(d\) becomes large was studied by Bloom-Shiffman~\cite{BS07}.   From their results (see also \cite{BBL22}) it follows that those zero sets distribute asymptotically in accordance to some equilibrium current \(s_{\text{eq}}\) under fairly general assumptions. With respect to the set-up considered here it is well-known (see~\cite{BS07}) that \(s_{\text{eq}}\) is given by \(\frac{i}{\pi}\partial\overline{\partial}V\) with \(V(z)=\max\{\log|z|,0\}\), \(z\in\C^{n+1}\), in the sense of distributions.  
In context of Theorem~\ref{thm:SequenceRndZerosIntro}
this situation can be expressed by choosing the cut-off function \(\chi\)
to be the characteristic function of the interval \([0,1]\subset \R\). 
Theorem~\ref{thm:SequenceRndZerosIntro} cannot be applied in this situation since \(\chi\) is not smooth. 
However, a direct calculation shows that 
\(\operatorname{mv}(|\chi|^2)=\frac{n+1}{n+2}\) for this choice of \(\chi\).
Then, considering test forms \(\psi\in \Omega_c^{n,n}(D)\) we have \((s_{\text{eq}},\psi)=0\) and
also the right-hand side of~\eqref{eq:SequenceRndZerosIntro} in 
Theorem~\ref{thm:SequenceRndZerosIntro} would vanish.      
\end{example}
For further examples where Theorem~\ref{thm:ExpactationRndZerosIntro} or
Theorem~\ref{thm:SequenceRndZerosIntro} apply, see Examples~\ref{ex:DiscBundle} and
\ref{ex:GrauertTube}. 	
We refer the reader to~\cite{HM23} for other results concerning
the semi-classical spectral asymptotics of Toeplitz operators on strictly pseudoconvex domains.

The proofs of the results above are obtained by studying an object \(\mathcal{C}_f\) 
which we will explain now. Let \((X,T^{1,0}X)\) be a compact strictly pseudoconvex CR manifold
satisfying the assumptions in Theorem~\ref{thm:ProjectiveEmbeddingIntro} and
consider a smooth CR function \(f\in H_b^0(X)\cap \mathscr{C}^\infty(X)\) such that \(df_p\neq 0\)
holds for all \(p\in X\) with \(f(p)=0\). The assumptions on \(f\) and \(X\) entail
that \(df/f\) is well defined in the sense of currents. 
More precisely, we have that \(df/f\) is integrable on \(X\) and hence the linear map
\begin{eqnarray}\label{eq:DefCfIntro}
\mathcal{C}_f\colon \Omega^{2n}(X)\to\C,\,\,\,\,\mathcal{C}_f(\psi):=\frac{1}{2\pi i}\int_X\frac{df}{f}\wedge \psi
\end{eqnarray}
is well defined (see Theorem~\ref{thm:OneOverFIntegrableCR}). In this work the study of  \(\mathcal{C}_f\)
is important for two reasons. On the one hand it is directly linked to the description of zero sets for 
CR functions via \(d\mathcal{C}_f=\divisor{f}\) in distributions sense for any smooth CR function \(f\)
such that zero is a regular value of the map \(f\colon X\to \C\)
(see Theorem~\ref{thm:PoincareLelongForCR}). 
On the other hand \(\mathcal{C}_f\) appears  as a boundary term in the Lelong-Poincar\'e formula
for domains with boundary (see Definition~\ref{def:Definitionddbarlogf} and 
Theorem~\ref{thm:PoincareLelongFormula}).

This work is organized as follows. In Section~\ref{sec:ProjectiveEmbedding} 
we prove the projective embedding result for CR manifolds as stated in 
Theorem~\ref{thm:ProjectiveEmbeddingIntro}. Section~\ref{subsec:PoincareLelong} 
contains a proof of the Lelong-Poincar\'e formula for domains with boundary. 
The tools developed in Section~\ref{sec:ProjectiveEmbedding} and 
Section~\ref{subsec:PoincareLelong} will be used to study \(\mathcal{C}_f\) given 
by~\eqref{eq:DefCfIntro} seen as a distribution valued random variable in 
Section~\ref{sec:EquidistributionCR}. The study of \(\mathcal{C}_f\) will then 
lead to a proof of the equidistribution results for CR manifolds as stated in 
Theorem~\ref{thm:ExpectationValueCRDistributionIntro} and Theorem~\ref{thm:ConvergenceStrongCFIntro}.  
After some remarks on complex manifolds with boundary in Section~\ref{sec:RemarksMfdBoundary}, we prove 
Theorem~\ref{thm:EquidistributionDomainsIntro}, \ref{thm:ExpactationRndZerosIntro} and \ref{thm:SequenceRndZerosIntro}.

\addtocontents{toc}{\protect\setcounter{tocdepth}{0}}
\section*{Acknowledgment}
Hendrik Herrmann and George Marinescu are  partially supported by the ANR-DFG project QuaSiDy (Project ID 490843120). Chin-Yu Hsiao is partially supported by Taiwan Ministry of Science and Technology projects  108-2115-M-001-012-MY5, 109-2923-M-001-010-MY4. George Marinescu is partially supported by the 
DFG Priority Program 2265 `Random Geometric Systems' (Project-ID 422743078). Wei-Chuan Shen is supported by the SFB/TRR 191 "Symplectic Structures in Geometry, Algebra and Dynamics", funded by the DFG (Projektnummer 281071066 – TRR 191).  The first named author would like to thank Turgay Bayraktar, Tobias Harz and Nikolay Shcherbina for fruitful discussions on equilibrium measures and CR extensions.
\addtocontents{toc}{\protect\setcounter{tocdepth}{2}}

\section{Preliminaries}
\subsection{Notations}
We use the following notations throughout this article. $\mathbb{Z}$ is the set of integers, $\mathbb N=\{1,2,3,\ldots\}$ is the set of natural numbers, $\mathbb N_0=\mathbb N\bigcup\{0\}$ and $\mathbb R$ is the set of 
real numbers. Also, $\overline{\mathbb R}_+=\{x\in\mathbb R:x\geq0\}$, $\R^*:=\R\setminus\{0\}$, $\R_+:=\overline{\mathbb R}_+\setminus\{0\}$.
We write $\alpha=(\alpha_1,\ldots,\alpha_n)\in\mathbb N^n_0$ 
if $\alpha_j\in\mathbb N_0$, 
$j=1,\ldots,n$. For $x=(x_1,\ldots,x_n)\in\mathbb R^n$, 
we write
\[
\begin{split}
	&x^\alpha=x_1^{\alpha_1}\ldots x^{\alpha_n}_n,\\
	& \partial_{x_j}=\frac{\partial}{\partial x_j}\,,\quad
	\partial^\alpha_x=\partial^{\alpha_1}_{x_1}\ldots\partial^{\alpha_n}_{x_n}
	=\frac{\partial^{|\alpha|}}{\partial x^\alpha}\,.
\end{split}
\]
Let $z=(z_1,\ldots,z_n)$, $z_j=x_{2j-1}+ix_{2j}$, $j=1,\ldots,n$, 
be coordinates of $\mathbb C^n$. We write
\[
\begin{split}
	&z^\alpha=z_1^{\alpha_1}\ldots z^{\alpha_n}_n\,,
	\quad\ol z^\alpha=\ol z_1^{\alpha_1}\ldots\ol z^{\alpha_n}_n\,,\\
	&\partial_{z_j}=\frac{\partial}{\partial z_j}=
	\frac{1}{2}\Big(\frac{\partial}{\partial x_{2j-1}}-i\frac{\partial}{\partial x_{2j}}\Big)\,,
	\quad\partial_{\ol z_j}=\frac{\partial}{\partial\ol z_j}
	=\frac{1}{2}\Big(\frac{\partial}{\partial x_{2j-1}}+i\frac{\partial}{\partial x_{2j}}\Big),\\
	&\partial^\alpha_z=\partial^{\alpha_1}_{z_1}\ldots\partial^{\alpha_n}_{z_n}
	=\frac{\partial^{|\alpha|}}{\partial z^\alpha}\,,\quad
	\partial^\alpha_{\ol z}=\partial^{\alpha_1}_{\ol z_1}\ldots\partial^{\alpha_n}_{\ol z_n}
	=\frac{\partial^{|\alpha|}}{\partial\ol z^\alpha}\,.
\end{split}
\]
For $j, s\in\mathbb Z$, set $\delta_{js}=1$ if $j=s$, 
$\delta_{js}=0$ if $j\neq s$.

Let $U$ be an open set in $\mathbb{R}^{n_1}$ and let $V$ be an open set in $\mathbb{R}^{n_2}$. Let $\cCc(V)$ and $\mathscr{C}^\infty(U)$ be the space of smooth functions with compact support in $V$ and the space of smooth functions on $U$, respectively; $\mathscr{D}'(U)$ and $\mathscr{E}'(V)$ be the space of distributions on $U$ and the space of distributions with compact support in $V$, respectively.

Let $F:\cC^\infty_c(V)\to\mathscr{D}'(U)$ be a continuous operator and let $F(x,y)\in\mathscr{D}'(U\times V)$ be the distribution kernel of $F$. In this work, we will identify $F$ with $F(x,y)$.
We say that
$F$ is a smoothing operator 
if $F(x,y)\in\cC^\infty(U\times V)$. 
For two continuous linear operators $A,B:\cC^\infty_c(V)\to\mathscr{D}'(U)$, we write $A\equiv B$ (on $U\times V$) or $A(x,y)\equiv B(x,y)$ (on $U\times V$) if $A-B$ is a smoothing operator, where $A(x,y),B(x,y)\in\mathscr{D}'(U\times V)$ are the distribution kernels of $A$ and $B$, respectively.

If $U=V$, for a smooth function $f(x,y)\in\mathscr{C}^\infty(U\times U)$, we write $f(x,y)=O\left(|x-y|^\infty\right)$ if for all multi-indices  $\alpha,\beta\in\N_0^{n_1}$, $\left(\partial_x^\alpha\partial_y^\beta f\right)(x,x)=0$, for all $x\in U$. 

Let us introduce some notion in microlocal analysis used in this paper. Let $D\subset\R^{2n+1}$ be an open set. For any $m\in\mathbb R$, $S^m_{1,0}(D\times D\times\mathbb{R}_+)$ 
	is the space of all $s(x,y,t)\in\cC^\infty(D\times D\times\mathbb{R}_+)$ 
	such that for all compact sets $K\Subset D\times D$, all $\alpha, 
	\beta\in\mathbb N^{2n+1}_0$ and $\gamma\in\mathbb N_0$, 
	there is a constant $C_{K,\alpha,\beta,\gamma}>0$ satisfying the estimate
	\begin{equation}
		\left|\partial^\alpha_x\partial^\beta_y\partial^\gamma_t a(x,y,t)\right|\leq 
		C_{K,\alpha,\beta,\gamma}(1+|t|)^{m-|\gamma|},\ \ 
		\mbox{for all $(x,y,t)\in K\times\mathbb R_+$, $|t|\geq1$}.
	\end{equation}
	We put 
	\[
	S^{-\infty}(D\times D\times\mathbb{R}_+)
	:=\bigcap_{m\in\mathbb R}S^m_{1,0}(D\times D\times\mathbb{R}_+).\]
Let $s_j\in S^{m_j}_{1,0}(D\times D\times\mathbb{R}_+)$, 
$j=0,1,2,\cdots$ with $m_j\rightarrow-\infty$ as $j\rightarrow+\infty$. 
By the argument of Borel construction, there always exists $s\in S^{m_0}_{1,0}(D\times D\times\mathbb{R}_+)$ 
unique modulo $S^{-\infty}$ such that 
\begin{equation}
	s-\sum^{\ell-1}_{j=0}s_j\in S^{m_\ell}_{1,0}(D\times D\times\mathbb{R}_+)
\end{equation}
for all $\ell=1,2,\cdots$. If $s$ and $s_j$ have the properties above, we write
\begin{equation}
	s(x,y,t)\sim\sum^{+\infty}_{j=0}s_j(x,y,t)~\text{in}~ 
	S^{m_0}_{1,0}(D\times D\times\mathbb{R}_+).
\end{equation}
Also, we use the notation 
\begin{equation}  
	s(x, y, t)\in S^{m}_{{\rm cl\,}}(D\times D\times\mathbb{R}_+)
\end{equation}
if $s(x, y, t)\in S^{m}_{1,0}(D\times D\times\mathbb{R}_+)$ and we can find $s_j(x, y)\in\cC^\infty(D\times D)$, $j\in\N_0$, such that 
\begin{equation}
	s(x, y, t)\sim\sum^{+\infty}_{j=0}s_j(x, y)t^{m-j}\text{ in }S^{m}_{1, 0}
	(D\times D\times\mathbb{R}_+).
\end{equation}

Let $M, M_1$ be smooth paracompact manifolds. We can define $S^m_{1,0}(M\times M_1\times\mathbb R_+)$, 
$S^m_{{\rm cl\,}}(M\times M_1\times\mathbb R_+)$ and asymptotic sums in the standard way. Given a smooth vector bundle \(E\to M\) over \(E\) we denote the space of smooth sections on \(M\) with values in \(E\) by \(\mathscr{C}^\infty(M,E)\). The space of those sections with compact support is denoted by \(\mathscr{C}_c^\infty(M,E)\). For any \(q\in\N_0\) we denote the space of smooth \(q\)-forms on \(M\) by \(\Omega^q(M)\) and denote space of those forms with compact support by \(\Omega_c^q(M)\). Recall the definition of the exterior differential \(d\colon \Omega^q(M)\to \Omega^{q+1}(M)\).

For $m\in\mathbb R$, let $L^m_{{\rm cl\,}}(D)$ denote the space of classical pseudodifferential operators of order $m$ on $D$. For $P\in L^m_{{\rm cl\,}}(D)$, let 
$\sigma_P$ denote the principal symbol of $P$.

We recall some notations of semi-classical analysis.
Let $U$ be an open set in $\mathbb{R}^{n_1}$ and let $V$ 
be an open set in $\mathbb{R}^{n_2}$. 

\begin{definition}\label{D:knegl}
	We say a $k$-dependent continuous linear operator
	$F_k:\cCc(V)\to\mathscr{D}'(U)$ is $k$-negligible 
	if for all $k$ large enough, $F_k$ is a smoothing operator,
	and for any compact set $K$ in $U\times V$, for all multi-index
	$\alpha\in\N_0^{n_1}$, $\beta\in\N_0^{n_2}$ and $N\in\mathbb{N}_0$, 
	there exists a constant $C_{K,\alpha,\beta,N}>0$ such that
	\begin{equation*} 
		\left|\partial^\alpha_x\partial^\beta_y F_k(x,y)\right|\leq 
		C_{K,\alpha,\beta,N} k^{-N},
	\end{equation*}
	for all $x,y\in K$ and $k$ large enough.
	We write $F_k(x,y)=O(k^{-\infty})$ (on $U\times V$) or $F_k=O(k^{-\infty})$ (on $U\times V$) if $F_k$ is $k$-negligible. 
\end{definition}

Next, we recall semi-classical symbol spaces. Let $W$ be an open set in $\mathbb{R}^{N}$, we define the space
\begin{equation}
	S(1;W):=\{a\in\mathscr{C}^\infty(W):\sup_{x\in W}|\partial^\alpha_x a(x)|<\infty~\text{for all}~\alpha\in\mathbb{N}_0^N\}.
\end{equation}
Consider the space $S^0_{\operatorname{loc}}(1;W)$ containing all smooth functions $a(x,k)$ with
real parameter $k$ such that for all multi-index $\alpha\in\mathbb{N}_0^N$, any
cut-off function $\chi\in\cCc(W)$, we have
\begin{equation}
	\sup_{\substack{k\in\mathbb{R}\\k\geq 1}}\sup_{x\in W}|\partial^\alpha_x(\chi(x)a(x,k))|<\infty.
\end{equation}
For general $m\in\mathbb{R}$, we can also consider
\begin{equation}
	S^m_{\operatorname{loc}}(1;W):=\{a(x,k):k^{-m}a(x,k)\in\ S^0_{\operatorname{loc}}(1;W)\}.  
\end{equation}
In other words, $S^m_{\operatorname{loc}}(1;W)$ takes all the smooth function $a(x,k)$ with parameter $k\in\mathbb{R}$ satisfying the following estimate. For any compact set $K\Subset W$, any multi-index $\alpha\in\mathbb N^n_0$, there is a constant $C_{K,\alpha}>0$ independent of $k$ such that 
\begin{equation}
	|\partial^\alpha_x (a(x,k))|\leq C_{K,\alpha} k^m,~\text{for all}~x\in K,~k\geq1. 
\end{equation}
For a sequence of $a_j\in S^{m_j}_{\operatorname{loc}}(1;W)$ with $m_j$ decreasing, $m_j\to-\infty$, and $a\in S^{m_0}_{\operatorname{loc}}(1;W)$. We say 
\begin{equation}
	a(x,k)\sim\sum_{j=0}^\infty a_j(x,k)~\text{in}~S^{m_0}_{\operatorname{loc}}(1;W)
\end{equation}
if for all $l\in\mathbb{N}$, we have
\begin{equation}
	a-\sum_{j=0}^{l-1} a_j\in S^{m_l}_{\operatorname{loc}}(1;W).
\end{equation}
In fact, for all sequence $a_j$ above, there always exists an element $a$ as the asymptotic sum, which is unique up to the elements in $S^{-\infty}_{\operatorname{loc}}(1;W):=\cap_{m} S^m_{\operatorname{loc}}(1;W)$.

The detail of discussion here about microlocal analysis or semi-classical analysis can be found in \cites{DS99, GS94} for example. All the notations introduced above can be generalized to the case on paracompact manifolds.  

Given  a complex manifold \(Y\) we denote its complex structure by \(T^{1,0}Y\). For any \(p,q\in\N_0\) we denote the space of smooth \((p,q)\)-forms by \(\Omega^{p,q}(Y):=\mathscr{C}^\infty(Y,\Lambda^{p,q}\C T^*Y)\). We further denote the space of those forms which af compact support by \(\Omega^{p,q}_c(Y)\).  Here, \(\Lambda^{p,q}\C T^*Y\) is the bundle of alternating multilinear forms on \(\C TY\) of type \((p,q)\). Recall that on \(Y\) we have the decomposition for the exterior differential \(d=\partial+\overline{\partial}\) with
\[\partial\colon \Omega^{p,q}(Y)\to \Omega^{p+1,q}(Y),\,\,\,\overline{\partial}\colon \Omega^{p,q}(Y)\to \Omega^{p,q+1}(Y).\]  

\subsection{Microlocal analysis of the Szeg\H{o} kernel}
\label{sec:CRmanifoldsMicroLocal}
Let $X$ be a smooth ($\mathscr{C}^\infty$) orientable manifold of real dimension $2n+1,~n\geq 1$. We say $X$ is a (codimension one) Cauchy--Riemann (CR for short) manifold with CR structure \(T^{1,0}X\) if 
$T^{1,0}X\subset\mathbb{C}TX$ is a subbundle such that
\begin{enumerate}
	\item (codimension one) $\dim_{\mathbb{C}}T^{1,0}_{p}X=n$ for any $p\in X$.
	\item (non-degenerate) $T^{1,0}_p X\cap T^{0,1}_p X=\{0\}$ for any $p\in X$, where $T^{0,1}_p X:=\overline{T^{1,0}_p X}$.
	\item (integrable) For $V_1, V_2\in \mathscr{C}^{\infty}(X,T^{1,0}X)$, then $[V_1,V_2]\in\mathscr{C}^{\infty}(X,T^{1,0}X)$, where 
	$[\cdot,\cdot]$ stands for the Lie bracket between vector fields. 
\end{enumerate}
Fix a smooth Hermitian metric $\langle\cdot|\cdot\rangle$ on the complexified tangent bundle $\C TX$ such that $T^{1,0}X$ is orthogonal to $T^{0,1}X$. Then locally there is a real vector field $\mathcal{T}$ with $|\mathcal{T}|^2:=\langle\mathcal{T}|\mathcal{T}\rangle=1$ which is pointwise orthogonal to $T^{1,0}X\oplus T^{0,1}X$. Notice that $\mathcal{T}$ is unique up to the choice of sign. Denote  $T^{*1,0}X$ and $T^{*0,1}X$ the dual bundles in  $\C T^*X$ annihilating $\C\mathcal{T}\bigoplus T^{0,1}X$ and $\C\mathcal{T}\bigoplus T^{1,0}X$, respectively. Define the bundles of $(0,q)$ forms by $T^{*0,q}X:=\wedge^qT^{*0,1}X$. The Hermitian metric on $\C TX$ induces, by duality, a Hermitian metric on $\mathbb CT^*X$, and also on $T^{*0,q}X$ for all $q\in\{0,\ldots,n\}$. We shall also denote all these induced metrics by $\langle\cdot|\cdot\rangle$. For any open set $D\subset X$, denote $\Omega^{0,q}(D)$ be the space of smooth sections of $T^{*0,q}X$ over $D$ and let $\Omega^{0,q}_c(D)$ be the subspace of $\Omega^{0,q}(D)$ whose elements have compact support in $D$.

Locally, there also exists an orthonormal frame 
$\{\omega_j\}_{j=1}^n$ of $T^{*1.0}X$. 
Notice that the real 
$(2n)$-form $\omega:=i^n\omega_1\wedge\overline{\omega}
_1\wedge\ldots\wedge\omega_n\wedge\overline{\omega}_n$ is 
independent of the choice of orthonormal frame, so it is globally defined. 
Locally, there exists a real $1$-form 
$\xi$ of $|\xi|^2:=\langle\xi|\xi\rangle=1$ 
which is orthogonal to $T^{*1,0}X\oplus T^{*0,1}X$. 
The $\xi$ is unique up to the choice of sign. 
Since $X$ is orientable, there is a nowhere vanishing $(2n+1)$-form 
$\Theta$ on $X$. Thus, $\xi$ can be specified uniquely by 
requiring that $\omega\wedge\xi=f\Theta$, 
where $f$ is a positive smooth function. 

\begin{definition}
	The so chosen $\xi$ is globally defined, and we call it the characteristic form on $X$.
\end{definition} 

The Levi distribution $HX$ of the CR manifold $X$ is the real part of 
$T^{1,0}X \oplus T^{0,1}X$,
i.e., the unique subbundle $HX$ of $TX$ such that 
\begin{align}\label{eq:2.5}
	\C HX=T^{1,0}X \oplus T^{0,1}X.
\end{align} 
Let $J:HX\to HX$ be the complex structure given by 
$J(u+\ol u)=iu-i\ol u$, for every $u\in T^{1,0}X$. 
If we extend $J$ complex linearly to $\C HX$ we have
$T^{1,0}X \, = \, \left\{ V \in \C HX \,:\, \, JV \, 
=  \,  iV  \right\}$.
Thus the CR structure $T^{1,0}X$ is determined by
the Levi distribution and $J$.
The annihilator $(HX)^{\perp}\subset T^*X$ is given by 
$(HX)^{\perp}=\R\xi$, so
we have
$\langle\,\xi(x),u\,\rangle=0$, 
for any $u\in H_xX,\: x\in X$. 
The restriction of $d\xi$
on $HX$ is a $(1,1)$-form. 

\begin{definition}\label{D:Leviform}
	The Levi form $\mathcal{L}_x=\mathcal{L}^{\xi}_x$ of $X$ at $x\in X$ 
	associated to $\xi$ is the symmetric bilinear map
	\begin{equation}\label{eq:2.12}
		\mathcal{L}_x:H_xX\times H_x X\to\R,\quad \mathcal{L}_x(u,v)
		=\frac{d\xi}{2}(u,Jv),\quad \text{ for } u, v\in H_xX.
	\end{equation}  
	It induces a Hermitian form
	\begin{equation}\label{eq:2.12b}
		\mathcal{L}_x:T^{1,0}_xX\times T^{1,0}_xX\to\C,
		\quad \mathcal{L}_x(U,V)=\frac{d\xi}{2i}(U, \ol V) ,
		\quad \text{ for } U, V\in T^{1,0}_xX.
	\end{equation}  
\end{definition}

\begin{definition}
	A CR manifold $X$ is said to be  strictly pseudoconvex if 
	there exists a characteristic $1$-form $\xi$
	such that for every $x\in X$ the Levi form
	$\mathcal{L}^{\xi}_x$ is positive definite. In this case, such $1$-form $\xi$ is also called a contact form because $\xi\wedge(d\xi)^n\neq 0$, where $\dim_\R X=2n+1$.
\end{definition} 

From now on, we assume that $(X,T^{1,0}X)$ is a compact strictly pseudoconvex CR manifold of dimension $2n+1$
and we fix a characteristic $1$-form $\xi$ on $X$ such that $\mathcal{L}^{\xi}_x$ is positive definite at every $x\in X$.

\begin{definition}\label{D:Reebfield}
	The Reeb vector field $\mathcal{T}=\mathcal{T}(\xi)$
	associated with the contact form $\xi$ is
	the vector field uniquely defined by the equations
	$i_{\mathcal{T}}\xi=1$, $i_{\mathcal{T}}d\xi=0$,
	cf.\ \cite{Gei08}*{Lemma/Definition 1.1.9}.
\end{definition} 

We take $\langle\cdot|\cdot\rangle$ so that $\langle\mathcal{T}|\mathcal{T}\rangle=1$, $\langle\xi|\xi\rangle=1$ on $X$.
With respect to the given Hermitian metric $\langle\cdot|\cdot\rangle$, we consider  the orthogonal projection
\begin{equation}
	\pi^{(0,q)}:\Lambda^q\mathbb{C}T^*X\to T^{*0,q}X.
\end{equation}
The tangential Cauchy--Riemann operator is defined to be 
\begin{equation}
	\label{tangential Cauchy Riemann operator}
	\overline{\partial}_b:=\pi^{(0,q+1)}\circ d:\Omega^{0,q}(X)\to\Omega^{0,q+1}(X).
\end{equation}
By Cartan's formula, we can check that 
\begin{equation}
	\overline{\partial}_b^2=0.  
\end{equation}
In this paper, we work with two volume forms on $X$.
\begin{itemize}
	\item[(\emph{i})] A given smooth positive $(2n+1)$-form $dV(x)$.
	\item[(\emph{ii})] The $(2n+1)$-form $dV_{\xi}:=\frac{2^{-n}}{n!}\xi\wedge\left(d\xi\right)^n$ given by the contact form $\xi$.
\end{itemize}
Take the $L^2$-inner product $(\cdot|\cdot)$ on $\Omega^{0,q}(X)$ induced by $dV$ and $\langle\cdot|\cdot\rangle$ via
\begin{equation}
	(f|g):=\int_X\langle f|g\rangle dV,~f,g\in\Omega^{0,q}(X).
\end{equation}
We denote by $L^2_{(0,q)}(X)$ the completion of $\Omega^{0,q}(X)$ with respect to $(\cdot|\cdot)$, and we write $L^2(X):=L^2_{(0,0)}(X)$. Furthermore, for \(q=0\) we will also use the notation \((\cdot,\cdot):=(\cdot|\cdot)\) to denote the \(L^2\)-inner product.
Let $\|\cdot\|$ denote the corresponding $L^2$-norm on $L^2_{(0,q)}(X)$. 
Let $\overline{\partial}_b^*$ be the formal adjoint of $\overline{\partial}_b$ with respect to $(\cdot|\cdot)$, and the Kohn Laplacian is defined by 
\begin{equation}
	\Box^{(0)}_{b}:=\overline{\partial}_b^*\overline{\partial}_b:\mathscr{C}^\infty(X)\to\mathscr{C}^\infty(X),
\end{equation}
which is not an elliptic operator because the principal symbol $\sigma_{\Box^{(0)}_{b}}(x,\eta)\in\mathscr{C}^\infty(T^*X)$ has a non-empty characteristic set
$\{(x,\eta)\in T^*X:\eta=\lambda\xi(x),~\lambda\in\mathbb{R}^*\}$. For later use, we denote the symplectic cone
\begin{equation}
	\Sigma:=\{(x,\eta)\in T^*X:\eta=\lambda\xi(x),~\lambda>0\}.
\end{equation}
In fact, $\Box^{(0)}_{b}$ can even not be hypoelliptic, that is,  $\Box^{(0)}_{b}u\in\mathscr{C}^\infty(X)$ cannot guarantee $u\in\mathscr{C}^\infty(X)$. We extend $\overline{\partial}_b$ to the $L^2$-space by
\begin{equation}\label{e-gue201223yydu}
\overline{\partial}_b: \Dom\overline{\partial}_b\subset L^2(X)\to 
L^2_{(0,1)}(X),\quad
\Dom\overline{\partial}_b:=\{u\in L^2(X):\overline{\partial}_bu
\in L^2_{(0,1)}(X)\}.
\end{equation} 
Let $\overline{\partial}^*_{b,H}: \Dom\overline{\partial}^*_{b,H}\subset L^2_{(0,1)}(X)\to L^2(X)$ be the $L^2$ adjoint of $\overline{\partial}_{b}$. Let 
$\Box^{(0)}_b: \Dom\Box^{(0)}_b\subset L^2(X)\to L^2(X)$, where
\[\Dom\Box^{(0)}_b=\{u\in L^2(X): u\in\Dom\overline{\partial}_b, \overline{\partial}_bu\in\Dom\overline{\partial}^*_{b,H}\}\] and $\Box^{(0)}_bu=\overline{\partial}^*_{b,H}\,\overline{\partial}_{b}u$, $u\in \Dom\Box^{(0)}_b$. 

\begin{definition}
The space of $L^2$ CR functions is given by 
\begin{equation}
H_b^{0}(X):=\Ker\Box^{(0)}_{b}:=
\left\{u\in\Dom\Box^{(0)}_{b}:
\Box^{(0)}_{b} u=0~\text{in the sense of distributions}\right\}.
	\end{equation}
\end{definition}

\begin{definition}
	The orthogonal projection
	\begin{equation}
		\Pi: L^2(X)\to\ker\Box^{(0)}_{b}
	\end{equation}
	is called the Szeg\H{o} projection on $X$, and we call
	its Schwartz kernel $\Pi(x,y)\in\mathscr{D}'(X\times X)$ the Szeg\H{o} kernel on $X$.
\end{definition} 

In our case, we can understand the singularities of $\Pi$ by microlocal analysis as in \cite{BS75}, see also \cites{Hs10,HsM17}. 

\begin{theorem}[\cite{Hs10}*{Part I Theorem 1.2}]
	\label{Boutet-Sjoestrand theorem}  
Let $(X,T^{1,0}X)$ be an orientable 
compact strictly pseudoconvex Cauchy--Riemann manifold with volume form \(dV\) such that the Kohn Laplacian on $X$ has closed range in $L^2(X)$. For any coordinate chart $(D,x)$ in $X$, we have 
	\begin{equation}
		\label{eq:Szego FIO}
		\Pi(x,y)\equiv\int_0^{+\infty} e^{it\varphi(x,y)}s^\varphi(x,y,t)dt,
	\end{equation}
	where
\begin{equation}\label{eq:varphi(x,y)}
\begin{split}
			&\varphi(x,y)\in \mathscr{C}^\infty(D\times D),\\
			&\mbox{${\rm Im\,}\varphi(x,y)\geq 0$},\\
			&\mbox{$\varphi(x,y)=0$ if and only if $y=x$},\\
			&\mbox{$d_x\varphi(x,x)=-d_y\varphi(x,x)=\xi(x)$},
		\end{split}
	\end{equation}
	and
	\begin{equation}
		\label{s(x,y,t)}
		\begin{split}
			&s^\varphi(x,y,t)\in S^n_{1,0}(D\times D\times\mathbb R_+),\\
			&s^\varphi(x,y,t)\sim\sum^{+\infty}_{j=0}s^\varphi_j(x,y)t^{n-j}~\text{in}~S^n_{1,0}(D\times D\times\mathbb R_+),\\
			&s^\varphi_j(x,y)\in\mathscr{C}^\infty(D\times D),~j=0,1,2,\ldots,\\
			&s^\varphi_0(x,x)=\frac{1}{2\pi^{n+1}}\frac{dV_{\xi}}{dV}(x)\neq 0.
		\end{split}
	\end{equation}
\end{theorem}
 In the local picture, near any point $p\in X$, let $\{W_j\}_{j=1}^n$ be an orthonormal frame of 
 $T^{1,0}X$ in a neighborhood of $p$ such that $\mathcal{L}_p(W_j,\overline{W}_s)=\delta_{js}\mu_j$, 
 for some $\mu_j>0$ and $j,s\in\{1,\ldots,n\}$.
 Define
\begin{equation}
	\det\mathcal{L}_p:=\Pi_{j=1}^n\mu_j(p).
\end{equation} 
Then
\begin{equation}
dV_\xi(x)=\det\mathcal{L}_x\sqrt{\det(g)}\,dx.
\end{equation}
where $g=(g_{jk})_{j,k=1}^{2n+1}$, 
$g_{jk}=\langle\frac{\partial}{\partial x_j}|
\frac{\partial}{\partial x_k}\rangle$.

We have the following local description of the tangential Hessian of the phase function calculated in \cite{Hs10}*{Part I Chapter 8}. Such local picture is essential for calculating the leading coefficient of the Szeg\H{o} kernel expansion \cites{BS75,Hs10,HsM17}. Also, it turns out that it is crucial in our proof of the CR embedding theorem (Theorem~\ref{thm:ProjectiveEmbeddingIntro}).
\begin{theorem}[\cite{Hs10}*{Part I Theorem 1.4}]
	\label{tangential hessian of varphi}
	For a given point $p\in X$, let $\{W_j\}_{j=1}^n$ be an orthonormal frame of $T^{1,0}X$ in a neighborhood of $p$ such that $\mathcal{L}_p(W_j,\overline{W}_s)=\delta_{js}\mu_j$, for some $\mu_j>0$ and $j,s\in\{1,\ldots,n\}$. Take local coordinates near $p$ such that
	\begin{equation}
		\label{eq_loccoord}
		x(p)=0,\quad \xi(p)=dx_{2n+1},\quad
		\mathcal{T}=\frac{\partial}{\partial x_{2n+1}},
	\end{equation}
	and 
	\begin{equation}\label{eq:LocalCoordnitatesT10X}
		\begin{split}
			W_j=\frac{\partial}{\partial z_j}+i\mu_j\overline{z}_j\frac{\partial}{\partial x_{2n+1}}-d_jx_{2n+1}\frac{\partial}{\partial x_{2n+1}}+\sum_{l=1}^{2n}e_{jl}(x)\frac{\partial}{\partial x_l},
		\end{split}
	\end{equation} 
	where $d_j\in\mathbb C$, $e_{jl}=O(|x|)$, $e_{jl}\in\cC^\infty$, $j\in\{1,\ldots,n\}$, $l\in\{1,\ldots,2n\}$. Then, for $\varphi$ in \eqref{eq:varphi(x,y)}, we have 
	\begin{equation}
		\label{tangential Hessian of varphi in Theorem}
		\begin{split}
			\varphi(x,y)&=x_{2n+1}-y_{2n+1}+i\sum_{j=1}^n\mu_j|z_j-w_j|^2+\sum_{j=1}^n i\mu_j(\overline{z}_jw_j-z_j\overline{w}_j)\\
			&+\sum_{j=1}^n \left(d_j(z_jx_{2n+1}-
			w_jy_{2n+1})+\overline{d}_j(\overline{z}_jx_{2n+1}-\overline{w}_jy_{2n+1})\right)\\
			&+(x_{2n+1}-y_{2n+1})f(x,y)+O\left(|(x,y)|^3\right),
		\end{split}
	\end{equation}
	where $f\in\mathscr{C}^\infty(D\times D)$ satisfies 
	$f(x,y)=\overline{f}(y,x)$ and $f(0,0)=0$.
\end{theorem}
We notice that under such coordinates, when the neighborhood is small enough, we can find some constant $C>0$ such that
\begin{equation}
	\label{eq:Im varphi>C|z-w|^2}
	\Im\varphi(x,y)\geq C|z-w|^2.
\end{equation}
See \cite{Hs10}*{Part I Proposition 7.16} for a proof of \eqref{eq:Im varphi>C|z-w|^2}. 
\subsection{Semi-classical spectral asymptotics for Toeplitz operators}
We recall the main result of~\cite{HHMS23} which is based on 
Theorem~\ref{Boutet-Sjoestrand theorem}. 
The following theorem  is essential for the proofs of the results in this work.  
\begin{theorem}[\cite{HHMS23}*{Theorem 1.1}]\label{thm:ExpansionMain}
	Let $(X,T^{1,0}X)$ be an orientable 
	compact strictly pseudoconvex Cauchy--Riemann manifold
	of dimension $2n+1$, $n\geq1$, with volume form \(dV\)
	such that the Kohn Laplacian on $X$ has closed range in 
	$L^2(X)$. Let $\xi$ be a contact form on $X$ such that the Levi form
	$\mathcal{L}=\frac12d\xi(\cdot,J\cdot)$ is positive definite.
	Let $(D,x)$ be any coordinates patch and let 
	$\varphi:D\times D\to\C$ be the phase function satisfying 
	\eqref{eq:Szego FIO} and \eqref{eq:varphi(x,y)}.
	Then for any formally self-adjoint first order pseudodifferential operator 
	$P\in L^1_\mathrm{cl}(X)$ with
	$\sigma_P(\xi)>0$ on $X$,  
	and for any $\chi\in\cC^\infty_c((0,+\infty))$, $\chi\not\equiv 0$,
	the Schwartz kernel of $\chi_k(T_P)$, 
	$\chi_{k}(\lambda):=\chi\left(k^{-1}\lambda\right)$ \(k>0\),
	can be represented for $k$ large by
	\begin{equation}
		\label{asymptotic expansion of chi_k(T_P)}
		\mbox{$\chi_k(T_P)(x,y)=\int_0^{+\infty} 
			e^{ikt\varphi(x,y)}{A}(x,y,t,k)dt+O\left(k^{-\infty}\right)$ on $D\times D$},
	\end{equation}
	where ${A}(x,y,t,k)\in S^{n+1}_{\mathrm{loc}}
	(1;D\times D\times{\R}_+)$,
	\begin{equation}
		\label{Eq:LeadingTermMainThm}
		\begin{split}
			&{A}(x,y,t,k)\sim\sum_{j=0}^\infty {A}_{j}(x,y,t)k^{n+1-j}~
			\mathrm{in}~S^{n+1}_{\mathrm{loc}}(1;D\times D\times{\R}_+),\\
			&A_j(x,y,t)\in\mathscr{C}^\infty(D\times D\times{\R}_+),~j=0,1,2,\ldots,\\
			&{A}_{0}(x,x,t)=\frac{1}{2\pi ^{n+1}}
			\frac{dV_{\xi}}{dV}(x)\,\chi(\sigma_P(\xi_x)t)t^n\not\equiv 0,
		\end{split}
	\end{equation}
	and for some compact interval $I\Subset\R_+$,
	\begin{equation}
		\begin{split}
			\operatorname{supp}_t A(x,y,t,k),~\operatorname{supp}_t A_j(x,y,t)\subset I,\ \ j=0,1,2,\ldots.
		\end{split}
	\end{equation}
	Moreover, for any $\tau_1,\tau_2\in\cC^\infty(X)$ 
	such that $\operatorname{supp}(\tau_1)\cap\operatorname{supp}(\tau_2)=\emptyset$, 
	we have
	\begin{equation}
		\label{Eq:FarAwayDiagonalMainThm}
		\tau_1\chi_k(T_P)\tau_2=O\left(k^{-\infty}\right).
	\end{equation}
\end{theorem} 
From Theorem~\ref{thm:ExpansionMain} we can deduce the following useful lemma as in~\cite{HHMS23}.
\begin{lemma}[cf.\ \cite{HHMS23}*{Lemma~4.12}]\label{lem:firstDifferentialMainThm}
	In the situation of Theorem~\ref{thm:ExpansionMain}  
	we have in $\mathscr{C}^\infty$-topology as $k\to+\infty$,
	\begin{equation}
		d_x\chi_k(T_P)(x,y)|_{y=x}=i\xi(x)\frac{k^{n+2}}{2\pi^{n+1}}
		\frac{dV_{\xi}}{dV}(x)\sigma_P(\xi_x)^{-n-2}
		\int_0^{+\infty} t^{n+1}\chi(t)dt+O(k^{n+1}),
	\end{equation}
	and
	\begin{equation}
		\left(d_x\otimes d_y\right)\chi_k(T_P)(x,y)|_{y=x}=\xi(x)\otimes\xi(x)\frac{k^{n+3}}{2\pi^{n+1}}\frac{dV_{\xi}}{dV}(x)\sigma_P(\xi_x)^{-n-3}\int_0^{+\infty} t^{n+2}\chi(t)dt+O(k^{n+2}).
	\end{equation}
\end{lemma}

\section{Projective Embeddings via \(T_P\)}\label{sec:ProjectiveEmbedding}
Let \((X,T^{1,0}X)\) and \(T_P\) be as in Theorem~\ref{thm:ProjectiveEmbeddingIntro}. In this section we are going to prove Theorem~\ref{thm:ProjectiveEmbeddingIntro}. More precisely, for each \(k>0\) we will construct a CR map \(F_k\colon X\to \C^{N_k+1}\) related to \(T_P\) and show that \(X\ni x\mapsto [F_k(x)]\in \C\mathbb{P}^{N_k}\) defines a CR embedding of \(X\) into the complex projective space when \(k\) is sufficiently large (see Theorem~\ref{thm:EmbedinCSCPMain}). As a consequence of this result we obtain Theorem~\ref{thm:ProjectiveEmbeddingIntro}. 

First, we will define two objects \(h^F,\mathcal{H}^F\) (see Definition~\ref{def:hfkHf}) which determine (see Lemma~\ref{lem:FProjectiveEmbeddingHh}) whether the projectivization of a map \(F\colon M\to \C^{N+1}\setminus{\{0\}}\), defined on any compact smooth real manifold \(M\) is an embedding. Then we will construct maps \(F_k\), \(k>0\), defined on \(X\) and such that \(h^{F_k},\mathcal{H}^{F_k}\) are directly linked to the functional calculus of \(T_P\). Using Theorem~\ref{thm:ExpansionMain} we can analyze the behavior of \(h^{F_k}\) and \(\mathcal{H}^{F_k}\) when \(k\) becomes large (see Corollary~\ref{cor:Hfkisnegativedefinite} and Lemma~\ref{lem:PropertiesOfhkForX}). The embedding result follows. At the end of this section we will prove two integral estimates (see Lemma~\ref{lem:IntegralSzegoDerivativeXtimesX} and Lemma~\ref{lem:EstimateIntegralhk}) related to \(h^{F_k}\) which we will need for the proof of the results in Section~\ref{sec:EquidistributionCR} and Section~\ref{sec:EquidistributionOnSPCdomains}.
\subsection{Projective Embeddability}
\begin{definition}\label{def:hfkHf}
	Let \(M\) be a real manifold and \(F\colon M\to \C^{N+1}\setminus\{0\}\) a smooth map.  We define the symmetric bilinear form \(\mathcal{H}^F=\{\mathcal{H}^F_x\}_{x\in M}\) on \(TM\) by \(\mathcal{H}^F_x\colon T_xM\times T_xM\to \R\), 
	\[\mathcal{H}^F_x(X,Y)=\frac{1}{|F(x)|^4}\text{Re}\left(\langle XF(x),F(x)\rangle\overline{\langle YF(x),F(x)\rangle}-\langle XF(x),YF(x)\rangle |F(x)|^2 \right).\]
	Furthermore, we denote by \(h^F\colon M\times M\to \R\) the function given by
	\[h^F(x,y)=\frac{|\langle F(x),F(y)\rangle|^2}{|F(x)|^2|F(y)|^2}.\]
\end{definition}
We have that \(h^F\) and \(\mathcal{H}^F\) are related in the following way.
\begin{lemma}\label{lem:HFisHessianofg}
	Let \(M\) be a real manifold and \(F\colon M\to \C^{N+1}\setminus\{0\}\) a smooth map. Given \(y\in M\) define \(g^F_y\colon M\to [0,1]\), \(g^F_y(x)=h^F(x,y)\).
	Then \(g^F_y\) is well defined and has a local maximum at \(x=y\) with \(g^F_y(y)=1\). Furthermore, the Hessian of \(g_y\) at \(x=y\) is given by \(\text{Hess}_y(g^F_y)=\mathcal{H}^F_y\).
\end{lemma}
\begin{proof}
	First recall that given a smooth function \(f\colon M\to \R\) which has a critical point at \(x\in M\) we have that the Hessian \(\text{Hess}_x(f)\colon T_xM\times T_xM\to \R\) of \(f\) at \(x\) is defined  by \(\text{Hess}_x(f)(X,Y)=(\tilde{X}(\tilde{Y}(f))(x)\) where \(\tilde{X}\) and \(\tilde{Y}\) are smooth vector fields on \(M\) with \(\tilde{X}_x=X\) and \(\tilde{Y}_x=Y\). Since \((df)_x=0\)  it turns out that \(\text{Hess}_x(f)(X,Y)\) is independent of the choice of \(\tilde{X}\) and \(\tilde{Y}\). Furthermore, since \(df([\tilde{X},\tilde{Y}]_x)=0\) it follows that \(\text{Hess}_x(f)\) is symmetric. 
	
	We have that \(g_y\) is smooth. From the Cauchy-Schwarz inequality we find that \(g_y\) is well defined and \(g_y(y)=1\). It follows that \(g_y\) has a local maximum at \(x=y\) and hence \((dg_y)_y=0\). Let \(X\in \Gamma(M,TM)\) be a smooth real vector field.
	Putting \(\tilde{F}=F/|F|\) we find for any \(v\in \C^{N+1}\) that 
	\[X|\langle \tilde{F},v\rangle|^2=\langle (X\tilde{F}),v\rangle\overline{\langle \tilde{F},v\rangle}+\langle \tilde{F},v\rangle\overline{\langle X\tilde{F},v\rangle}\] and 
	\begin{eqnarray*}
		X(X|\langle \tilde{F},v\rangle|^2)&=&\langle X(X\tilde{F}),v\rangle\overline{\langle \tilde{F},v\rangle}+\langle \tilde{F},v\rangle\overline{\langle X(X\tilde{F}),v\rangle}+2|\langle X\tilde{F},v\rangle|^2
	\end{eqnarray*}
	Since \(|\tilde{F}|\equiv 1\) we find
	\[0=X|\tilde{F}|^2=2\text{Re}(\langle X\tilde{F},\tilde{F}\rangle).\]
	With \(v=\tilde{F}(y)\) we conclude that
	\begin{eqnarray}\label{eq:HessianFubiniDistance01}
		(X(Xg_y))|_{x=y}&=& 
		2\text{Re}(\langle X(X\tilde{F}),\tilde{F}\rangle)
		+2\text{Im}(\langle X\tilde{F},\tilde{F}\rangle)^2
	\end{eqnarray}
	Since \(\tilde{F}=F/|F|\) we have
	\begin{eqnarray*}
		X\tilde{F}=\frac{XF}{|F|}-\frac{\text{Re}(\langle XF, F\rangle)F}{|F|^3}
	\end{eqnarray*}
	and hence
	\begin{eqnarray}\label{eq:HessianImaginaryPartEqual}
		\langle X\tilde{F},\tilde{F}\rangle =\frac{i\text{Im}(\langle XF, F\rangle)}{|F|^2}.
	\end{eqnarray}
	Furthermore, we calculate
	\begin{eqnarray*}
		X(X\tilde{F})&=&\frac{X(XF)}{|F|}-\frac{\text{Re}(\langle XF, F\rangle)XF}{|F|^3}-\frac{\text{Re}(\langle XF, F\rangle)XF}{|F|^3}\\
		&&-\frac{\text{Re}(\langle X(XF),F\rangle)F}{|F|^3}-\frac{\langle XF,XF\rangle F}{|F|^3}+3\frac{\text{Re}(\langle XF, F\rangle)^2F}{|F|^5}
	\end{eqnarray*}
	which leads to
	\begin{eqnarray*}
		\text{Re}(\langle X(X\tilde{F}),\tilde{F}\rangle)=- \frac{\langle XF,XF\rangle}{|F|^2}+\frac{\text{Re}(\langle XF, F\rangle)^2}{|F|^4}
	\end{eqnarray*}
	With~\eqref{eq:HessianFubiniDistance01} and~\eqref{eq:HessianImaginaryPartEqual} we conclude
	\begin{eqnarray*}
		((X(Xg_y))(x))|_{x=y}&=& \frac{|\langle XF,F\rangle|^2-\langle XF,XF\rangle|F|^2}{|F|^4}(x). 
	\end{eqnarray*}
	We have shown that \(\text{Hess}_y(g_y)(X,X)=\mathcal{H}_y(X,X)\) for all \(X\in T_xM\). Since for \(X,Y\in T_XM\) we have 
	\[\text{Hess}_y(g_y)(X,Y)=\frac{1}{2}\left(\text{Hess}_y(g_y)(X+X,X+Y)-\text{Hess}_y(g_y)(X,X)-\text{Hess}_y(g_y)(Y,Y)\right)\]
	the claim follows from a direct calculation of \(\mathcal{H}_y(X+Y,X+Y)\).
\end{proof}
We are interested in understanding whether \(M\ni x\mapsto [F(x)]\in \C\mathbb{P}^{N}\) defines an embedding with respect to the properties of  \(h^F\) and \(\mathcal{H}^F\). 
When \(M\) is compact, necessary and sufficient conditions can be obtained from the following observation.
\begin{lemma}\label{lem:hfHFRelatedFubinStudy}
	 Denote by 
	\(\operatorname{dist}^{\C}\colon \C\mathbb{P}^N\times \C\mathbb{P}^N\to \R  \) the Fubini-Study distance on \(\C\mathbb{P}^N\) that is
	\[\operatorname{dist}^{\C}([v],[w])=\sqrt{1-\frac{|\langle v,w\rangle|}{|v||w|}}, \text{ for }[v],[w]\in \C\mathbb{P}^N.\]
	Then \(\operatorname{dist}^\C([F(x)],[F(y)])=\sqrt{1-\sqrt{h^F(x,y)}}\) for all \(x,y\in M\). Furthermore, let \(ds_{\operatorname{FS}}^2\) be the Fubini-Study metric on \(\C\mathbb{P}^N\). We have \(\Re ([F]^*ds_{\operatorname{FS}}^2)=-\mathcal{H}^F\).
\end{lemma} 
\begin{proof}
	The first part of the statement is obvious. In order to prove the second part consider the map \(pr\colon \C^{N+1}\setminus\{0\}\to\C\mathbb{P}^N\), \(pr(z)=[z]\). We find
	\[(pr^*ds_{\operatorname{FS}}^2)(z)=\frac{|z|^2\sum_{j=0}^Ndz_j\otimes d\overline{z}_j -\left(\sum_{j=0}^{N}\overline{z}_jdz_j\right)\otimes\left(\sum_{j=0}^{N}z_jd\overline{z}_j\right)}{|z|^4}\]
	and hence  \(\Re ([F]^*ds_{\operatorname{FS}}^2)=\Re( (pr\circ F)^*ds_{\operatorname{FS}}^2)=-\mathcal{H}^F\).
\end{proof}
Using Lemma~\ref{lem:hfHFRelatedFubinStudy} we obtain the following.
\begin{lemma}\label{lem:FProjectiveEmbeddingHh}
	Let \(M\) be a compact smooth real manifold and \(F\colon M\to \C^{N+1}\setminus\{0\}\) a smooth map. The following are equivalent.
	\begin{itemize}
		\item [(i)] The map \([F]\colon M\to \C \mathbb{P}^N\), \([F](x)=[F(x)]\) is an embedding of \(M\) into \(\C \mathbb{P}^N\).
		\item [(ii)] One has \(h^F(x,y)=1\) if and only if \(x=y\) and \(\mathcal{H}^F_x\) is negative definite for all \(x\in M\). 
	\end{itemize}	
\end{lemma}
\begin{proof}
		Since \(M\) is compact we have that \([F]\) is an embedding if and only if it is an injective immersion. Then the equivalence of (i) and (ii) follows from Lemma~\ref{lem:hfHFRelatedFubinStudy}.
\end{proof}

\subsection{Proof of Theorem~\ref{thm:ProjectiveEmbeddingIntro}}
Now let \((X,T^{1,0}X)\) be a compact orientable strictly pseudoconvex CR manifold such that the Kohn-Laplacian has closed range. Denote by \(\xi\) a contact form on \(X\) such that the respective Levi form is positive definite and let \(\mathcal{T}\) be the respective Reeb vector field uniquely determined by \(\iota_\mathcal{T}\xi\equiv1\) and \(\iota_\mathcal{T}d\xi\equiv0\). Let  $P\in L^1_{\mathrm{cl}}(X)$ be a  first order formally self-adjoint classical pseudodifferential operator with \(\sigma_P(\xi)>0\) where   $\sigma_P$ denotes the principal symbol of $P$. Denote by \(T_P=\Pi P\Pi\) the corresponding Toeplitz operator. Let \(0<\lambda_1\leq\lambda_2\leq\ldots\) be the positive eigenvalues of \(T_P\) counting multiplicity and let \(f_1,f_2\ldots\in H_{b}^0(X)\cap\mathscr{C}^\infty(X)\) be a respective orthonormal system of eigenvectors. In particular, we have \(T_Pf_j=\lambda_jf_j\), \((f_j, f_j)=1\) and \((f_j, f_\ell)=0\) for all \(j,\ell\in \N\), \(j\neq \ell\).  Let \(\chi\in \mathscr{C}^\infty((\delta_1,\delta_2))\), \(\chi\not\equiv 0\), be a cut-off function for some \(0<\delta_1<\delta_2<1\) and put \(\eta(t):=|\chi(t)|^2\), \(t\in \R\), with \(\tau_j:=\int_\R t^{n+j}\eta(t)dt\) for \(j\in\N_0\). Put \(N_k=|\{j\in \N\mid\lambda_j\leq k\delta_2\}|\).  Let \(\kappa\colon \R\to \C\) be a  function such that \(|\kappa(k)|^2/ (1+|k|^{n})\) is bounded in \(k\).
For \(k>0\) define \(F_k\colon X\to \C^{N_k+1}\),
\begin{eqnarray}\label{eq:DefinitionFKwKappa}
	F_k\colon X\to \C^{N_k+1}, \,\,\,\,\, F_k(x)=(\kappa(k),\chi(k^{-1}\lambda_1)f_1(x),\ldots,\chi(k^{-1}\lambda_{N_k})f_{N_k}(x)).
\end{eqnarray}
The term \(\kappa(k)\) is introduced here for later applications. We will mainly consider \(\kappa\equiv 0\) or \(\kappa\equiv 1\) then. Using Lemma~\ref{lem:FProjectiveEmbeddingHh} we would like to show that the projectivization \([F_k]\) of \(F_k\) is an embedding when \(k\) is sufficiently large. We have the following.
\begin{theorem}\label{thm:EmbedinCSCPMain}
	With the assumptions in Theorem~\ref{thm:ProjectiveEmbeddingIntro}, \(F_k\) as in~\eqref{eq:DefinitionFKwKappa} and the notations above we have that there exists \(k_0>0\) such that \([F_k]\colon X\to \C\mathbb{P}^{N_k}\), \([F_k](x)=[F_k(x)]\), is well defined and a CR embedding of \(X\) into \(\C\mathbb{P}^{N_k}\) for all \(k\geq k_0\). Furthermore, we have that \([F_k]^*ds_{FS}^2\) has an asymptotic expansion in \(k\). More precisely, there exist smooth sesquilinear forms \(r_j\colon \C TX\times \C TX\to \C\), \(j\in\N\), such that for any \(M,\ell\in\N\) there exists a constant \(C_{M,\ell}>0\) with
	\[\left\|[F_k]^*ds^2_{FS} -\sum_{j=0}^Mk^{2-j}r_j\right\|_{\mathscr{C}^\ell(X)}\leq Ck^{-M+1},\,\,\,\text{ for all }k\geq k_0.\]
	In addition,  we have \(r_0=\frac{\tau_2\tau_0-\tau_1^2}{\tau_0^2 \sigma_P(\xi)^2}\xi\otimes \xi\) and \(r_1(Z,W)=\frac{\tau_1}{i\tau_0 \sigma_P(\xi)}d\xi(Z,\overline{W})\) for all \(Z,W\in T_x^{1,0}X\), \(x\in X\).
\end{theorem}
\begin{proof}
	Since \(|F_k(x)|^2=|\kappa(k)|^2+\eta_k(T_P)(x,x)\) it follows from Theorem~\ref{thm:ExpansionMain} and the assumptions on \(\kappa\) that there exists \(C,k_1>0\) such that \(|F_k(x)|^2\geq C\) for all \(x\in X\) and \(k\geq k_1\). Hence \([F_k]\) is well defined for \(k\geq k_1\). Furthermore, by definition we have that \(F_k\) is a CR map for all \(k>0\). Since \(\C^{N+1}\setminus \{0\} \ni v\mapsto [v]\in \C\mathbb{P}^N\) is holomorphic it follows that \([F_k]\) is a CR map for all \(k\geq k_1\). Since \(X\) is a CR manifold of codimension one we have that a CR map which is an embedding is automatically a CR embedding. Hence it remains to show that \([F_k]\) is an embedding when \(k\) is large. In view of Lemma~\ref{lem:FProjectiveEmbeddingHh} the first part of the claim follows immediately from Corollary~\ref{cor:Hfkisnegativedefinite} and Lemma~\ref{lem:PropertiesOfhkForX}  below. \\
	For the second part of the claim we consider
	\[[F_k]^*ds_{FS}^2=(pr\circ F_k)^*ds_{FS}^2=\frac{|F_k|^2\sum_{j=1}^{N_k}\eta_k(\lambda_j)df_j\otimes d\overline{f}_j- \left(\sum_{j=0}^{N}\eta_k(\lambda_j)\overline{f}_jdf_j\right)\otimes\left(\sum_{j=0}^{N_k}\eta_k(\lambda_j)f_jd\overline{f}_j\right)}{|F_k|^4}\]
	with \(pr\colon \C^{N_k+1}\setminus\{0\}\to \C\mathbb{P}^{N_k}\), \(pr(z)=[z]\). Since \(|F_k(x)|^2=|\kappa(k)|^2+\eta_k(T_P)(x,x)\) we obtain
	\begin{eqnarray*}
		[F_k]^*ds_{FS}^2=\frac{(|\kappa(k)|^2+\eta_k(T_P)(x,x))(d_x\otimes d_y \eta_k(T_P)(x,y))|_{x=y}-d_x\eta_k(T_P)(x,y)|_{x=y}\otimes \overline{d_x\eta_k(T_P)(x,y)|_{x=y}}}{(|\kappa(k)|^2+\eta_k(T_P)(x,x))^2}.
	\end{eqnarray*}
Hence from Lemma~\ref{lem:firstDifferentialMainThm}, Theorem~\ref{thm:ExpansionMain} and the assumption on \(\kappa\), we find that there exist smooth bilinear forms \(\tilde{r}_j\colon TX\times TX\to \C\), \(j\in\N\), on \(TX\) such that \([F_k]^*ds_{FS}^2\sim \sum_{j=0}^\infty k^{2-j}\tilde{r}_j\) has an asymptotic expansion in \(\mathscr{C}^\infty\)-topology. Extending each \(\tilde{r}_j\) to a smooth sesquiliniear form \(r_j\) on \(\C TX\) the claim follows.   Applying Lemma~\ref{lem:firstDifferentialMainThm} again we immediately obtain
\[r_0=\frac{\tau_2\tau_0-\tau_1^2}{\tau_0^2\sigma_P(\xi)^2}\xi\otimes \xi.\]
Given \(Z,W\in T_x^{1,0}X\) for some \(x\in X\) we find with \(\xi(Z)=\xi(W)=0\) that
\begin{eqnarray*}
	r_1(Z,W)&=&\lim_{k\to\infty}\frac{(d_x\otimes d_y \eta_k(T_P)(x,y))|_{x=y}(Z,\overline{W})}{k(|\kappa(k)|^2+\eta_k(T_P)(x,x))}=\lim_{k\to\infty}\frac{\langle ZF_k,WF_k\rangle}{k|F_k|^2}.\\
\end{eqnarray*}
We can take smooth sections \(Z',W'\) with values in \(T^{1,0}X\) in a small neighborhood around \(x\) such that \(Z'_x=Z\) and \(W'_x=W\). It follows that
\begin{eqnarray*}
	\langle Z'F_k,W'F_k\rangle=\overline{W'}\langle Z'F_k,F_k\rangle-\langle \overline{W'}Z'F_k,F_k\rangle=\overline{W'}\langle Z'F_k,F_k\rangle+\langle [Z',\overline{W'}]F_k,F_k\rangle
\end{eqnarray*} 
with \(	\overline{W'}\langle Z'F_k,F_k\rangle/|F_k|^2 = O(1)\) by Lemma~\ref{lem:firstDifferentialMainThm}. Since
\begin{eqnarray*}
		\frac{\langle [Z',\overline{W'}]F_k,F_k\rangle}{|F_k|^2}(x)= \frac{d_x\eta_k(T_P)(x,y)|_{x=y}([Z',\overline{W'}]_x)}{|\kappa(k)|^2+\eta_k(T_P)(x,x)}= ik\frac{\tau_1}{\tau_0}\alpha_P([Z',\overline{W'}]_x)+O(1)
\end{eqnarray*}
by Lemma~\ref{lem:firstDifferentialMainThm} and using \(\xi([Z',\overline{W'}]_x)=-d\xi(Z,\overline{W})\) we find \(r_1(Z,W)=\frac{\tau_1}{i\tau_0\sigma_P(\xi)}d\xi(Z,\overline{W})\).
\end{proof}
As a consequence of Theorem~\ref{thm:EmbedinCSCPMain} we obtain Theorem~\ref{thm:ProjectiveEmbeddingIntro}.
\begin{proof}[\textbf{Proof of Theorem~\ref{thm:ProjectiveEmbeddingIntro}}]
	We note that \(\C\mathbb{P}^{N-1}\) can be isometrically identified with the set \(\{[z_0:\ldots:z_N]\in \C\mathbb{P}^N\mid z_0=0\}\). Then the claim follows from Theorem~\ref{thm:EmbedinCSCPMain} by taking \(\kappa\equiv 0\).
\end{proof}
In order to complete the proof of Theorem~\ref{thm:EmbedinCSCPMain} it remains to show that \(\mathcal{H}^{F_k}\) is negative definite and that \(h^{F_k}(x,y)=1\) if and only if \(x=y\). We have the following.
\begin{lemma}\label{lem:PropertiesOfgkForX}
	Put \(H=\{H_x\}_{x\in X}\), \(H_x:=\text{Re}( T^{1,0}_xX)\), and  
	define an isomorphism \(L_k\colon   TX \to   TX\) by \(L_k(\mathcal{T})=\frac{1}{k}\mathcal{T}\) and \(L_k(V)=\frac{1}{\sqrt{k}}V\) for \(V\in  H\). Denote by \(h\) the  positive symmetric bilinearform on \(H\) induced by the Levi-form. Extend \(h\) to a symmetric bilinearform on \(TX\) by \(h(\mathcal{T},\cdot)\equiv 0\). Moreover, let \(\langle \cdot\mid\cdot\rangle\) be a smooth Hermitian metric on \(\C TX\) as in Section~\ref{sec:CRmanifoldsMicroLocal}. There exists positive smooth  functions \(a,b\colon X\to \R_+\) such that
	\[L_k^*(\mathcal{H}^{F_k})=- a\xi\otimes\xi-bh+O(k^{-\frac{1}{2}})\]
	for all sufficiently large \(k>0\).
	More precisely, there exists \(k_0,C>0\) such that
	\[\left|L_k^*(\mathcal{H}^{F_k})(V,W) + a\xi(V)\xi(W)+bh(V,W)\right|\leq \frac{C}{\sqrt{k}}|V||W|\]
	for all \(k\geq k_0\), \(V,W \in T_pX\), \(p\in X\).
\end{lemma}
\begin{proof}
	We choose \(k_0,C_0>0\) such that \(|\kappa(k)|^2+\eta_k(T_P)(x,x)\geq C_0k^{n+1}\) for all \(x\in X\) and \(k\geq k_0\).
	With Theorem~\ref{thm:ExpansionMain} and Lemma~\ref{lem:firstDifferentialMainThm} we have 
	\begin{eqnarray*}
		L_k^*(\mathcal{H}^{F_k})(\mathcal{T},\mathcal{T})&=&k^{-2}\left(\frac{|(d_x\eta_k(T_P)(x,y)|_{x=y})(\mathcal{T})|^2}{||\kappa(k)|^2+\eta_k(T_P)(x,x)|^2}- \frac{|(d_x\otimes d_y\eta_k(T_P)(x,y)|_{x=y})(\mathcal{T}\otimes \mathcal{T})|}{|\kappa(k)|^2+\eta_k(T_P)(x,x)} \right)\\
		&=& \frac{\tau_1^2-\tau_2\tau_0}{\sigma_P(\xi)^2\tau_0^2}+O(k^{-1})
	\end{eqnarray*}
	where \(\tau_j:=\int_\R\eta(t)t^{n+j} dt\). By the Cauchy-Schwarz inequality we find \(\tau_1< \sqrt{\tau_2}\sqrt{\tau_0}\). Hence there exist a positive function  \(a\colon X\to\R_+\) and a constant \(C_1>0\) with \(|L_k^*(\mathcal{H}^{F_k})(\mathcal{T},\mathcal{T})+a|\leq C_1k^{-1}\) for all \(k\geq k_0\). From Lemma~\ref{lem:firstDifferentialMainThm} we also have for \(V\in H_x\) that 
	\begin{eqnarray*}
		\frac{|(d_x\otimes d_y\eta_k(T_P)(x,y)|_{x=y})(\mathcal{T}\otimes V)}{|\kappa(k)|^2+\eta_k(T_P)(x,x)}&=& O(k)\\
		\frac{(d_x\eta_k(T_P)(x,y)|_{x=y})(\mathcal{T})(d_x\eta_k(T_P)(x,y)|_{x=y})(V)}{||\kappa(k)|^2+\eta_k(T_P)(x,x)|^2}&=& O(k).
	\end{eqnarray*}
	We conclude that there exist \(C_2>0\) such that \(|L_k^*(\mathcal{H}^{F_k})(\mathcal{T},V)|=|k^{-3/2}\mathcal{H}^{F_k}(\mathcal{T},V)|\leq C_2k^{-\frac{1}{2}}\|V\|\) holds for all \(V\in H\) and all \(k\geq k_0\) and \(V\in H\). Furthermore, from Lemma~\ref{lem:firstDifferentialMainThm} we also find that there exists a constant \(C_3>0\) such that
	\begin{eqnarray*}
		\left|\frac{(d_x\eta_k(T_P)(x,y)|_{x=y})(V_1)(d_x\eta_k(T_P)(x,y)|_{x=y})(V_2)}{||\kappa(k)|^2+\eta_k(T_P)(x,x)|^2}\right|&\leq& C_3|V_1||V_2|.
	\end{eqnarray*}
	for all \(k\geq k_0\), all \(x\in X\) and all \(V_1,V_2\in H_x\).
	We define a symmetric bilinearform on \(H\) by \(R_k=\{(R_k)_x\}_{x\in X}\), \((R_k)_x\colon H_x\times H_x\to \C \),
	\[(R_k)_x(V_1,V_2)=\frac{\text{Re}\langle V_1F_k(x),V_2F_k(x)\rangle}{|F_k(x)|^2}.\]
	From Theorem~\ref{thm:ExpansionMain} we find bilinearforms \(r_j=\{(r_j)_x\}_{x\in X}\), \((r_j)_x\colon H_x\times H_x\to \R \), \(j=0,1\), and a constant \(C_4>0\) such that 
	\[|R_k(V_1,V_2)-k^{2}r_0(V_1,V_2)-k^1r_1(V_1,V_2)|\leq C_4|V_1||V_2|\]
	holds for all \(k\geq k_0\), \(x\in X\) and \(V_1,V_2\in H_x\). From Lemma~\ref{lem:firstDifferentialMainThm} we obtain \(r_0\equiv 0\). So we will calculate \(r_1\). Given \(x\in X\) and \(V_1,V_2\in H_x\) we can find smooth sections  \(Z_j\in\Gamma(U,T^{1,0}X)\) with \(2\text{Re}(Z_j)_x=V_j\), \(j=1,2\) defined in an open neighborhood \(U\) around \(x\).
	At \(x\in X\) we find
	\[\langle V_1 F_k,V_2 F_k\rangle = \langle Z_1 F_k,Z_2 F_k\rangle=\overline{Z_2}(Z_1\langle F_k, F_k\rangle)-\langle \overline{Z_2}(Z_1 F_k), F_k\rangle=\overline{Z_2}(Z_1\langle F_k, F_k\rangle)+\langle [Z_1,\overline{Z_2}] F_k, F_k\rangle.\]
	Since
	\begin{eqnarray*}
		\frac{\overline{Z_2}(Z_1 \eta_k(T_P)(x,x))}{|\kappa(k)|^2+\eta_k(T_P)(x,x)} &=& O(1),\\
		\frac{(d_x\eta_k(T_P)(x,y)|_{x=y})([Z_1,\overline{Z_2}])}{|\kappa(k)|^2+\eta_k(T_P)(x,x)}-k\frac{\tau_1i}{\sigma_P(\xi)\tau_0 }\xi([Z_1,\overline{Z}_2])&=& O(1),
	\end{eqnarray*}
	by Lemma~\ref{lem:firstDifferentialMainThm}, we find with \(\xi([Z_1,\overline{Z}_2])=-2i\mathcal{L}(Z_1,Z_2)\) that \(r_1(V_1,V_2)=\frac{2\tau_1}{\sigma_P(\xi)\tau_0 }h(V_1,V_2)\). Hence we conclude that there exist a positive smooth function \(b\colon X\to \R\) and a constant \(C_5>0\) such that  \[|L_k^*(\mathcal{H}^{F_k})(V_1,V_2)+bh(V_1,V_2)|\leq C_5k^{-1}|V_1||V_2|\] holds for all \(k\geq k_0\) and all \(V_1,V_2\in H_x\), \(x\in X\). The statement follows.
\end{proof}
\begin{corollary}\label{cor:Hfkisnegativedefinite}
	There exists \(k_0>0\) such that \(\mathcal{H}_x^{F_k}\) is negative definite for all \(x\in X\) and \(k\geq k_0\).
\end{corollary}
\begin{lemma}\label{lem:PropertiesOfhkForX}
	With the assumptions and notation above	there exists  \(k_0>0\) such that for any \(k\geq k_0\) we have that  \(h^{F_k}(x,y)=1\) if and only if \(x=y\). 
\end{lemma}
We will prove Lemma~\ref{lem:PropertiesOfhkForX} by using that \(h^{F_k}\) is directly linked to the integral kernel of \(\eta_k(T_P)\). Let us recall this relation first.
We write \(S_k(x,y)=\eta_k(T_P)(x,y)\) and \(B_k(x)=\eta_k(T_P)(x,x)\). Note that there exists \(C_0>0\) such that \(|S_k|\leq C_0k^{n+1}\) and \(C_0^{-1}\leq k^{-n-1}B_k\leq C_0\) for all sufficiently large \(k>0\). With
\[h^{F_k}(x,y)=\frac{||\kappa(k)|^2+S_k(x,y)|^2}{(|\kappa(k)|^2 + B_k(x))(|\kappa(k)|^2 + B_k(y))}=\frac{|\kappa(k)|^4+|\kappa(k)|^2(S_k(x,y)+\overline{S_k(x,y)}) +|S_k|^2}{(|\kappa(k)|^2 + B_k(x))(|\kappa(k)|^2 + B_k(y))}\]
and using the assumptions on \(\kappa\) we find \(k_0>0\) and a constant \(C>0\) such that
\begin{eqnarray}\label{eq:hFkcloseToSzego}
	\left|h^{F_k}(x,y)-\frac{|S_k(x,y)|^2}{B_k(y)B_k(x)}\right|\leq \frac{C}{k}
\end{eqnarray}
for all \(k\geq k_0\) and all \(x,y\in X\).
\begin{remark}\label{rmk:StrategyOfProofEmbedding}
	  The strategy for deducing results like Lemma~\ref{lem:PropertiesOfhkForX} from Bergman and Szeg\H{o} kernel asymptotics in a various kind of set-ups is due to Bouche~\cite{Bch96}, Shiffman-Zelditch~\cite{SZ02}, Ma-Marinescu~\cite{MM07} and Hsiao, Li \cite{Hs15,HLM21,HHMS23}. The results in there are obtained via proof by contradiction arguments. In this work, we give a direct proof of  Lemma~\ref{lem:PropertiesOfhkForX} for the reason that we need to estimate integrals involving the function \(h^{F_k}\) (see Section~\ref{sec:IntegralEstimatesHFk}).     
\end{remark}
For the proof of Lemma~\ref{lem:PropertiesOfhkForX} we will mainly work in local coordinates as follows.
Let \(p\in X\) be a point. Choose local coordinates \((x,y)\),  \(x=(x_1,\ldots,x_{2n+1})=(x',x_{2n+1})\),  \(y=(y_1,\ldots,y_{2n+1})=(y',y_{2n+1})\) on an open neighborhood \(D\times D\subset X\times X\) around \((p,p)\) as in Theorem~\ref{tangential hessian of varphi}. 
With these coordinates we identify \(D\) with an open neighborhood in \(\R^{2n+1}\) such that \(p\) is identified with \(0\in\R^{2n+1}\). 
In these coordinates we may write on \(D\times D\)
\begin{eqnarray}
	S_k(x,y)=\int_{\delta_1}^{\delta_2} e^{ikt\varphi(x,y)}A(x,y,t,k)dt + O(k^{-\infty})
\end{eqnarray}
where \(\varphi(x,y)\) has the form~\eqref{tangential Hessian of varphi in Theorem}, \(T^{1,0}X\) is locally given by~\eqref{eq:LocalCoordnitatesT10X}  and \(A(x,y,t,k)\) satisfies~\eqref{Eq:LeadingTermMainThm}. We then define
\begin{eqnarray}
	\tilde{S}_k(x,y)=\int_{\delta_1}^{\delta_2} e^{ikt\varphi(x,y)}A(x,y,t,k)dt
\end{eqnarray}
and find for \(k_0>0\) and \(C>0\) large enough that 
 \begin{eqnarray}\label{eq:hFkcloseToSzegoLocal}
 	\left|h^{F_k}(x,y)-\frac{|\tilde{S}_k(x,y)|^2}{B_k(y)B_k(x)}\right|\leq \frac{C}{k}
 \end{eqnarray}
holds for all \(k\geq k_0\) and \(x,y\in D\).

Recall that there exists \(c>0\) such that 
\begin{eqnarray}
	\operatorname{Im}\varphi(x,y)\geq c|x'-y'|^2
\end{eqnarray} 
for all \(x,y\in D\). Furthermore, there exists \(\hat{C}>0\) such that
\begin{eqnarray}
	\left|\left(\frac{\partial}{\partial t}\right)^{\beta}\partial_x^{\alpha}A(x,y,t,k)\right|\leq \hat{C}k^{n+1}
\end{eqnarray}
for all \(k\geq 1\), \(|\alpha|,\beta\leq 3\) and all \(x,y\in D\). 
With these coordinates we will prove in three steps (see Lemma~\ref{lem:hFksmallsemiclose}, Lemma~\ref{lem:hFksmallclose} and Lemma~\ref{lem:hFksmallveryclose}) that Lemma~\ref{lem:PropertiesOfhkForX} holds true near the diagonal in \(X\times X\).
\begin{lemma}\label{lem:hFksmallsemiclose}
	There exist \(C,k_0>0\) such that \(h^{F_k}(x,y)\leq \frac{1}{2}\) for all \(x,y\in D\) with \(|x'-y'|\geq \frac{C}{\sqrt{k}}\) or \(|\Re\varphi(x,y)|\geq\frac{C}{k}\)
\end{lemma}
\begin{proof}
	Choose \(k_0,C_0>1\) such that~\eqref{eq:hFkcloseToSzegoLocal}, \(|\tilde{S}_k|\leq C_0k^{n+1}\) and \(C_0^{-1}\leq k^{-n-1}B_k\leq C_0\) hold for  all \(k\geq k_0\). 
	Assuming \(|x'-y'|\geq\frac{C}{\sqrt{k}}\) for some \(C>0\) we find
	\[|\tilde{S}_k(x,y)|\leq\int_{\delta_1}^{\delta_2}e^{-ktc|x'-y'|^2}|A(x,y,t,k)|dt \leq  e^{-\delta_1 cC}(\delta_2-\delta_1)\hat{C}k^{n+1}.\]
	for all \(k\geq k_0\).
	Assuming	\(|\text{Re}\varphi(x,y)|\geq\frac{C}{k}\)  we find
	\begin{eqnarray}
		|\tilde{S}_k(x,y)|&=&\left|\int_{\delta_1}^{\delta_2}\frac{1}{ik\varphi(x,y)}e^{ikt\varphi(x,y)}\frac{\partial}{\partial t}A(x,y,t,k)dt\right|\\
		&\leq&  \frac{1}{C}\int_{\delta_1}^{\delta_2} \left|\frac{\partial}{\partial t}A(x,y,t,k)\right|dt\leq \frac{\hat{C}}{C}(\delta_2-\delta_1)k^{n+1}
	\end{eqnarray}
	for all \(k\geq k_0\).
	From~\ref{eq:hFkcloseToSzegoLocal} we conclude that there exist \(C_1,C_2,k_1>0\) such that
	\[h^{F_k}(x,y)\leq C_1\max\{\frac{1}{C},e^{-\delta_1cC}\}+\frac{C_2}{k}\]
	for all \(C>0\), \(x,y\in D\) with \(|x'-y'|\geq \frac{C}{\sqrt{k}}\) or \(|\Re\varphi(x,y)|\geq\frac{C}{k}\) and all \(k\geq k_1\).
	The statement follows from taking \(C,k_1\) large enough.
\end{proof}
We are going to prove now a stronger version of Lemma~\ref{lem:hFksmallsemiclose}.
\begin{lemma}\label{lem:hFksmallclose}
	There exists an open neighborhood \(U\subset D\) around zero such that for all \(\varepsilon >0\) there exist \(\delta,k_0>0\) with \(h^{F_k}(x,y)\leq 1-\delta\) for all \(k\geq k_0\) and \(x,y\in U\) satisfying  \(|x'-y'|\geq \frac{\varepsilon}{\sqrt{k}}\) or \(|\Re\varphi(x,y)|\geq\frac{\varepsilon}{k}\).
\end{lemma}
For the proof of Lemma~\ref{lem:hFksmallclose} we need the following.
\begin{lemma}\label{lem:x2n1y2n1boundedByRe}
	There exists an open neighborhood \(U\subset D\) around zero and \(C>0\) such that 
	\[|x_{2n+1}-y_{2n+1}|\leq C(|\operatorname{Re}\varphi(x,y)|+|x'-y'|)\]
	and 
	\[\operatorname{Im}\varphi(x,y)\leq C(|\operatorname{Re}\varphi(x,y)|^2+|x'-y'|^2)\]
	for all \(x,y\in U\).
\end{lemma}
\begin{proof}
	By Taylor expansion of \(\text{Re}\varphi\) at \(x=y\) in the variable \(x\) we obtain
	\[|\text{Re}\varphi(x,y)-(d_x\text{Re}\varphi)_{x=y}(x-y)|\leq C_1(|x_{2n+1}-y_{2n+1}|^2+|x'-y'|^2)\]
	for some constant \(C_1>0\) for all \(x,y\in D\).
	Hence with 
	\[|(d_{x'}\text{Re}\varphi)_{x=y}(x'-y')|\leq C_2|x'-y'|\] for some constant \(C_2>0\) we obtain
	\[|x_{2n+1}-y_{2n+1}|\left(\left|\frac{\partial \text{Re}\varphi}{\partial x_{2n+1}}(y,y)\right|-C_1|x_{2n+1}-y_{2n+1}|\right)\leq  |\text{Re}\varphi(x,y)|+ |x'-y'|(C_2 +C_1|x'-y'|).\]
	Since \(\frac{\partial \text{Re}\varphi}{\partial x_{2n+1}}(0,0)\neq 0\) the first part of the claim follows from taking \(U\) small enough.
	In order to prove the second part we first observe that since \(\text{Im}\varphi\geq 0\) and \(\text{Im}\varphi(y,y)=0\) we have \((d_x\Im\varphi)_{x=y}=0\). Then it follows from Taylor expansion of  \(\text{Im}\varphi\) at \(x=y\) in the variable \(x\) that there exists \(C_3>0\) with
	\(\text{Im}\varphi(x,y)\leq C_3|x-y|^2\) 
	for all \(x,y\in U\). Since \(|x-y|^2=|x_{2n+1}-y_{2n+1}|^2+|x'-y'|^2\) the  claim follows from the already proven first part.
\end{proof}
\begin{lemma}\label{lem:eitNotLinearDependendOne}
	Let \(\rho\colon \R\to \R\) be a smooth, non-negative compactly supported function with \(\rho\not\equiv 0\). Given \(C_1,C_2,\varepsilon>0\) there exist \(\delta>0\) such that
	\[\left|\int_{\R}e^{i\lambda t}e^{-t\tau} \rho(t)dt\right|\leq (1-\delta)\int_{\R} \rho(t)dt\]
	for all \(\varepsilon \leq \lambda \leq C_1\) and all \(0\leq \tau \leq C_2\).
\end{lemma}
\begin{proof}
	By the Cauchy-Schwarz inequality, \(e^{-\tau t}\leq 1\) and since \(t\mapsto e^{i\lambda t}\) and \( t\mapsto 1\) are linearly independent for \(\lambda>0\) we have
	\[\left|\int_{\R} e^{it\lambda}e^{-t\tau}\rho(t)dt\right|<\left|\int_{\R}\rho(t)dt\right|=:C_0\] for all \(\lambda>0\) and \(\tau\geq 0\).
	Hence we have for the continuous function \(f(\lambda,\tau):=C_0^{-1}|\int_{R} e^{it\lambda}e^{-t\tau}\rho(t)dt|\) that \(f(\lambda,\tau)<1\) for all \(\lambda>0\) and \(\tau\geq 0\). Since  \([\varepsilon,C_1]\times[0,C_2]\) is compact we find \(\delta>0\) such that \(f(\lambda,\tau)\leq 1-\delta\) for all \((\lambda,\tau)\in[\varepsilon,C_1]\times[0,C_2]\). The claim follows.
\end{proof}
\begin{proof}[\textbf{Proof of Lemma~\ref{lem:hFksmallclose}}]
	From Lemma~\ref{lem:x2n1y2n1boundedByRe} it follows that after shrinking \(D\) there exists a constant \(C_0>0\) such that
	\begin{eqnarray}\label{eq:EstimateRePartPhi}
		|x-y|\leq C_0(|\text{Re}\varphi(x,y)|+|x'-y'|)
	\end{eqnarray}
	and
	\begin{eqnarray}\label{eq:EstimateImPartPhi}
		\Im\varphi(x,y)\leq C_0(|\operatorname{Re}\varphi(x,y)|^2+|x'-y'|^2)
	\end{eqnarray}
	for all \(x,y\in D\).
	For \(C,k>0\) put
	\[U_k:=\left\{(x,y)\in D\times D\mid |x'-y'|\leq \frac{C}{\sqrt{k}} \text{ and } |\text{Re}\varphi(x,y)|\leq\frac{C}{k}\right\} \]
	Taking \(C,k_0>0\) large enough it follows from Lemma~\ref{lem:hFksmallsemiclose} that we only need to consider \((x,y)\in U_k\) for \(k\geq k_0\).
	We conclude that there exists \(C_1>0\) such that
	\[|x-y|\leq \frac{C_1}{\sqrt{k}}\] for all \((x,y)\in U_k\)  and all \(k\geq k_0\). Since \(|A_0(x,y,t)-A_0(y,y,t)|\leq C_2|x-y|\) for some constant \(C_2>0\) we find that there exists \(C_3>0\) such that 
	\[|A(x,y,t,k)-A_0(y,y,t)|\leq C_3k^{n+\frac{1}{2}}\]
	or all \((x,y)\in U_k\)  and all \(k\geq k_0\). Let \(\varepsilon>0\) be arbitrary. 
	Assuming \((x,y)\in U_k\) with \(|x'-y'|\geq \frac{\varepsilon}{\sqrt{k}}\) we find
	\[k^{-n-1}|\tilde{S}_k(x,y)|\leq e^{-\delta_1c\varepsilon}\int_{\delta_1}^{\delta_2}|A_0(y,y,t)|dt+ \frac{C_4}{\sqrt{k}}\]
	for all \(k\geq k_0\) where \(C_4>0\) is a constant independent of \(x,y,k\).
	With \(A_0(y,y,t)\geq 0\) we recall that by Theorem~\ref{thm:ExpansionMain} there exists \(C_5>0\) such that 
	\[\left|k^{-1-n}B_k(y)-\int_{\delta_1}^{\delta_2}|A_0(y,y,t)|dt\right|\leq C_5k^{-1}\] 
	for all sufficiently large \(k\) and all \(y\in D\).
	Then from~\eqref{eq:hFkcloseToSzegoLocal} we conclude that there exist \(C_6,k_1>0\) such that
	\begin{eqnarray}\label{eq:hFkclosesmall1}
		h^{F_k}(x,y)\leq e^{-2\delta_1 c\varepsilon}+\frac{C_6}{\sqrt{k}}	
	\end{eqnarray}
	for all \((x,y)\in U_k\) with \(|x'-y'|\geq \frac{\varepsilon}{\sqrt{k}}\) and all \(k\geq k_1\).
	Now assume \(|\text{Re}(\varphi(x,y))|\geq\frac{\varepsilon}{k}\). 
	We find that there exists a constant \(C_7>0\) such that 
	\[k^{-1-n}|\tilde{S}_k(x,y)|\leq \left|\int_{\delta_1}^{\delta_2}e^{itk\text{Re}\varphi(x,y)} e^{-tk\text{Im}\varphi(x,y)}A_0(y,y,t)\right|+\frac{C_7}{\sqrt{k}}\]
	for all \((x,y)\in U_k\) and all sufficiently large \(k\). We have \(A_0(y,y,t)=a(y)|\chi(t)|^2\) for some positive function \(a\). Since \(D\) was chosen small enough by~\eqref{eq:EstimateRePartPhi} and~\eqref{eq:EstimateImPartPhi} we can choose a constant \(\tilde{C}>0\) such that  \( \varepsilon\leq k|\text{Re}\varphi(x,y)| \leq C\) and \(0\leq k\text{Im}(x,y)\leq \tilde{C}\) for all \(k\geq 1\) and all \((x,y)\in U_k\) with  \(|\text{Re}\varphi(x,y)|\geq\frac{\varepsilon}{k}\). We obtain from Lemma~\ref{lem:eitNotLinearDependendOne} with \(\lambda=k|\text{Re}\varphi(x,y)|\), \(\tau =k\text{Im}\varphi(x,y)\) and \(\rho(t)=|\chi(t)|^2\) that there exists \(\delta>0\) such that
	\[k^{-1-n}|\tilde{S}_k(x,y)|\leq (1-\delta)\int_{\delta_1}^{\delta_2}A_0(y,y,t)+\frac{C_7}{\sqrt{k}}\]
	for all \((x,y)\in U_k\) with \(|\Re\varphi(x,y)|\geq\frac{\varepsilon}{k}\)  and all sufficiently large \(k\).  With~\eqref{eq:hFkcloseToSzegoLocal} we conclude that there are \(C_8,k_2>0\) such that
	\begin{eqnarray}\label{eq:hFkclosesmall2}
		h^{F_k}(x,y)\leq (1-\delta)^2 + \frac{C_{8}}{\sqrt{k}}.
	\end{eqnarray}
	for all \((x,y)\in U_k\) with \(|\Re\varphi(x,y)|\geq\frac{\varepsilon}{k}\) and all \(k\geq k_2\).
	Hence taking \(k_0\) large enough the claim follows from~\eqref{eq:hFkclosesmall1} and~\eqref{eq:hFkclosesmall2}.
\end{proof}
\begin{lemma}\label{lem:hFksmallveryclose}
	There exist an open neighborhood \(U\subset D\) around zero and \(\varepsilon,\delta,k_0>0\) such that 
	\[h^{F_k}(x,y)\leq 1- \delta(k^2(\Re\varphi(x,y))^2+k|x'-y'|^2)\] for all \(x,y\in U\) with \(|x'-y'|\leq \frac{\varepsilon}{\sqrt{k}}\) and \(|\Re\varphi(x,y)|\leq \frac{\varepsilon}{k}\) and all \(k\geq k_0\). 
\end{lemma}
\begin{proof}
	For \(y\in D\) fixed consider the map defined by \(Q_y\colon D\to \R^{2n+1}\) \(Q_y(x)=(x'-y',\text{Re}\varphi(x,y))\). Since \(\frac{\partial\text{Re}\varphi}{\partial x_{2n+1}}(0,0)\neq 0\) we find by shrinking \(D\) that \(Q_y\) defines a diffeomorphism on its image for all \(y\in D\). We find that there exist open neighborhoods \(U,V\subset \R^{2n+1}\) around zero such that \(U\subset Q_y(D)\) for all \(y\in V\). Then we introduce new coordinates \(\hat{x}=Q_y(x)\). For this coordinates we find \(\varphi\circ Q_y^{-1}(\hat{x})=\hat{x}_{2n+1}+i\psi_y(\tilde{x})\), \(\hat{x}\in U\).  For some smooth real function \(\psi_y\) with \(\psi_y\geq 0\), \(\psi_y(0)=0\). It follows that there exists a constant \(C_1>0\) such that 
	\begin{eqnarray}
		\left|\frac{\partial \psi_y}{\partial \hat{x}_j}(\hat{x})\right| \leq C_1|\hat{x}|
	\end{eqnarray}	 
	for all \(1\leq j\leq 2n+1\) and all \(y\in V\), \(\hat{x}\in U\). Define \(L_k(\hat{x}):=(\sqrt{k}\hat{x}',k\hat{x}_{2n+1})\) and consider the coordinates \(\tilde{x}=L_k(\hat{x})\). Put \(\tilde{h}_{y,k}(\tilde{x}):=h^{F_k}(Q_y^{-1}L_k^{-1}(\tilde{x}),y)\). We are interested in the Taylor expansion of \(\tilde{h}_{y,k}\) at \(\tilde{x}=0\). Therefore, we first show that all its derivatives up to order three are uniformly bounded in \(k\).  Note that there exists \(C_2>0\) such that
	\begin{eqnarray}
		|\partial_{x}^{\alpha}A(x,y,t,k)|\leq C_2k^{n+1}
	\end{eqnarray}
	and
	\begin{eqnarray}
		|\partial_x^{\alpha}B_k(x)|\leq C_2k^{n+1}
	\end{eqnarray}
	for all \(x,y\in D\), \(k\geq1\) and \(|\alpha|\leq 3\). Then by the definition of \(h^{F_k}\) we find that the only way to increase the growth in \(k\) of its derivatives are derivatives acting on \(e^{ikt\varphi}\).
	Given \(1\leq j\leq 2n\) we have
	\[\left|\frac{\partial \varphi\circ Q_y^{-1}}{\partial \hat{x_j}}(\hat{x})\right|=\left|\frac{\partial \psi_y}{\partial \hat{x}_j}(\hat{x})\right|\leq C_1|\hat{x}|\] and hence
	\[\left|e^{-ikt\varphi\circ (L_k\circ Q_y)^{-1}}\frac{\partial }{\partial \tilde{x}_j}e^{ikt \varphi\circ (L_k\circ Q_y)^{-1}}(\hat{x})\right|=\frac{tk}{\sqrt{k}}\left|\frac{\partial \psi_y}{\partial \hat{x_j}}(L_k^{-1}(\tilde{x}))\right|\leq \sqrt{k}tC_1|L_k^{-1}(\tilde{x})|\leq C_1\delta_2|\tilde{x}| \] 
	for all \(y\in V\) and \(k\geq 1\).
	Furthermore, we have
	\[\left|e^{-ikt\varphi\circ (L_k\circ Q_y)^{-1}}\frac{\partial }{\partial \tilde{x}_{2n+1}}e^{ikt \varphi\circ (L_k\circ Q_y)^{-1}}(\hat{x})\right|=\frac{tk}{k}\left|1+\frac{\partial \psi_y}{\partial \hat{x_{2n+1}}}(L_k^{-1}(\tilde{x}))\right|\leq \delta_2(1+C_1|L_k^{-1}(\tilde{x})|).\]
	It follows that there exist \(C_3,k_0>0\) and a ball \(\tilde{B}\subset \R^{2n+1}\) centered in zero such that
	\[|\partial^{\alpha}_{\tilde{x}}\tilde{h}_{k,y}(\tilde{x})|\leq C_3\]
	for all \(y\in V\), \(k\geq k_0\), \(|\alpha|\leq 3\) and \(\tilde{x}\in \tilde{B}\).
	Since \(\tilde{h}_{k,y}\) has a maximum at \(\tilde{x}=0\), Taylor expansion of \(\tilde{h}_{k,y}\) at \(\tilde{x}=0\) yields
	\[\tilde{h}_{k,y}(\tilde{x})=1+\tilde{x}^TM_{k,y}\tilde{x}+R_{k,y}(\tilde{x})\]
	where  \(M_{k,y}\) is a symmetric \((2n+1) \times (2n+1)\) matrix and we have 	\(|R_{k,y}(\tilde{x})|\leq C_4|\tilde{x}|^3\) for some constant \(C_4>0\) independent of \(k\geq k_0\), \(y\in V\) and \(\tilde{x}\in \tilde{B}\).
	Using the definition of \(\tilde{h}_{k,y}\) and the fact that \(S_k\) has an asymptotic expansion we further obtain that there exists \(C_5>0\) and  for any \(\alpha\in \N^{2n+1}\), \(|\alpha|\leq 3\), a smooth function \(r_{\alpha}\colon U\times \tilde{B}\to \R\) such that 
	\[\left|\partial^{\alpha}_{\tilde{x}}\tilde{h}_{k,y}(\tilde{x})-r_{\alpha}(\tilde{x},y)\right|\leq \frac{C_5}{\sqrt{k}}.\]
	Hence we can write \(M_{k,y}=M_{y}+\tilde{M}_{k,y}\) and there exists a constant \(C_6>0\) such that \(|\tilde{x}^T\tilde{M}_{k,y}\tilde{x}|\leq \frac{C_6}{\sqrt{k}}|\tilde{x}|^2\) for all \(y\in V\), \(k\geq k_0\) and  \(\tilde{x}\in\tilde{B}\).  We will now show that \(M_y\) is negative definite when \(y\) is sufficiently close to zero. Therefore we note that
	\[dQ_y=\begin{pmatrix}
		1 & 0 & \ldots & 0\\
		0 & 1 & \ldots & 0\\
		\dots & \dots & \dots & \dots\\
		
		\frac{\partial \text{Re}(\varphi)}{\partial x_1} & \frac{\partial \text{Re}(\varphi)}{\partial x_2} & \ldots & \frac{\partial \text{Re}(\varphi)}{\partial x_{2n+1}}
	\end{pmatrix} \text{ and } dQ^{-1}_y=\begin{pmatrix}
		1 & 0 & \ldots & 0\\
		0 & 1 & \ldots & 0\\
		\dots & \dots & \dots & \dots\\
		
		-\frac{\frac{\partial \text{Re}(\varphi)}{\partial x_1}}{\frac{\partial \text{Re}(\varphi)}{\partial x_{2n+1}}} & -\frac{\frac{\partial \text{Re}(\varphi)}{\partial x_2}}{\frac{\partial \text{Re}(\varphi)}{\partial x_{2n+1}}} & \ldots & \frac{1}{\frac{\partial \text{Re}(\varphi)}{\partial x_{2n+1}}}.
	\end{pmatrix} \] 
	Hence we find
	\[\frac{\partial}{\partial \hat{x}_j}=\frac{\partial}{\partial x_j} -\left(\frac{\partial \text{Re}(\varphi)}{\partial x_{2n+1}}\right)^{-1}\frac{\partial \text{Re}(\varphi)}{\partial x_1}\frac{\partial}{\partial x_{2n+1}} \]
	for \(1\leq j\leq 2n\) and
	\[\frac{\partial}{\partial \hat{x}_{2n+1}}=\left(\frac{\partial \text{Re}(\varphi)}{\partial x_{2n+1}}\right)^{-1}\frac{\partial}{\partial x_{2n+1}}.\]
	By the choice of coordinates \(x,y\) (see Theorem~\ref{tangential hessian of varphi}) we have \(\mathcal{T}=\frac{\partial}{\partial x_{2n+1}}\) and at \(x=y=0\)  we have that \(\{\frac{\partial}{\partial x_j}\}_{1\leq j\leq 2n}\) is a basis of \(\Re (T^{1,0}_pX)\). Furthermore, we have \(\frac{\partial \Re\varphi}{\partial x_j}(0,0)=0\) for all \(1\leq j\leq 2n\). It follows that for \(y=0\) and \(\hat{x}=0\) we have that \(\{\frac{\partial}{\partial \hat{x}_j}\}_{1\leq j\leq 2n}\) is a basis of \(\Re(T^{1,0}_pX)\) and \(\frac{\partial}{\partial \hat{x}_{2n+1}}\) is a non-zero multiple of \(\mathcal{T}\). Let \(\hat{M}_k\) be the matrix representation of \(\mathcal{H}^{F_k}_p\) in the coordinates \(\hat{x}\). By Lemma~\ref{lem:HFisHessianofg} we have that  \(\mathcal{H}^{F_k}_p\) is the Hessian of the function \(g_p\) defined by \(g_p(q):=h^{F_k}(p,q)\). From the definition of \(\tilde{h}_{k,y}\) we observe that \(M_{0,k}=L_k^*(\hat{M}_k)\). Then it follows from Lemma~\ref{lem:PropertiesOfgkForX}  that  \(M_0\) is negative definite. We conclude that by shrinking \(V\) we can achieve that \(M_y\) is negative definite for all \(y\in V\). Then there exist \(\delta,k_2,\varepsilon>0\) such that
	\[\tilde{h}_{k,y}(\tilde{x})\leq 1-\delta|\tilde{x}|^2\]
	for all \(k\geq k_2\), all \(y\in V\) and all \(\tilde{x}\in \tilde{B}\) with \(\|\tilde{x}\|\leq\varepsilon\). The claim follows from
	\[h^{F_k}(x,y)=\tilde{h}_{k,y}(L_k\circ Q_y(x))\leq 1-\delta(k^2(\Re\varphi(x,y))^2+k|x'-y'|^2).\]
\end{proof}
Now we are ready to prove Lemma~\ref{lem:PropertiesOfhkForX}.
\begin{proof}[\textbf{Proof of Lemma~\ref{lem:PropertiesOfhkForX}}]
	By the definition of \(h^{F_k}\) it follows that \(h^{F_k}(x,x)=1\). So we need to show that for \(k\) large enough we have that \(h^{F_k}(x,y)=1\) implies \(x=y\).
	We first observe that given an open neighborhood \(W\) around the diagonal in \(X\times X\) we find by Theorem~\ref{thm:ExpansionMain} and~\eqref{eq:hFkcloseToSzego} that there exists \(k_1>0\) such that \(h^{F_k}(x,y)\leq \frac{1}{2}\) for all \((x,y)\not\in W\). From this observation and the compactness of \(X\) we need to show that given any point \(p\in X\) we can find an open neighborhood \(D\times D\) of \((p,p)\) in \(X\times X\) and \(k_0>0\) such that for all \(x,y\in D\) and all \(k\geq k_0\) we have that \(h^{F_k}(x,y)=1\) implies \(x=y\). Let \(p\in X\) be a point. We choose coordinates \((x,y)=(x',x_{2n+1},y',y_{2n+1})\) around \((p,p)\) as above. Recall that in these coordinates we have \(\eta_k(T_P)(x,y)=:S_k(x,y)=\int_{\delta_1}^{\delta_2}e^{ikt\varphi(x,y)}A(x,y,t,k) dt+O(k^{-\infty})\). We can apply Lemma~\ref{lem:x2n1y2n1boundedByRe}, Lemma~\ref{lem:hFksmallclose} and Lemma~\ref{lem:hFksmallveryclose}  to achieve after shrinking \(D\) that there exists \(\varepsilon,\delta,k_2,C,c>0\) such that for all \(x,y\in D\) and all \(k\geq k_2\) we have that \(|x-y|\leq C(|\text{Re}\varphi(x,y)|+|x'-y'|)\) and
	\[h^{F_k}(x,y)\leq \begin{cases}
		1-\delta,& \text{ if } |\text{Re}\varphi(x,y)|\geq \frac{\varepsilon}{k} \text{ or } |x'-y'|\geq \frac{\varepsilon}{\sqrt{k}},\\
		1-c(k^2(\text{Re}\varphi(x,y))^2+k|x'-y'|^2),& \text{ else.}
	\end{cases}\]
	Hence given \(k\geq k_2\) and \(x,y\in D\) with \(h^{F_k}(x,y)=1\) we find \(|x'-y'|=|\text{Re}\varphi(x,y)|=0\) which leads to \(|x-y|=0\) that is \(x=y\).
\end{proof}

\subsection{Integral Estimates for \(h^{F_k}\) and \(\chi_k(T_P)\)}\label{sec:IntegralEstimatesHFk}
We end this section by proving two integral estimates involving \(S_k\) and \(h^{F_k}\) which will be useful for later applications.
\begin{lemma}\label{lem:EstimateIntegralhk}
	Let \(dV\) be a volume form on \(X\times X\). We have that \(\left(\sqrt{1-h^{F_k}}\right)^{-1}\) is integrable for all sufficiently large \(k\) and there exist \(k_0,\delta,C>0\) such that
	\begin{eqnarray}\label{eq:IntegralhFkOnXxX}
		\int_{X\times X}\frac{|\psi|}{\sqrt{1-h^{F_k}}}dV\leq C\left(k^{-1-n}\operatorname{sup}|\psi|+ \int_{h^{F_k}\leq 1-\delta}|\psi|dV\right)
	\end{eqnarray}
	holds for all \(k\geq k_0\) and all bounded measurable functions \(\psi\colon  X\times X \to \C\).
\end{lemma}
\begin{proof}
    We note that since \(X\times X\) is compact we have for any  bounded measurable function \(\psi\colon  X\times X \to \C\) that \(\psi\in L^1(X\times X)\).
	First, we observe that given an open neighborhood \(W\) around the diagonal in \(X\times X\) we find by Theorem~\ref{thm:ExpansionMain} that there exists \(k_1>0\) such that \(h^{F_k}(x,y)\leq \frac{1}{2}\) for all \(k\geq k_1\) and all \((x,y)\not\in W\). Hence \[\int_{X\times X}\frac{|\psi|}{\sqrt{1-h^{F_k}}}dV\leq \sqrt{2}\int_{h^{F_k}\leq \frac{1}{2}}|\psi|dV\]
	for all \(k\geq k_1\) and all \(\psi\in L^{1}(X\times X)\) with \(\text{supp}(\psi)\cap W=\emptyset\). From this observation and the compactness of \(X\) we just need to show that given any point \(p\in X\) we can find an open neighborhood \(D\times D\) around \((p,p)\) and \(k_0,\delta,C>0\) such that~\eqref{eq:IntegralhFkOnXxX} holds for all \(k\geq k_0\) and all bounded,  measurable functions \(\psi\colon  X\times X \to \C\) with \(\operatorname{supp}(\psi)\subset D\times D\). 
	Let \(p\in X\) be a point. We choose coordinates \((x,y)=(x',x_{2n+1},y',y_{2n+1})\) around \((p,p)\) as above. Denote the Lebesgue measure on \(\R^{2n+1}\) by \(dx\) or \(dy\) respectively. Recall that in these coordinates we have \(\eta_k(T_P)(x,y)=S_k(x,y)=\int_{\delta_1}^{\delta_2}e^{ikt\varphi(x,y)}A(x,y,t,k) dt + O(k^{-\infty})\). As in the proof of Lemma~\ref{lem:hFksmallveryclose}, after shrinking \(D\), we can find open subsets \(V,U\subset\subset D\) around zero such that for any \(y\in V\)  we can introduce new coordinates \(\hat{x}=Q_y(x)\) by \(\hat{x}'=x'-y'\) and \(\hat{x}_{2n+1}=\text{Re}\varphi(x,y)\) with \(U\subset \{(x'-y',\text{Re}\varphi(x,y))\mid x\in D\}\). For \(\varepsilon>0\) and \(k>0\) put
	\[U_k:=\{\hat{x}\in U\mid k^2|\hat{x}_{2n+1}|^2+k|\hat{x}'|^2\geq \varepsilon^2\}.\] We can apply Lemma~\ref{lem:hFksmallclose} and  Lemma~\ref{lem:hFksmallveryclose} to achieve after shrinking \(V\) and choosing \(\varepsilon>0\) sufficiently small that there exist \(\delta,k_2,C,c>0\) such that for \(y\in V\) and \(\hat{x}\in U\) we have 
	\[h^{F_k}(Q_y^{-1}(\hat{x}),y)\leq \begin{cases}
		1-\delta,& \text{ if } \hat{x}\in U_k\\
		1-c(k^2(\hat{x}_{2n+1})^2+k|\hat{x}'|^2),& \text{ else.}
	\end{cases}\]
	Then 
	\[\int_{{U_k}}\frac{|\psi(\hat{x},y)|}{\sqrt{1-h^{F_k}(Q^{-1}_y(\hat{x}),y)}}d\hat{x}\leq \frac{1}{\sqrt{\delta}}\int_{\{\hat{h}_{k,y}\leq 1-\delta\}} |\psi(\hat{x},y)|d\hat{x}\]
	for all \(k\geq k_2\), \(y\in V\) and \(\psi\in L^{1}(U\times V)\) where \(\hat{h}_{k,y}(\hat{x}):=h^{F_k}(Q_y^{-1}(\hat{x}),y)\).
	Using coordinates \(\tilde{x}\) by \(\tilde{x}'=\sqrt{k} \hat{x}'\) and \(\tilde{x}_{2n+1}=k\hat{x}_{2n+1}\) we find since \(d\hat{x}=k^{-1-n}d\tilde{x}\) that
	\[\int_{U\setminus U_k}\frac{|\psi(\hat{x},y)|}{\sqrt{1-h^{F_k}(Q^{-1}_y(\hat{x}),y)}}d\hat{x}\leq (k)^{-1-n}\frac{\tilde{C}}{\sqrt{c}}
	\sup|\psi|.\]	 
	for all \(k\geq k_2\), \(y\in V\) and all bounded measurable functions \(\psi\colon U\times V\to \C\) where  \(\tilde{C}:=\int_{|\tilde{x}|<\varepsilon}\frac{d\tilde{x}}{|\tilde{x}|}<\infty\). We conclude that there exists a constant \(C_0>0\) such that
	\begin{eqnarray}\label{eq:integralEstimatehfkLocal}
		\,\,\,\,\,\,\,\,\,\,\,\,\,\,\,\,\,\,\,\,\, \int_{U\times V}\frac{|\psi(\hat{x},y)|}{\sqrt{1-h^{F_k}(Q^{-1}_y(\hat{x}),y)}}d\hat{x}\otimes dy\leq C_0\left(k^{-1-n}\sup|\psi|+ \int_V\left(\int_{\hat{h}_{k,y}\leq 1-\delta} |\psi(\hat{x},y)|d\hat{x}\right) dy\right)
	\end{eqnarray}
	for all \(k\geq k_2\) and all bounded measurable functions \(\psi\colon U\times V\to \C\). We can find an open neighborhood \(\tilde{D}\subset V\) around zero such that \(\tilde{D}\subset Q_y^{-1}(U)\) for all \(y\in \tilde{D}\). Furthermore, given \(x,y\in\tilde{D}\) with \(\hat{h}_{k,y}(Q_y(x))\leq 1-\delta\) we find \(h^{F_k}(x,y)\leq 1-\delta\). Since \(\tilde{D}\times \tilde{D}\) is identified with an open neighborhood of \(X\times X\) around \((p,p)\) it follows from~\eqref{eq:integralEstimatehfkLocal} that there exist \(C,k_0>0\) such that
	\[	\int_{X\times X}\frac{|\psi|}{\sqrt{1-h^{F_k}}}dV\leq C\left(k^{-1-n}\sup|\psi|+ \int_{h^{F_k}\leq 1-\delta}|\psi|dV\right)\]
	for all \(k\geq k_0\) and all bounded measurable functions \(\psi\colon X\times X\to \C\) with \(\operatorname{supp}(\psi)\subset\tilde{D}\times \tilde{D}\). 
\end{proof}
\begin{lemma}\label{lem:IntegralSzegoDerivativeXtimesX}
	Let  \(L\colon \mathscr{C}^\infty(X\times X)\to \mathscr{C}^\infty(X\times X) \) be a partial differential operator of order less or equal than \(d\geq 0\) and let \(dV\) be a volume form on \(X\times X\). Put \(S_k(x,y):=\eta_k(T_P)(x,y)\). Given \(0<\varepsilon<1\) there exist \(C,k_0>0\) such that
	\[\int_{X\times X}|LS_k|dV\leq Ck^{d+\varepsilon}\]
	for all \(k\geq k_0\).
\end{lemma}

\begin{proof}
	We first observe that given an open neighborhood \(W\) around the diagonal in \(X\times X\) and \(N\in \N\) we find by Theorem~\ref{thm:ExpansionMain} that there exists \(k_1>0\) such that \(|LS_k|\leq k^{-N}\) for all \(k\geq k_1\) and all \((x,y)\not\in W\). Hence \[\int_{X\times X\setminus W}|LS_k|dV\leq k^{-N}\int_{X\times X}dV\]
	for all \(k\geq k_1\). From this observation and the compactness of \(X\) we just need to show that given any point \(p\in X\) we can find an open neighborhood \(D\times D\) around \((p,p)\) and \(k_0,C>0\) such that
	\[\int_{D\times D}|LS_k|dV\leq Ck^{d+\varepsilon}\]
	holds for all \(k\geq k_0\). 
	Let \(p\in X\) be a point. We choose coordinates \((x,y)=(x',x_{2n+1},y',y_{2n+1})\) on \(D\times D\) around \((p,p)\) as above. Denote the Lebesgue measure on \(\R^{2n+1}\) by \(dx\) or \(dy\) respectively. Recall that in these coordinates we have \(S_k(x,y)=\int_{\delta_1}^{\delta_2}e^{ikt\varphi(x,y)}A(x,y,t,k) dt+O(k^{-\infty})\). It follows immediately that \[(LS_k)(x,y)=\int_{\delta_1}^{\delta_2}e^{ikt\varphi(x,y)}A^L(x,y,t,k) dt+O(k^{-\infty})\]
	where \(A^L(x,y,t,k)\sim \sum_{j=0}^{\infty}A^L_{j}(x,y,t)k^{n+1+d-j}\) in \(S^{n+1}_{\mathrm{loc}}(1;D\times D\times(\delta_1,\delta_2))\) has an asymptotic expansion. Hence for any \(m\geq 0\) we find \(\hat{C}_m>0\) such that 
	\[\left|\left(\frac{\partial}{\partial t}\right)^mA^L(x,y,k,t)\right|\leq \hat{C}_mk^{d+n+1}.\]
	Put
	\[S^L_k(x,y):=\int_{\delta_1}^{\delta_2}e^{ikt\varphi(x,y)}A^L(x,y,t,k) dt\]
	for \(x,y\in D\).
	 As in the proof of Lemma~\ref{lem:hFksmallveryclose}, after shrinking \(D\), we can find open subsets \(V,U\subset D\) around zero such that for any \(y\in V\)  we can introduce new coordinates \(\hat{x}=Q_y(x)\) by \(\hat{x}'=x'-y'\) and \(\hat{x}_{2n+1}=\text{Re}\varphi(x,y)\) with \(U\subset \{(x'-y',\text{Re}\varphi(x,y))\mid x\in D\}\). Furthermore, we can achieve that \(U\) is of the form \(U=(-r,r)\times U'\) for some \(r>0\). Choose \(k_1>1\) with \(k_1^{\varepsilon-1}< r\).
	For \(k>k_1\) put \(U_k:=(- k^{\varepsilon-1},k^{\varepsilon-1})\times U'\). Recall that \(\text{Im}\varphi(x,y)\geq c|x'-y'|\) for some \(c>0\).
	Since \(|S^L_k(x,y)|\leq k^{n+1+d}\hat{C}_0e^{-\delta_1c |x'-y'|}\)
	it follows that
	\begin{eqnarray}
		\int_{U_k}|S^L_k(Q_y^{-1}(\hat{x}),y)| d\hat{x}\leq \hat{C}_0k^{d+n+1}2k^{\varepsilon-1}\int_{U'}e^{-c\delta_1k|\hat{x}'|^2}d\hat{x}'.
	\end{eqnarray}
	With \(\int_{R^{2n}}e^{-c\delta_1k|x|^2}dx=\pi^n(c\delta_1k)^{-n}\) it follows that there exist \(C_1,k_2>0\) such that
	\[\int_{U_k}|LS_k(Q_y^{-1}(\hat{x},y))| d\hat{x}\leq C_1k^{d+\varepsilon}\]
	for all \(y\in V\) and \(k\geq k_2\).
	Recall that for all \(x,y\in D\) and \(k\) we have that \(\text{supp}(A^L(x,y,\cdot,k))\subset (\delta_1,\delta_2)\).
	Then for \(\text{Re}\varphi(x,y)\neq 0\) we find
	\[|(S^L_k)(x,y)|=\left|\frac{1}{(ik\varphi(x,y))^m}\int_{\delta_1}^{\delta_2}e^{ikt\varphi(x,y)}\left(\frac{\partial}{\partial t}\right)^mA^L(x,y,t,k) dt\right|\leq \frac{\hat{C}_m}{(k|\text{Re}\varphi(x,y)|)^m}\]
	We obtain \(|(S^L_k)(Q_y^{-1}(\hat{x}),y)|\leq \hat{C}_mk^{-m\varepsilon} \)
	for all \(y\in V\), \(\hat{x}\in U\setminus U_k\), \(k\geq k_1\). Choosing \(m\geq \frac{n+1}{\varepsilon}\) we conclude
	\[\int_{U\setminus U_k}|S^L_k(Q_y^{-1}(\hat{x}),y)| d\hat{x}\leq \hat{C}_mk^{d}2r\int_{U'}d\hat{x}'\]
	for all \(y\in V\) and all \(k\geq k_1\). It follows that there exist \(k_3,C_2>0\) such that
	\[\int_{V\times U}|LS_k(Q_y^{-1}(\hat{x}),y)| d\hat{x}\otimes dy\leq C_2k^{d+\varepsilon}\]
	for all \(k\geq k_3\). Since we can find an open neighborhood \(\tilde{D}\subset V\) around zero such that \(\tilde{D}\subset Q_y^{-1}(U)\) for all \(y\in \tilde{D}\) and since \(\tilde{D}\times \tilde{D}\) is identified with an open neighborhood of \(X\times X\) around \((p,p)\) we have that there exist \(C,k_0>0\) such that
	\[\int_{\tilde{D}\times \tilde{D}}|LS_k| dV\leq Ck^{d+\varepsilon}\]
	for all \(k\geq k_0\).
\end{proof}

\section{Lelong-Poincar\'e Formula for Domains with Boundary}\label{subsec:PoincareLelong}
Let \(Y\) be a complex manifold of dimension \(\dim_\C Y=n+1\). Let \(G\subset Y\) be a domain with smooth boundary \(bG\) and denote by \(\iota\colon bG\to Y\) the inclusion map. We have that \(bG\) is a codimension one CR manifold with CR structure \(T^{1,0}bG:=\C TbG\cap T^{1,0}Y\) where \(T^{1,0}Y\) denotes the complex structure of \(Y\). By \(\Omega_c^{q}(\overline{G})\) and \(\Omega_c^{p,q}(\overline{G})\) we denote the space of smooth \(p\)- and \((p,q)\)-forms with compact support in \(\overline{G}\). In particular, the elements in  \(\Omega_c^{q}(\overline{G})\) and \(\Omega_c^{p,q}(\overline{G})\) are restrictions of the elements in \(\Omega_c^{q}(Y)\) and \(\Omega_c^{p,q}(Y)\) to \(\overline{G}\). The aim of this section is to prove a version of the Lelong-Poincar\'e formula for  test forms in \(\Omega_c^{p,q}(\overline{G})\) and holomorphic functions on \(G\) which are smooth up to the boundary \(bG\) (see Theorem~\ref{thm:PoincareLelongFormula}). More precisely, we want to show that for any holomorphic function  \(f\in\mathcal{O}(G)\cap \mathscr{C}^\infty(\overline{G})\) which is regular with respect to \(bG\) (see Definition~\ref{def:RegularWithRespectToBoundary}) and any \(\psi\in \Omega_c^{n,n}(\overline{G})\) we have
\begin{eqnarray}\label{eq:PoincareLelongFormulaIntro}
	\left( \divisor{f}, \psi\right)=\frac{i}{\pi}\int_G \partial\overline{\partial}\log(|f|)\wedge \psi
\end{eqnarray}
where \(\divisor{f}\) is the zero divisor of \(f\). We note that this formula is well-known for \(\psi\in \Omega_c^{n,n}(G)\) where the right-hand side is defined by \(\frac{i}{\pi}\int_G \log(|f|)\wedge \partial\overline{\partial}\psi\). When \(\psi\) does not vanish along \(bG\) the definition of the left-hand side requires the understanding of the zero set of \(f\) near and on the boundary (see Remark~\ref{rmk:ZeroDivisorWellDefined}). Furthermore, the definition of the right-hand side will involve integrals on the boundary of \(G\) (see Definition~\ref{def:RegularWithRespectToBoundary}). In the end of this section we will prove two results (see Theorem~\ref{thm:OneOverFIntegrableCR} and Theorem~\ref{thm:PoincareLelongForCR}) regarding the zero sets of CR functions on strictly pseudoconvex CR manifolds.

\subsection{Existence of Some Integrals}
We start with the following definition.
\begin{definition}\label{def:RegularWithRespectToBoundary}
	Let \(f\in\mathcal{O}(G)\cap \mathscr{C}^\infty(\overline{G})\) be a holomorphic function on \(G\) which is smooth up to the boundary. Given \(p\in bG\) with \(f(p)=0\) we say that \(f\) is regular at \(p\) with respect to \(bG\) if \(\partial f_p\neq 0\) and at least one of the following conditions is satisfied
	\begin{itemize}
		\item[(i)] \(\ker \partial f_p \cap T^{1,0}_pY + T^{1,0}_pbG= T^{1,0}_p Y\),
		\item [(ii)]  The Levi form of \(bG\) at \(p\) does not vanish.  
	\end{itemize} 
	We say that  \(f\) is regular with respect to \(bG\) if \(f\) is regular at \(p\) with respect to \(bG\) at any point \(p\in bG\) with \(f(p)=0\).
\end{definition}
We note that the definition above does not dependent on the choice of Levi form at \(p\in bG\).
\begin{example}\label{ex:NonRegularAtBoundary}
	Put \(Y=\C^{2}\) and \(G=\{z\in Y\mid \text{Re}(z_1)< 0\}\). Then \(G\) is a domain with smooth boundary.	Let \(f\in\mathcal{O}(G)\cap \mathscr{C}^\infty(\overline{G})\) be the holomorphic function defined by \(f(z)=z_1\). It follows that \(f\) is not regular with respect to \(bG\).
\end{example}
\begin{example}\label{ex:RegularAtBoundary}
	Put \(Y=\C^{2}\) and \(G=\{z\in Y\mid |z_1|^2+|z_2|^2< 1\}\). Then \(G\) is a domain with smooth boundary.	Let \(f\in\mathcal{O}(G)\cap \mathscr{C}^\infty(\overline{G})\) be the holomorphic function defined by \(f(z)=z_1-1\). It follows that \(f\) is regular with respect to \(bG\).
\end{example}	
In the following, given \(f\in\mathcal{O}(G)\cap \mathscr{C}^\infty(\overline{G})\), we denote by \(\{f=0\}\) the zero set of \(f\) that is \(\{f=0\}:=\{p\in \overline{G}\mid f(p)=0\}\).
\begin{lemma}\label{lem:BoundaryIntegrableGlobal}
	With the notation above let \(f\in\mathcal{O}(G)\cap \mathscr{C}^\infty(\overline{G})\) be a holomorphic function on \(G\) which is smooth up to the boundary and \(p\in bG\) a point with \(f(p)=0\) such that \(f\) is regular at \(p\) with respect to \(bG\). There exists an open neighborhood \(U\subset Y\) around \(p\) such that \(U\cap \{f=0\}\) has measure zero in \(G\) and \(U\cap\{f=0\}\cap bG\) has measure zero in \(bG\) and
	\begin{itemize}
		\item [(i)] \(\frac{1}{f}\partial f\wedge \psi\) is integrable on \(U\cap \overline{G}\) for all \(\psi\in \Omega_{c}^{n,n+1}(\overline{G}\cap U)\).
		\item [(ii)]  \(\log(|f|) \psi\) is integrable on \(U\cap \overline{G}\) for all \(\psi\in \Omega_{c}^{n+1,n+1}(\overline{G}\cap U)\).
		\item[(iii)] \( \iota^*(\frac{1}{f}\partial f\wedge \psi)\) is integrable on \(U\cap bG\) for all \(\psi\in \Omega_{c}^{2n}(\overline{G}\cap U)\).
		\item[(iv)]\(\iota^*(\log(|f|) \psi)\) is integrable on \(U\cap bG\) for all \(\psi\in \Omega_{c}^{2n+1}(\overline{G}\cap U)\).
		\item[(v)] \(\int_G \frac{1}{2f}\partial f\wedge \overline{\partial}\psi = \int_{bG} \iota^*(\log(|f|) \overline{\partial}\psi)-\int_{G}\log(|f|) \partial\overline{\partial}\psi\) for all \(\psi\in \Omega_{c}^{n,n}(\overline{G}\cap U)\).
	\end{itemize}
\end{lemma}
\begin{corollary}\label{cor:BoundaryIntegrableGlobal}
	Let \(f\in\mathcal{O}(G)\cap \mathscr{C}^\infty(\overline{G})\), \(f\not\equiv 0\), be regular with respect to \(bG\). We have that the zero set of \(f\) has measure zero in \(G\) and that the zero set of \(f|_{bG}\) has measure zero in \(bG\). Furthermore,  we have that \(\log(|f|)\) is locally integrable on \(\overline{G}\) and that \(\iota^*\frac{1}{f}\partial f\) and \(\iota^*\log(|f|)\) are locally integrable on \(bG\).
\end{corollary}
\begin{proof}
	We have to show that any point \(p\in \overline{G}\) has an open neighborhood  \(U\) such that the statement holds on \(U\cap \overline{G}\) and \(U\cap bG\) respectively.
	It is well known that \(\log|f|\) is locally integrable on \(G\). Let \(p\in bG\) a point in the boundary such that \(f(p)\neq 0\). Then the objects we consider are smooth in an open neighborhood around \(p\)  and hence locally integrable. For \(p\in bG\) with \(f(p)=0\) the statement follows from Lemma~\ref{lem:BoundaryIntegrableGlobal}. 
\end{proof}
\begin{remark}
	In Corollary~\ref{cor:BoundaryIntegrableGlobal} we mean by locally integrable that all integrals with respect to the pairing with forms in \(\Omega_c^{\bullet}(\overline{G})\) and \(\iota^*\Omega_c^{\bullet}(\overline{G})\) respectively exist and are finite. 
\end{remark}
The proof of Lemma~\ref{lem:BoundaryIntegrableGlobal} requires some preparation which we will do in Lemma~\ref{lem:logIntegrableReal}, Lemma~\ref{lem:EstimateIntegral1D} and Lemma~\ref{lem:BoundaryIntegrableLocal}.
\begin{lemma}\label{lem:logIntegrableReal}
	Let \(X\subset \R^n\), \(n\geq 2\), be a smooth orientable submanifold of dimension \(\dim X=n-1\). We denote by \(dV_X\) the volume form on \(X\) and by \(\iota\colon X\to \R^n\) the inclusion map. Define \(f\colon \R^n\to \R\cup\{\infty\}\), \(f(x)=-\log(x_1^2+x_2^2)\). We have \(\iota^*f\in L^1_{\operatorname{loc}}(X)\).
\end{lemma}
\begin{proof}
	Let \(p=(p_1,\ldots,p_n)\in X\) be a point.
	Given \(p_1^2+p_2^2>0\) we find an open neighborhood \(U\) in \(X\) around \(p\) such that \(\iota^*f|_{U}\in \mathscr{C}^\infty(U)\). Hence \(\iota^*f\) is locally integrable near \(p\).
	Now assume \(p_1=p_2=0\). Let \(\Pi_j\colon \R^n \to \R^{n-1}\) be the projection obtained by forgetting the \(j\)-th component. Consider the case that there exists \(j\geq 3\) such that the restriction of \(\Pi_j\) to an open neighborhood \(U\) in \(X\) around \(p\) is a diffeomorphism on its image. Without loss of generality assume \(j=n\geq 3\).  By shrinking \(U\) we can ensure that \(\Pi_j(U)\) is of the form \(\Pi_j(U)=D_\varepsilon\times V\) where \(D_\varepsilon =\{(x_1,x_2)\in\R^2 \mid x_1^2+x_2^2<\varepsilon^2\}\), \(0<\varepsilon<1\), and \(V\) is an open subset of \(\R^{n-3}\). By the transformation law we find \(\int_U \iota^*f dV_X=\int_{\Pi_j(U)}-\log(x_1^2+x_2^2)\rho(x_1,\ldots,x_{n-1})dV_{R^{n-1}}\) for some smooth positive function \(\rho\). Since \((x_1,x_2)\mapsto-\log(x_1^2+x_2^2)\) is positive and integrable on \(D_\varepsilon\) we conclude by the Fubini–Tonelli theorem that \(\iota^*f\) is integrable near \(p\). Now  consider the case that there exists \(j\leq 2\) such that the restriction of \(\Pi_j\) to an open neighborhood \(U\) in \(X\) around \(p\) is a diffeomorphism on its image.  Without loss of generality assume \(j=1\). By shrinking \(U\) we can ensure that \(\Pi_j(U)\) is of the form \(\Pi_j(U)=(-\varepsilon,\varepsilon)\times V\), \(0<\varepsilon<1\) and \(V\) is an open subset of \(\R^{n-2}\). Furthermore, we can assume that for all \(x\in U\) we have \(x_1^2+x_2^2<1\).  By the transformation law we find \(\int_U \iota^*f dV_X=\int_{\Pi_j(U)}-\log(g(x_2,\ldots,x_n)^2+x_2^2)\rho(x_2,\ldots,x_{n})dV_{R^{n-1}}\) for some smooth positive function \(\rho\) and some smooth function \(g\). We have \(-\log(g(x_2,\ldots,x_n)^2+x_2^2)\leq -2\log(|x_2|)\). Since \(x_2\mapsto -\log(|x_2|)\) is integrable on \((-\varepsilon,\varepsilon)\) and \(-\log(g(x_2,\ldots,x_n)^2+x_2^2)\) is positive we conclude by the Fubini–Tonelli theorem that \(\iota^*f\) is integrable near \(p\).\\
	We have shown that any point \(p\in X\) has an open neighborhood where \(\iota^*f\) is integrable hence it follows that \(\iota^*f\in L^1_{\text{loc}}(X)\).
\end{proof}
\begin{lemma}\label{lem:EstimateIntegral1D}
	For any \(0\leq c <1\) and \(-1<y<1\), \(y\neq 0\), we have
	\[I(c,y):=\int_0^1\left(\sqrt{(x^2-c)^2+y^2}\right)^{-1}dx\leq \left(\sqrt{|y|})\right)^{-1}(-\log|y|+4).\]
\end{lemma}
\begin{proof}
	With \(t=(x^2-c)/|y|\) we find
	\[I(c,y)=\int_{-c/|y|}^{(1-c)/|y|}\frac{dt}{2\sqrt{|y|t+c}\sqrt{t^2+1}}.\]
	We have
	\[\int_0^1\frac{dt}{2\sqrt{|y|t+c}\sqrt{t^2+1}}\leq\frac{1}{\sqrt{|y|}}\int_0^1 \frac{dt}{2\sqrt{t}}=\frac{1}{\sqrt{|y|}}.\]
	Given \(c+|y|<1\) that is \((1-c)/|y|>1\) we obtain
	\[\int_{1}^{(1-c)/|y|}\frac{dt}{2\sqrt{|y|t+c}\sqrt{t^2+1}}\leq\frac{1}{\sqrt{|y|+c}}\int_{1}^{(1-c)/|y|}\frac{dt}{|t|}\leq \frac{1}{2\sqrt{|y|}}\log\left(\frac{1-c}{|y|}\right)\leq -\frac{\log|y|}{2\sqrt{|y|}}.\] 
	Given \(c\leq 2|y|\) we have
	\[\int_{-c/|y|}^{0}\frac{dt}{2\sqrt{|y|t+c}\sqrt{t^2+1}}\leq \frac{1}{\sqrt{|y|}}\int_{-c/|y|}^{0}\frac{dt}{2\sqrt{t+c/|y|}}\leq \frac{\sqrt{2}}{\sqrt{|y|}}.\]
	Given \(c>2|y|\) we write \([-c/|y|,0]=[-c/|y|,1-c/|y|]\cup[1-c/|y|,-1]\cup[-1,0]\). We find
	\[\int_{-1}^0\frac{dt}{2\sqrt{|y|t+c}\sqrt{t^2+1}}\leq \int_{-1}^0\frac{dt}{2\sqrt{c-|y|}}\leq \frac{1}{2\sqrt{|y|}} \]
	and
	\[\int_{1-c/|y|}^{-1}\frac{dt}{2\sqrt{|y|t+c}\sqrt{t^2+1}}\leq\int_{1-c/|y|}^{-1}\frac{dt}{2\sqrt{|y|}|t|}\leq \frac{1}{2\sqrt{|y|}}\log\left(c/|y|-1\right)\leq -\frac{1}{2\sqrt{|y|}}\log|y|\]
	and
	\[\int_{-c/|y|}^{1-c/|y|}\frac{dt}{2\sqrt{|y|t+c}\sqrt{t^2+1}}\leq\frac{1}{\sqrt{|y|}}\int_{-c/|y|}^{1-c/|y|}\frac{dt}{2\sqrt{t+c/|y|}}=\frac{1}{\sqrt{|y|}}.\]
	From the calculations above the claim follows.
\end{proof}
\begin{lemma}\label{lem:BoundaryIntegrableLocal}
	Let \(Y\subset \C^{n+1}\) be a domain around zero and \(\varphi\colon Y\to \R\) a smooth function such that \(\varphi(0)=0\), \(d\varphi_p\neq 0\) for all \(p\in Y\) and at least one of the following conditions is satisfied
	\begin{itemize}
		\item [(i)] \(\frac{\partial \varphi}{\partial z_j}(0)\neq 0\) for some \(j\geq2\),
		\item[(ii)] there exists \(v\in\C^{n+1}\) with \(\partial\varphi_0(v)=0\) and \((\partial\overline{\partial}\varphi)_0(v,\overline{v})\neq 0\).
	\end{itemize}
	Consider the domain \(G:=\{z\in Y\mid \varphi(z)<0\}\) in \(Y\) with smooth boundary \(bG=\{z\in Y\mid \varphi(z)=0\}\). Let \(\iota\colon bG\to \C^{n+1}\) denote the inclusion map. Define \(\alpha,\beta\colon \C^n\to\R\) by \(\alpha(z)=-\log|z_1|\), \(\beta(z)=|z_1|^{-1}\) for \(z_1\neq 0\). There exists a neighborhood \(U\subset Y\) around zero such that \(\alpha\) and \(\beta\) are integrable on \(U\cap G\) and \(\iota^*\alpha\) and \(\iota^*\beta\) are integrable on \(U\cap bG\).   
\end{lemma}
\begin{proof}
	We have that \(\alpha,\beta\) are locally integrable on \(\C^{n+1}\). Then for \(U\) small enough we have that \(\int_{G\cap U} |\alpha|dV_{\C^{n+1}}\) and \(\int_{G\cap U} |\beta| dV_{\C^{n+1}}\) are finite. \\
	We have that \(bG\) is a real codimension one submanifold of \(\C^{n+1}\). Then by Lemma~\ref{lem:logIntegrableReal} we find that \(\iota^*\alpha\) is integrable on \(U\cap bG\) for \(U\) small enough.\\ 
	We have to show that \(\iota^*\beta\) is integrable in a neighborhood of zero in \(bG\). In the following we will use the identification \( \R^{n+1}\times\R^{n+1}\simeq\C^{n+1}\), \((x,y)\mapsto z=x+iy\). First, let us consider the case that \(\frac{\partial \varphi}{\partial z_j}(0)\neq 0\) for some \(j\geq2\). Without loss of generality assume \(\frac{\partial \varphi}{\partial x_{n+1}}(0)-i\frac{\partial \varphi}{\partial y_{n+1}}(0)=2\frac{\partial \varphi}{\partial z_{n+1}}(0)\neq 0\). In particular, assume \(\frac{\partial \varphi}{\partial y_{n+1}}(0)\neq 0\). Then the restriction of the projection \(\R^{2n+2}\ni(x,y)\mapsto (x,y_1,\ldots,y_n)\in\R^{2n+1}\) to a small neighborhood of zero in \(bG\) defines a diffeomorphism on its image.
	Since \((x,y_1,\ldots,y_n)\mapsto \left(\sqrt{x^2_1+y^2_1}\right)^{-1}\) is integrable on a neighborhood around zero in \(\R^{2n+1}\) we obtain from the transformation law that \(\iota^*\beta\) is integrable on \(bG\cap U\) when the open set \(U\) around zero in \(\C^{n+1}\) is chosen small enough. \\
	Now consider the case that \(\frac{\partial \varphi}{\partial z_j}(0)= 0\) for all \(j\geq2\). Then by the assumptions on \(\varphi\) we have \(\frac{\partial \varphi}{\partial x_1}(0)-i\frac{\partial \varphi}{\partial y_1}(0)=2\frac{\partial \varphi}{\partial z_1}(0)\neq 0\). Without loss of generality assume \(\frac{\partial \varphi}{\partial x_1}(0)\neq 0\).  By the implicit function theorem there exists a smooth real valued function \(g\) defined on an open neighborhood of zero in \(\R^{2n+1}\) with \(g(0)=0\) such that near the origin \(bG\) can be written as the graph of \(g\) that is \(x_1=g(x',y)\) with \(x'=(x_2,\ldots,x_{n+1})\). Then by the transformation it is enough to show that 
	\begin{eqnarray}\label{eq:FunctionLocalIntegrable}
		(x',y)\mapsto \left(\sqrt{(g(x',y))^2+y_1^2}\right)^{-1}
	\end{eqnarray}
	is integrable near zero in order to prove that \(\iota^*\beta\) is integrable near zero. We will show now that under the assumptions on \(\varphi\) the integrability of \eqref{eq:FunctionLocalIntegrable} near zero follows from the estimate in Lemma~\ref{lem:EstimateIntegral1D}.  By the Levi form assumptions on \(\varphi\) we can assume without loss of generality that \(\frac{\partial^2 \varphi}{\partial^2 x_{2}}(0)-\frac{\partial^2 \varphi}{\partial^2 y_{2}}(0)=2\frac{\partial^2 \varphi}{\partial \overline{z}_{2}\partial z_{2}}(0)\neq 0\). Let us say \(\frac{\partial^2 \varphi}{\partial^2 x_{2}}(0)\neq 0\). Since \(\frac{\partial \varphi}{\partial x_{2}}(0)=0\) we find \(\frac{\partial g}{\partial x_{2}}(0)=0\) and hence that \(\frac{\partial \eta}{\partial x_{2}}(0)\neq0\) where \(\eta\) is defined by \(\eta(x',y):=\left(\frac{\partial \varphi}{\partial x_{2}}\right)(g(x',y),x',y)\).  By the implicit function theorem there exists a smooth real valued function \(h\) with \(h(0)=0\) such that near the origin \(\eta^{-1}(0)\) can be written as the graph of \(h\) that is \(x_2=h(x_3,\ldots,x_{n+1},y)\). For \(x=(x_1,\ldots,x_{n+1})\in\R^{n+1}\) write \(x''=(x_3,\ldots,x_{n+1})\). 
	Using Taylor expansion we find \[g(h(x'',y)+t,x'',y)=g(h(x'',y),x'',y)+\frac{\partial g}{\partial x_{2}}(h(x'',y),x'',y)t+\frac{1}{2}\left(\frac{\partial^2 g}{\partial^2 x_{2}}\right)(h(x',y),x',y)t^2+O(t^3).\]
	By the construction of \(h\) we have \(\frac{\partial g}{\partial x_{2}}(h(x'',y),x'',y)=0\) for all \((x'',y)\) near zero. Furthermore, \(\frac{\partial^2 \varphi}{\partial^2 x_{2}}(h(x'',y),x'',y)\neq 0\) for all \((x'',y)\) near zero.  Hence we find \(g(h(x'',y)+t,x'',y)=g(h(x'',y),x'',y)+(f(t,x'',y)t)^2\) where \(f\) is smooth function defined in a neighborhood of zero in \(\R^{2n+1}\) where \(f\) is bounded and bounded from below by some positive constant and has bounded derivatives. Now it is enough to show that the function
	\begin{eqnarray}\label{eq:FunctionLocalIntegrable2}
		(t,x'',y)\mapsto \left(\sqrt{(g(h(x'',y),x',y)+t^2)^2+y_1^2}\right)^{-1}
	\end{eqnarray} 
	is integrable in a sufficiently small open neigborhood of zero in \(\R^{2n+1}\). Because then the integrability of \(\iota^*\beta\) on \(bG\cap U\) for some open neighbourhood \(U\subset \C^{n+1}\) around zero follows from the transformation law.\\
	We find
	\[ \frac{1}{\sqrt{(g(h(x'',y),x'',y)+t^2)^2+y_1^2}}\leq \frac{1}{\sqrt{(\min\{g(h(x'',y),x'',y),0)\}+t^2)^2+y_1^2}}\]
	and 
	\[\int_{-\varepsilon}^\varepsilon\left(\sqrt{(\min\{g(h(x'',y),x'',y),0\}+t^2)^2+y_1^2}\right)^{-1} dt\leq \frac{1}{\sqrt{|y_1|}}(|\log(|y_1|)|+4)\]
	by Lemma~\ref{lem:EstimateIntegral1D} for sufficiently small \(\varepsilon>0\) and all \((x'',y)\), \(y_1\neq 0\), in a neighborhood of zero in \(\R^{2n}\). Note that the function \((x'',y)\mapsto \left(\sqrt{|y_1|}\right)^{-1}(|\log(|y_1|)|+4)\) is integrable in a small neighborhood of \(0\in\R^{2n}\). Then, since the function in \eqref{eq:FunctionLocalIntegrable2} is non-negative, its integrability near zero follows from the Fubini-Tonelli theorem.		 
\end{proof}
From Lemma~\ref{lem:BoundaryIntegrableLocal} we can deduce Lemma~\ref{lem:BoundaryIntegrableGlobal}.
\begin{proof}[\textbf{Proof of Lemma~\ref{lem:BoundaryIntegrableGlobal}}]
	Since \(f\) is smooth up to \(bG\) we can assume that \(f\) is the restriction of a smooth function  denoted by \(f_1\) defined on an open neighborhood of \(\overline{G}\). Since  \(Y\) is a complex manifold and \((\partial f)_p\neq 0\) we find holomorphic functions \(f_2,\ldots,f_{n+1}\), \(f_2(p)=\ldots=f_{n+1}(p)=0\), defined on an open neighborhood of \(p\) in \(Y\) such that the differential \(dF_p\) of \(F=(f_1,\ldots,f_{n+1})\) in \(p\) is invertable. Hence \(F\) is a diffeomorphism on its image in an open neighborhood \(U\) of \(p\) in \(Y\) such that \(F(p)=0\), \(F|_{U\cap G}\) is holomorhic and \(F|_{U\cap bG}\) is CR.  By shrinking \(U\) we can assume that \(F(U\cap \overline{G})\) can be written as \(\{\varphi<0\}\) for some smooth real valued function \(\varphi\) defined on an open neighborhood around zero in \(\C^{n+1}\) satisfying the properties in Lemma~\ref{lem:BoundaryIntegrableLocal}.  Furthermore, in the coordinates given by \(F\) we find that \(f=z_1\) and \(\partial f=dz_1\) on \(F(U\cap \overline{G})\). Then the claims (i)-(iv) follow from Lemma~\ref{lem:BoundaryIntegrableLocal}.\\
	It remains to show that (v) holds. By Stokes formula we find
	\[\int_G \frac{\overline{f}}{|f|^2+\delta}\partial f \wedge \overline{\partial}\psi=\int_{bG}\log(|f|^2+\delta)\wedge \overline{\partial} \psi -\int_G \log(|f|^2+\delta)\wedge \partial\overline{\partial}\psi\] 
	for all \(\delta>0\) and all \(\psi\in \Omega_c^{n,n}(\overline{G})\) since the forms in the integrals are smooth.
	We have \(|f|/(|f|^2+\delta)\leq 1/|f|\) and \(-\log(|f|^2+\delta)\leq -2\log(|f|)\) for all \(\delta>0\). Furthermore, we have \(\lim_{\delta\to 0}\overline{f}/(|f|^2+\delta)= 1/f\) and \(\lim_{\delta\to 0}\log(|f|^2+\delta)=2\log(|f|)\) pointwise almost everywhere. Now by the already proven statements (i)-(iv) in Lemma~\ref{lem:BoundaryIntegrableGlobal} we can choose \(U\subset Y\) around \(p\) small enough such that \(\iota^*(\frac{1}{f}\partial f\wedge \overline{\partial}\psi)\) and \(\iota^*(\log(|f|)\overline{\partial}\psi\) are integrable on \(bG\) and \(\log(|f|)\wedge \partial\overline{\partial}\psi\) is integrable on \(G\) for all \(\psi\in \Omega_c^{n,n}(\overline{G}\cap U)\). For \(\psi\in \Omega_c^{n,n}(\overline{G}\cap U)\) we then conclude from the dominated convergence theorem that
	\begin{eqnarray*}
		\int_G \frac{1}{f}\partial f\wedge \overline{\partial}\psi &=&\lim_{\delta\to 0} \int_G \frac{\overline{f}}{|f|^2+\delta}\partial f \wedge \overline{\partial}\psi \\
		&=& \lim_{\delta\to 0} \int_{bG}\log(|f|^2+\delta)\wedge \overline{\partial} \psi -\int_G \log(|f|^2+\delta)\wedge \partial\overline{\partial}\psi \\
		&=& \int_{bG}\log(|f|^2)\wedge \overline{\partial} \psi -\int_G \log(|f|^2)\wedge \partial\overline{\partial}\psi.
	\end{eqnarray*}
\end{proof}
\subsection{Zero Sets of Holomorhic Functions on Domains with Boundary}
Now we are ready to give a definition of \( \int_G\partial\overline{\partial}\log(|f|)\wedge \psi\) where \(f\in\mathcal{O}\cap \mathscr{C}^\infty(\overline{G})\) is regular with respect to \(bG\) and \(\psi\in\Omega_c^{n,n}(\overline{G})\) does not need to vanish near the boundary. 
\begin{definition}\label{def:Definitionddbarlogf}
	With the notations above let \(f\in\mathcal{O}(G)\cap \mathscr{C}^\infty(\overline{G})\), \(f\not\equiv 0\), be regular with respect to \(bG\). Given \(\psi\in \Omega_c^{n,n}(\overline{G})\) we define
	\begin{eqnarray}\label{eq:Definitionddbarlogf}
		\text{\,\,\,\,\,\,\,}\int_G\partial\overline{\partial}\log(|f|)\wedge \psi:=-\int_{bG}\iota^*(\frac{1}{2f}\partial f\wedge \psi)-\int_{bG}\iota^*(\log(|f|)\wedge\overline{\partial}\psi)+\int_G\log(|f|)\wedge\partial \overline{\partial} \psi.
	\end{eqnarray}
\end{definition}	
We note that the right-hand side of \eqref{eq:Definitionddbarlogf} is well defined by Corollary~\ref{cor:BoundaryIntegrableGlobal}. Furthermore, we see that for \(\psi\in\Omega_c^{n,n}(G)\) in \eqref{eq:Definitionddbarlogf} we obtain the usual definition of \( \int_G\partial\overline{\partial}\log(|f|)\wedge \psi\).  Recall that by  Stokes formula we find for any \(\alpha\in \mathscr{C}^\infty(\overline{G})=\Omega^{0,0}(\overline{G})\) and \(\beta\in \Omega_c^{n,n}(\overline{G})\) that
\begin{eqnarray}\label{eq:StokesFormula}
	\int_G\partial\overline{\partial}\alpha\wedge\beta=-\int_{bG}\iota^*(\partial\alpha\wedge\beta)-\int_{bG}\iota^*(\alpha\overline{\partial}\beta)+\int_G\alpha\partial\overline{\partial}\beta.
\end{eqnarray}
Given \(\psi\in \Omega_c^{n,n}(\overline{G})\) with \(\text{supp}(\psi)\cap \{f=0\}=\emptyset\) we find that \(\log|f|\) is smooth in a neighborhood of \(\text{supp}(\psi)\). Then from \eqref{eq:StokesFormula} it follows that  \(\int_G\partial\overline{\partial}\log(|f|)\wedge \psi=0\).
\begin{remark}\label{rmk:CauchyRepresentationddbarlogf}
	From (v) in Lemma~\ref{lem:BoundaryIntegrableGlobal} it follows that any point \(p\in bG\) with \(f(p)=0\) has an open neighborhood \(U\subset Y\) such that
	\[\int_G\partial\overline{\partial}\log(|f|)\wedge \psi= -\int_{bG}\iota^*(\frac{1}{2f}\partial f\wedge \psi)-\int_G \frac{1}{2f}\partial f\wedge \overline{\partial}\psi\]
	holds for all \(\psi\in \Omega_c^{n,n}(\overline{G}\cap U)\).
\end{remark}
\begin{lemma}\label{lem:vanishingfornontagentialforms}
	With the notations above let \(f\in\mathcal{O}(G)\cap \mathscr{C}^\infty(\overline{G})\), \(f\not\equiv 0\), be regular with respect to \(bG\). Let \(\psi\in\Omega_c^{n,n}(\overline{G})\) be a form  which  satisfies \((\psi\wedge \partial f)_p=0\) or \((\psi\wedge \overline{\partial} f)_p=0\) at any point \(p\in \overline{G}\). Then  \[\int_G\partial\overline{\partial}\log(|f|)\wedge \psi=0.\]  
\end{lemma}
\begin{proof}
	For \(\delta>0\) we have that \(\alpha_\delta:=\log(|f|^2+\delta)\) is smooth on \(\overline{G}\). We calculate
	\[\partial\overline{\partial}\alpha_\delta=\frac{\delta}{(|f|^2+\delta)^2}\partial f\wedge\overline{\partial} f.\]
	Then for \(\psi\in\Omega_c^{n,n}(\overline{G})\) with \(\psi\wedge \partial f=0\) or \(\psi\wedge \overline{\partial} f=0\) at any point we observe that \(\int_{G}\partial\overline{\partial}\alpha_\delta\wedge \psi=0\).
	From Stokes formula (see \eqref{eq:StokesFormula}) we then find
	\[0=-\int_{bG}\frac{\overline{f}}{|f|^2+\delta}\partial f\wedge \psi-\int_{bG} \log(|f|^2+\delta)\wedge\overline{\partial}\psi+\int_{G}\log(|f|^2+\delta)\partial\overline{\partial}\psi\]
	for all \(\delta>0\). We have \(|f|/(|f|^2+\delta)\leq 1/|f|\) and the pullback of \(z\mapsto 1/|f|\) to \(bG\) is locally integrable by Corollary~\ref{cor:BoundaryIntegrableGlobal}. Furthermore, \(-\log(|f|^2+\delta)\leq -2\log(|f|)\) and \(z\mapsto -\log(|f|)\) is locally integrable on \(G\) and its pullback to  \(bG\) is also locally integrable by Corollary~\ref{cor:BoundaryIntegrableGlobal}. It then follows from the dominated convergence theorem that
	\begin{eqnarray*}
		0&=&\lim_{\delta\to 0}\left( -\int_{bG}\frac{\overline{f}}{|f|^2+\delta}\wedge \psi-\int_{bG} \log(|f|^2+\delta)\wedge\overline{\partial}\psi+\int_{G}\log(|f|^2+\delta)\partial\overline{\partial}\psi\right)\\
		&=&
		-\int_{bG}\frac{1}{f}\partial f\wedge \psi -\int_{bG} \log(|f|^2)\wedge\overline{\partial}\psi+\int_{G}\log(|f|^2)\partial\overline{\partial}\psi\\
	\end{eqnarray*}
	Hence we obtain
	\[\int_{G}\partial\overline{\partial}\log(|f|)\wedge \psi:=-\int_{bG}\iota^*(\frac{1}{2f}\partial f\wedge \psi)-\int_{bG}\iota^*(\log(|f|)\wedge\overline{\partial}\psi)+\int_{G}\log(|f|)\wedge\partial \overline{\partial} \psi=0.\]
\end{proof}
Before we can state the main theorem of this section we need the following.
\begin{lemma}\label{lem:ZeroDivisorWellDefined}
	Let \(f\in\mathcal{O}(G)\cap \mathscr{C}^\infty(\overline{G})\) be a holomorphic function on \(G\) which is smooth up to the boundary and \(p\in bG\) a point with \(f(p)=0\) such that \(f\) is regular at \(p\) with respect to \(bG\). There exists an open neighborhood \(U\subset Y\) around \(p\) and a smooth submanifold \(X\) of \(Y\) through \(p\) with real dimension \(\dim_\R X=2n\)  such that
	\begin{itemize}
		\item [(i)] \(X\) is connected and orientable. 
		\item [(ii)] \(U\cap G\cap \{f=0\}\) is a complex submanifold of \(G\) of complex dimension \(\dim_\C=n\),
		\item [(iii)] \(U\cap G\cap \{f=0\}\) is an open relatively compact subset of \(X\),
		\item [(iv)] \(U\cap bG\cap \{f=0\}\) has measure zero in \(X\). 
	\end{itemize} 
\end{lemma}
\begin{proof}
	We can write \(f=F|_{\overline{G}}\) where \(F\) is a smooth function defined on an open neighborhood of \(\overline{G}\) in \(Y\). Since \(\partial f_p\neq 0\) and \(f(p)=0\) we have  
	that there exists an open set \(V\) of \(Y\) around \(p\) such that \(X:=F^{-1}(0)\cap V\) is a real connected orientable submanifold of \(Y\) with real dimension \(\dim_\R X=2n\) and \(p\in X\) which proves (i).  Since \(G\) is open in \(Y\) we find that \(G\cap X\) is open in \(X\). In addition we can choose \(V\) such that \(\partial f\) does not vanish on \(X\cap G\). Then, since \(X\cap G= \{f=0\}\cap V\) and \(f\) is holomorphic on \(G\), we have that \(X\cap G\) is a complex \(n\)-dimensional submanifold of \(G\). Choosing \(U\) around \(p\) small enough we have that (ii) and (iii) holds. In order to prove (iv) we choose coordinates around \(p\) as in the proof of Lemma~\ref{lem:BoundaryIntegrableGlobal} keeping in mind that \(f=F|_{\overline{G}}\). Hence we can consider the case that \(V\) is an open neighborhood of \(\C^{n+1}\) around zero, \(X=\{z\in V\mid z_1=0\}\), \(G\cap V=\{z\in V\mid\varphi(z)<0\}\), \( bG\cap V=\{z\in V\mid \varphi(z)=0\}\) for some smooth function \(\varphi\colon V\to \R\) with \(\varphi(0)=0\), \(d\varphi_p\neq 0\) for all \(p\in V\) and \(\varphi\) satisfies at least one of the conditions (i) or (ii) in the assumptions of  Lemma~\ref{lem:BoundaryIntegrableLocal}. Given the case that \(\frac{\partial \varphi}{\partial z_j}(0)\neq 0\) for some \(j\geq2\) we can choose \(U\) around zero small enough such that \(U\cap bG\cap X\) is a  submanifold of \(X\) of real dimension \(2n-1\). Hence \(U\cap bG\cap \{f=0\}\) has measure zero in \(X\). Now consider the case that \(\frac{\partial \varphi}{\partial z_j}(0)= 0\) for all \(j\geq2\).  By the Levi form assumptions on \(\varphi\) we can assume without loss of generality that \(\frac{\partial^2 \varphi}{\partial^2 x_{2}}(0)-\frac{\partial^2 \varphi}{\partial^2 y_{2}}(0)=2\frac{\partial^2 \varphi}{\partial \overline{z}_{2}\partial z_{2}}(0)\neq 0\). Let us say \(\frac{\partial^2 \varphi}{\partial^2 x_{2}}(0)\neq 0\). Then choosing \(U\) around zero small enough we can assume \(\frac{\partial^2 \varphi}{\partial^2 x_{2}}(z)\neq 0\) on \(U\). Consider the sets \(A:=\{z\in U\mid \frac{\partial \varphi}{\partial x_2}(z)\neq 0\}\) and \(B:=U\setminus A\). Then \(A\cap X\cap bG\) is a submanifold of \(X\) of real dimension \(2n-1\). Hence \(A\cap X\cap bG\) has measure zero in \(X\). Furthermore, \(B\cap X\) is real submanifold of \(U\)  of real dimension \(2n-1\). It follows that \(B\cap X\cap bG\)  has measure zero in \(X\). We conclude \(bG\cap U\cap X=(A\cup B)\cap X\cap bG\) has measure zero in \(X\).
\end{proof}
Now we are ready to state the Lelong-Poincar\'e formula for domains with boundary.
\begin{theorem}\label{thm:PoincareLelongFormula}
	Let \(Y\) be a complex manifold of dimension \(\dim_\C Y=n+1\). Let \(G\subset Y\) be a domain with smooth boundary \(bG\). Let \(f\in\mathcal{O}(G)\cap \mathscr{C}^\infty(\overline{G})\), \(f\not\equiv 0\), be regular with respect to \(bG\).  We have
	\begin{eqnarray}\label{eq:PoincareLelongFormula}
		\left(\divisor{f},\psi\right)=\frac{i}{\pi}\int_G \partial\overline{\partial}\log(|f|)\wedge \psi
	\end{eqnarray}
	for all \(\psi\in \Omega_c^{n,n}(\overline{G})\) where \(\divisor{f}\) denotes the zero divisor of \(f\) and the right-hand side is given by Definition~\ref{def:Definitionddbarlogf}.
\end{theorem}
\begin{remark}\label{rmk:ZeroDivisorWellDefined}
	We note that the left-hand side in~\eqref{eq:PoincareLelongFormula} is well-defined 
	by Lemma~\ref{lem:ZeroDivisorWellDefined}. Furthermore, it follows from (iv) in Lemma~\ref{lem:ZeroDivisorWellDefined} that the value of \(\left(\divisor{f},\psi\right)\) for  \(\psi\in \Omega_c^{n,n}(\overline{G})\) stays the same when integration is performed on \(G\) instead of \(\overline{G}\). This issue should be considered with respect to Example~\ref{ex:NonRegularAtBoundary} and Example~\ref{ex:RegularAtBoundary}. 
\end{remark}

\begin{proof}[\textbf{Proof of Theorem~\ref{thm:PoincareLelongFormula}}]
	It is enough to show that for any point \(p\in\overline{G}\) there is an open neighborhood \(U\subset Y\) around \(p\) such that the statement holds for all \(\psi\in\Omega_c^{n,n}(U\cap \overline{G})\). For \(p\in G\) we can take \(U=G\) and the statement follows from the usual Lelong-Poincar\'e formula. Given \(p\in bG\) with \(f(p)\neq 0\) we can choose \(U\) around \(p\) such that \(f\) does not vanish on \(U\cap \overline{G}\). Then on the one hand one has \((\divisor{f},\psi)=0\) for all \(\psi\in\Omega_c^{n,n}(U\cap \overline{G})\). On the other hand we have that \(\log|f|\) is smooth on \(U\cap\overline{G}\). Hence \(\int_G \partial\overline{\partial}\log(|f|)\wedge \psi=0\) for all \(\psi\in\Omega_c^{n,n}(U\cap \overline{G})\) by Definition~\ref{def:Definitionddbarlogf} and \eqref{eq:StokesFormula} which shows that \eqref{eq:PoincareLelongFormula} holds in that case. Now take \(p\in bG\) with \(f(p)= 0\). Choosing local coordinates as in the proof of Lemma~\ref{lem:BoundaryIntegrableGlobal} we can consider  an open neighborhood of \(V\subset \C^{n+1}\) around \(p=0\), \(f(z)=z_1\), \(G\cap V=\{z\in V\mid\varphi(z)<0\}\), \( bG\cap V=\{z\in V\mid \varphi(z)=0\}\) for some smooth function \(\varphi\colon V\to \R\) with \(\varphi(0)=0\), \(d\varphi_0\neq 0\) and there exists \(v\in\C^{n+1}\) with \(\partial\varphi_0(v)=0\) and \((\partial\overline{\partial}\varphi)_0(v,\overline{v})\neq 0\). We will first show that we can assume \(\frac{\partial \varphi}{\partial z_1}(z)\neq0\) on \(V\). If \(\frac{\partial \varphi}{\partial z_1}(0)\neq0\) this follows immediately by shrinking \(V\).  If \(\frac{\partial \varphi}{\partial z_1}(0)=0\) we have that \(\frac{\partial \varphi}{\partial z_j}(0)\neq0\) for some \(j\geq 2\). Then replace the coordinate \(z_j\) by \(z_j+z_1\). In the new coordinates we have \(f(z)=z_1\), \(\frac{\partial \varphi}{\partial z_1}(0)\neq0\) and the other assumptions on \(\varphi\) are still valid on some small open neighborhood around zero. Then by shrinking \(V\)  it follows that \(\frac{\partial \varphi}{\partial z_1}(z)\neq0\) on \(V\).
	Let \(B\) be the open ball in \(\C^{n}\) centered at zero of radius \(\varepsilon>0\). Let \(D\) be the open disc in \(\C\) centered at zero of radius \(\varepsilon>0\). Given \(\varepsilon>0\) small enough we have \(U:=D\times B\subset\subset V\) is an open neighborhood around zero relatively compact in \(V\). For \(z=(z_1,\ldots,z_{n+1})\) write \(z'=(z_2,\ldots,z_{n+1})\) and for \(z'\in B\) put \(D(z')=\{z\in D\mid \varphi(z,z')<0\}\). By the assumptions on \(\varphi\) we have for any \(z'\in B\) that \(D(z')\) is an open set with smooth boundary in \(D\) (that is the boundary of \(D(z')\) relative to \(D\) is smooth). We define \(A_1:=\{z'\in B\mid \varphi(0,z')<0\}\), \(A_2:=\{z'\in B\mid \varphi(0,z')=0\}\) and \(A_3:=\{z'\in B\mid \varphi(0,z')>0\}\). Then \(A_j\), \(j=1,2,3\), are pairwise disjoint with \(B=A_1\cup A_2\cup A_3\). Furthermore, \(A_1\) and \(A_3\) are open and from (iv) in Lemma~\ref{lem:ZeroDivisorWellDefined} we have that \(A_2\) has measure zero in \(\C^n\). Now take \(\psi\in \Omega_c^{n,n}(U\cap\overline{G})\). We can write \(\psi=\sum_{1\leq j,k\leq n+1}\psi_{j,k}\widehat{dz_j}\wedge \widehat{d\overline{z}_k}\) for some \(\psi_{j,k}\in \mathscr{C}_c^\infty(U\cap\overline{G})\) where \(\widehat{dz_j}=dz_1\wedge\ldots \wedge dz_{j-1}\wedge dz_{j+1}\wedge \ldots \wedge dz_{n+1}\) and \(\widehat{d\overline{z}_k}=\overline{\widehat{dz_k}}\). Since \(\{f=0\}\cap U\cap G=\{0\}\times A_1\) it follows that
	\[\left(\divisor{f},\psi\right)=\int_{A_1}\psi_{1,1}\widehat{dz_1}\wedge \widehat{d\overline{z}_1}.\]
	Put \(\psi'=\psi-\psi_{1,1}\widehat{dz_1}\wedge \widehat{d\overline{z}_1}\).
	Since \(\partial f=dz_1\)  we obtain from Lemma~\ref{lem:vanishingfornontagentialforms} that \(\int_{G}\partial\overline{\partial}\log(|f|)\wedge \psi'=0\). 
	It remains to show that \(\int_{G}\partial\overline{\partial}\log(|f|)\wedge \psi_{1,1}\widehat{dz_1}\wedge \widehat{d\overline{z}_1}=-i\pi\int_{A_1}\psi_{1,1}\widehat{dz_1}\wedge \widehat{d\overline{z}_1}\). Choosing \(\varepsilon>0\) in the definition of \(U\) small enough we find by Remark~\ref{rmk:CauchyRepresentationddbarlogf} that
	\[\int_{G}\partial\overline{\partial}\log(|f|)\wedge \psi_{1,1}\widehat{dz_1}\wedge \widehat{d\overline{z}_1}=-\int_{bG} \iota^*(\frac{1}{2z_1}\psi_{1,1}dz_1\wedge\widehat{dz_1}\wedge \widehat{d\overline{z}_1})-\int_{G}\frac{1}{2z_1}\frac{\partial \psi_{1,1}}{\partial \overline{z}_1}dz_1\wedge d\overline{z}_1\wedge\widehat{dz_1}\wedge \widehat{d\overline{z}_1}.\]
	Given \(z'\in A_3\) we have that \(0\notin D(z')\) and hence \(z_1\mapsto 1/z_1\) is smooth on \(D(z')\).  We obtain from Stokes formula in \(\C\) that
	\begin{eqnarray}\label{eq:StikesFormulaInCApplied}
		\int_{bD(z')}\frac{\psi_{1,1}}{z_1}dz_1=-\int_{D(z')}\frac{1}{z_1}\frac{\partial \psi_{1,1}}{\partial \overline{z}_1}dz_1\wedge d\overline{z}_1.
	\end{eqnarray}
	Given \(z'\in A_1\) we have that \(0\in D(z')\). Then it follows from the Cauchy-Pompeiu formula (see Remark~\ref{rmk:RemarkCauchyPompeiu} below) that 
	\begin{eqnarray}\label{eq:CauchyPompeiuFormulaApplied}
		2\pi i\psi_{1,1}(0,z')=\int_{ bD(z')}\frac{\psi_{1,1}}{z_1}dz_1+\int_{ D(z')}\frac{1}{z_1}\frac{\partial \psi_{1,1}}{\partial \overline{z}_1}dz_1\wedge d\overline{z}_1.
	\end{eqnarray}
	Since \(B\) is the union of the pairwise disjoint sets \(A_j\), \(j=1,2,3\), and \(A_2\) has measure zero we find by the Fubini theorem that
	\[\int_{bG}\iota^*(\frac{1}{z_1}\psi_{1,1}dz_1\wedge\widehat{dz_1}\wedge \widehat{d\overline{z}_1})=\int_{A_1\cup A_3}\left(\int_{bD(z')}\frac{1}{z_1}\psi_{1,1}dz_1\right)\widehat{dz_1}\wedge \widehat{d\overline{z}_1}\] 
	and
	\[\int_{G}\frac{1}{z_1}\frac{\partial \psi_{1,1}}{\partial \overline{z}_1}dz_1\wedge d\overline{z}_1\wedge\widehat{dz_1}\wedge \widehat{d\overline{z}_1}=\int_{A_1\cup A_3}\left(\int_{D(z')}\frac{1}{z_1}\frac{\partial \psi_{1,1}}{\partial \overline{z}_1}dz_1\wedge d\overline{z}_1\right)\widehat{dz_1}\wedge \widehat{d\overline{z}_1}.\]
	From~\eqref{eq:StikesFormulaInCApplied} and~\eqref{eq:CauchyPompeiuFormulaApplied} it follows that 
	\[\int_{A_3}\left(-\int_{ bD(z')}\frac{\psi_{1,1}}{z_1}dz_1-\int_{ D(z')}\frac{1}{z_1}\frac{\partial \psi_{1,1}}{\partial \overline{z}_1}dz_1\wedge d\overline{z}_1.\right)\widehat{dz_1}\wedge \widehat{d\overline{z}_1}=0.\]
	and
	\[\int_{A_1}\left(-\int_{ bD(z')}\frac{\psi_{1,1}}{z_1}dz_1-\int_{ D(z')}\frac{1}{z_1}\frac{\partial \psi_{1,1}}{\partial \overline{z}_1}dz_1\wedge d\overline{z}_1.\right)\widehat{dz_1}\wedge \widehat{d\overline{z}_1}=-2\pi i\int_{A_1}\psi_{1,1}\widehat{dz_1}\wedge \widehat{d\overline{z}_1}.\]
	Hence we can conclude
	\[\frac{i}{\pi}\int_{G}\partial\overline{\partial}\log(|f|)\wedge\psi=\int_{A_1}\psi_{1,1}\widehat{dz_1}\wedge \widehat{d\overline{z}_1}=\left(\divisor{f},\psi\right)\]
	and the claim follows.
\end{proof}
\begin{remark}\label{rmk:RemarkCauchyPompeiu}
	Let \(U\subset \C\) be a domain. Let \(D\) be an open subset of \(U\) such that its boundary relative to \(U\) denoted by \(b^UD\) is smooth. Let \(\overline{D}^U\) be the closure of \(D\) relative to \(U\). For any \(\psi\in \mathscr{C}_c^\infty(\overline{D}^U)\) and any \(w\in D\) one has
	\[\psi(w)=\frac{1}{2\pi i}\left(\int_{b^UD}\psi\frac{dz}{z-w}+\int_D\frac{\partial \psi}{\partial \overline{z}}\frac{dz\wedge d\overline{z}}{z-w}\right).\]
	This formula is a version of the Cauchy-Pompeiu formula for non bounded domains in \(\C\). For the proof take a small disc \(B\) around \(w\) relatively compact in \(D\). The statement then follows from Stokes theorem applied with respect to the domain \(D\setminus \overline{B}\) and the Cauchy-Pompeiu formula for the disc \(B\).   
\end{remark}

\subsection{Zero Sets of CR Functions}
We are now going to prove two results (see Theorem~\ref{thm:OneOverFIntegrableCR} and Theorem~\ref{thm:PoincareLelongForCR}) concerning the zero sets of CR functions on strictly pseudoconvex CR manifolds. For the proofs we need the following.
\begin{proposition}\label{prop:CnCoordinatesForX}
	Let \((X,T^{1,0}X)\) be as in Theorem~\ref{thm:ProjectiveEmbeddingIntro} and let \(p\in X\) be a point. There exists an open neighborhood \(V\) around \(p\), a domain \(Y\subset \C^{n+1}\) around zero, a smooth function \(\varphi\colon Y\to \R\) with non-vanishing differential \(d\varphi\) and a CR embedding \(F\colon V\to Y\) with the following properties: 
	\begin{itemize}
		\item [(i)] \(F(V)=\{\varphi=0\}\), \(F(p)=0\) and \(\{\varphi=0\}\) is connected,
		\item [(ii)] \(G:=\{z\in Y\mid \varphi(z)<0\}\) is a domain in \(Y\) with smooth boundary (relative to \(Y\)),
		\item[(iii)] \(\sum_{0\leq j,k\leq n}\frac{\partial^{2}\varphi}{\partial z_j\partial \overline{z}_k}(z)w_j\overline{w}_k>0\) for all \(z\in\{\varphi=0\}\) and all \(w=(w_0,\ldots,w_n)\) with \(\sum_{j=0}^nw_j\frac{\partial\varphi}{\partial z_j}(z)=0\).
	\end{itemize}
\end{proposition}
\begin{proof}
	 By the assumptions on \(X\) we have that \(X\) is CR embeddable into the complex euclidean space. Hence we can find an open neighborhood \(V\subset X\) around \(p\), an open set \(Y\subset \C^{n+1}\) around zero, a smooth function \(\varphi\colon Y\to \R\) with non-vanishing differential \(d\varphi\) and a CR embedding \(F\colon V\to Y\) with \(F(V)=\{\varphi=0\}\) and \(F(p)=0\). Shrinking \(V\) and \(Y\) we can assume that \(\{\varphi=0\}\) is connected. Since \(X\) is strictly pseudoconvex, by possibly changing the sign of \(\varphi\)  we can assume that (iii) is satisfied. Shrinking \(Y\) again we obtain that \(G\) is a domain in \(Y\).
\end{proof}
\begin{remark}
	In the following when we apply Proposition~\ref{prop:CnCoordinatesForX} we consider the closure and boundary of \(G\) always relative to \(Y\), that is, \(\overline{G}:=\overline{G}^Y\) and \(bG:=b^YG\). 
\end{remark}

\begin{theorem}\label{thm:OneOverFIntegrableCR}
	Let \((X,T^{1,0}X)\) be as in Theorem~\ref{thm:ProjectiveEmbeddingIntro} and let \(f\in H^0_b(X)\cap\mathscr{C}^\infty(X)\) be a smooth CR function such that \(df_p\neq 0\) for any \(p\in X\) with \(f(p)=0\). Then \(1/f\) is integrable on \(X\).
\end{theorem}
\begin{remark}\label{rmk:dfneq0notregularvalue}
	We note that the assumptions on \(f\) do not imply in general that \(0\) is a regular value of the map \(f\colon X\to \C\simeq \R^2\). (See Example~\ref{ex:RegularAtBoundary}).
\end{remark}
\begin{proof}[\textbf{Proof of Theorem~\ref{thm:OneOverFIntegrableCR}}]
	We show that any point \(p\in X\) has an open neighborhood such that \(1/f\) is integrable there. Let \(p\in X\) be a point. If \(f(p)\neq 0\) there is nothing to show. So we assume \(f(p)=0\). By Proposition~\ref{prop:CnCoordinatesForX} we find an open neighborhood around \(p\) which can be CR isomorphically identified with a set \(bG=\{\varphi=0\}\subset Y\subset \C^{n+1}\) where \(Y,G,\varphi\) satisfy the conditions in Proposition~\ref{prop:CnCoordinatesForX}.   By \cite[Theorem 2.6.13]{Hoermander_2000}  after shrinking \(Y\) we find a function \(F\in \mathcal{O}(G)\cap \mathscr{C}^\infty(\overline{G})\) with \(F|_{bG}=f\). In addition, since \(df_p\neq 0\), we can achieve by shrinking  \(Y\) again that \(df\) does not vanish on \(bG\). Since \(bG\) is strictly pseudoconvex we have that \(F\) is regular with respect to \(bG\). Then it follows from Lemma~\ref{lem:BoundaryIntegrableGlobal}  that \(1/f\) is integrable in a small neighborhood around \(p\). 
\end{proof}

In the situation of Theorem~\ref{thm:OneOverFIntegrableCR} assuming that \(0\) is a regular value of the map \(f\colon X\to \C\) we have that \(\{f=0\}:=f^{-1}(0)\) is a smooth compact oriented submanifold in \(X\) of real codimension two. Given \(\psi\in \Omega^{2n-1}(X)\) we then define 
\begin{eqnarray}\label{eq:DefZeroDivisorCRFkt}
	\left(\divisor{f},\psi\right):=\int_{\{f=0\}}\ell^*\psi
\end{eqnarray}
where \(\ell\colon \{f=0\}\to X\) denotes the inclusion map. We have the following.
\begin{theorem}\label{thm:PoincareLelongForCR}
	Let \((X,T^{1,0}X)\) be as in Theorem~\ref{thm:ProjectiveEmbeddingIntro} and let \(f\in H^0_b(X)\cap\mathscr{C}^\infty(X)\) be a smooth CR function such that \(0\) is a regular value of \(f\colon X\to \C\). Then for any \(\psi\in\Omega^{2n-1}(X)\) we have
	\[\left(\divisor{f},\psi\right)=\frac{1}{2\pi i} \int_{X}\frac{df}{f}\wedge d\psi.\]
\end{theorem}
\begin{proof}
	The theorem follows immediately from Theorem~\ref{thm:ZeroDivisorComplexValuedFunctions} below.
\end{proof}
It turns out that the result in Theorem~\ref{thm:PoincareLelongForCR} holds for any oriented smooth real manifold and any smooth complex valued function which has zero as a regular value what we are going to explain now.
   
Let \(M\) be an oriented smooth  real manifold of dimension \(n\geq 2\). Given a smooth function \(f\in\mathscr{C}^\infty(M)\) such that zero is a regular value of the map \(f\colon M\to\C\) we have that \(\{f=0\}:=\{x\in M\mid f(x)=0\}\) is an oriented smooth  submanifold of dimension \(n-2\).  We then define
\begin{eqnarray}
	\left(\divisor{f},\psi\right) = \int_{\{f=0\}}\iota^*\psi
\end{eqnarray}
for \(\psi\in\Omega^{n-2}_c(M)\) where \(\iota\colon \{f=0\}\to M\) denotes the inclusion map. We would like to show that \(\divisor{f}\) is given by \((2\pi i)^{-1}d(df/f)\) in the sense of distributions (see Theorem~\ref{thm:ZeroDivisorComplexValuedFunctions}). Therefore we need the following.
\begin{lemma}\label{lem:OneOverFWellDefinedRealManifold}
	Let \(M\) be an oriented smooth  real manifold of dimension \(n\geq 2\). Given a  function \(f\in\mathscr{C}^\infty(M)\) such that zero is a regular value of the map \(f\colon M\to\C\) we have that \(1/f\) is locally integrable.
\end{lemma}
\begin{proof}
	We need to show that any point \(p\in M\) has an open neighborhood where \(1/f\) is integrable. Given \(f(p)\neq 0\) there is nothing to show. So let us assume \(f(p)=0\). By the assumptions on \(f\) we can choose local coordinates \((x_1,\ldots,x_n)\) in an open neighborhood \(U\) around \(p\) such that \(x_1=\Re f\), \(x_2=\Im f\). In this coordinates we find \(1/f=\frac{1}{x_1+ix_2}\). We note that \((x_1,x_2)\mapsto |x_1+ix_2|^{-1}=|(x_1,x_2)|^{-1}\) 
	is locally integrable on \(\R^2\). Then the claim follows from the Fubini-Tonelli theorem.
\end{proof}
\begin{theorem}\label{thm:ZeroDivisorComplexValuedFunctions}
	Let \(M\) be an oriented smooth  real manifold of dimension \(n\geq 2\). For any function \(f\in\mathscr{C}^\infty(M)\) such that zero is a regular value of the map \(f\colon M\to\C\) we have \(\divisor{f}=(2\pi i)^{-1}d(df/f)\) in the sense of distributions, that is,
	\begin{eqnarray}\label{eq:ZeroDivisorComplexValuedFunctions}
			\left(\divisor{f},\psi\right)= \frac{1}{2\pi i}\int_M \frac{df}{f}\wedge d\psi
	\end{eqnarray}
	for all \(\psi\in \Omega^{n-2}_c(M)\). 
\end{theorem}
We note that the right-hand side of~\eqref{eq:ZeroDivisorComplexValuedFunctions} is well-defined by Lemma~\ref{lem:OneOverFWellDefinedRealManifold}.
\begin{remark}\label{rmk:DivisorsOfVectorvaluedMaps}
	Theorem~\ref{thm:ZeroDivisorComplexValuedFunctions} can be obtained from the results due to Harvey--Semmes~\cite{HS92} under much more general assumptions on the function \(f\). Here, we will give an elementary proof for the simple case that zero is a regular value of the map \(f\colon X\to\C\).
\end{remark}
\begin{proof}[\textbf{Proof of Theorem~\ref{thm:ZeroDivisorComplexValuedFunctions}}]
	It is enough to show that any point \(p\in M\) has an open neighborhood where the statement holds. Let \(p\in M\) be a point with \(f(p)\neq 0\). Consider the open neighborhood  \(U:=M\setminus \{f=0\}\) around \(p\). We have that \(df/f\) is smooth on \(U\) with \(d(df/f)=0\). Given \(\psi\in \Omega^{n-2}_c(U)\) we find by Stokes theorem that
	\[\int_M\frac{df}{f}\wedge d\psi=\int_U\frac{df}{f}\wedge d\psi=\int_U d\left(\frac{df}{f}\right)\wedge \psi=0.\]
	On the other hand we have \(\left(\divisor{f},\psi\right)=0\). Hence~\eqref{eq:ZeroDivisorComplexValuedFunctions} holds on \(U\).\\
	Now let \(p\in M\) be a point with \(f(p)=0\). By the assumptions on \(f\) we find local coordinates \(x_1,x_2,\ldots,x_n\) on an open neighborhood \(U\subset M\) around \(p\) such that \(\Re f=x_1\), \(\Im f=x_2\) and so that \(p\) is mapped to the origin in \(\R^n\). In these coordinates, after shrinking \(U\), we can identify \(U\) with an open neighborhood in \(\R^n\) of the form \(U=D\times B\), where \(D\subset \R^2\), \(B\subset \R^{n-2}\) are balls centered in zero, such that \(p=0\in\R^n\) and \(\{f=0\}\cap U=\{x\in U\mid x_1=x_2=0\}=\{0\}\times B\). We use the notation \(x=(x_1,x_2,x')\in\R^2\times \R^{n-2}\) with \(x'=(x_3,\ldots,x_n)\). Furthermore, put \(z:=x_1+ix_2\in\C\) with
	\[dz=dx_1+idx_2,\,\, d\overline{z}=dx_1-idx_2,\,\,\frac{\partial}{\partial z}=\frac{1}{2}\left(\frac{\partial}{\partial x_1}-i\frac{\partial}{\partial x_2}\right),\,\,\frac{\partial}{\partial \overline{z}}=\frac{1}{2}\left(\frac{\partial}{\partial x_1}+i\frac{\partial}{\partial x_2}\right).\]
	Using this notation we find \(df/f=dz/z\) and
	\[d\nu = \partial \nu+\overline{\partial}\nu +d'\nu \]
	for all \(\nu\in\mathscr{C}^\infty(U)\) where \(\partial \nu=  \frac{\partial \nu}{\partial z}dz\), \(\overline{\partial}\nu = \frac{\partial \nu}{\partial \overline{z}}d\overline{z}\), \(d'\nu=\sum_{j=3}^n\frac{\partial \nu}{\partial x_j}dx_j\). Note that we can declare \(\partial,\overline{\partial},d'\) to act on forms in \(\Omega^\bullet(U)\) in the standard way.\\
	Let  \(\psi\in \Omega^{n-2}_c(U)\) be arbitrary.  Since \(2dx_1= dz+d\overline{z}\), \(2idx_2= dz-d\overline{z}\) we can write \(\psi=\alpha_0+\psi'\) with
	\[\psi'=dz\wedge \alpha_1+d\overline{z}\wedge \alpha_2 +dz\wedge d\overline{z}\wedge \alpha_3\]
	such that \(\alpha_0\in \Omega^{n-2}_c(U)\), \(\alpha_1,\alpha_2\in \Omega^{n-3}_c(U)\), \(\alpha_3\in \Omega^{n-4}_c(U)\) and \(\iota_{\frac{\partial }{\partial x_j}}\alpha_k=0\) for \(j=1,2\) and \(k=0,\ldots,3\). It follows that   
	\(\alpha_0=\nu dx_3\wedge\ldots\wedge dx_n\)
	for some smooth function \(\nu\in C^\infty_c(U)\) and hence we obtain \(\iota^*\psi=\nu\circ \iota dx_3\wedge\ldots\wedge dx_n\).
	We conclude that
	\begin{eqnarray}\label{eq:ZeroDivisorComplexValuedFunctionsLeftSide}
	\left(\divisor{f},\psi\right)=\int_{\{f=0\}\cap U}\iota^*\psi=\int_{\{0\}\times B}\iota^*\psi=\int_{B}\nu(0,x')dx_3\ldots dx_n.
	\end{eqnarray}
	We will now show that~\eqref{eq:ZeroDivisorComplexValuedFunctions} holds on \(U\) in two steps: First we prove \(\int_U \frac{df}{f}\wedge d\psi'=0\). Second we calculate \(\int_U \frac{df}{f}\wedge d\alpha_0\).\\
	We have
	\(dz\wedge d\psi'=dz\wedge d\overline{z}\wedge d'\alpha_2\).
	For  \(z\neq 0\) fixed put \(U_z=\{z\}\times B\) and denote by \(\ell\colon U_z\to U\) the inclusion map. We obtain
	\begin{eqnarray*}
		\int_{U_z}\ell^*(d'\alpha_2)=\int_{U_z}\ell^*(d\alpha_2)=\int_{U_z}d\ell^*\alpha_2=0.
	\end{eqnarray*}
	since \(\ell^*\alpha_2\) has compact support in \(U_z\). Since \(z\mapsto \frac{1}{z}\) is locally integrable on \(\C\) 
	we conclude from the Fubini-Tonelli theorem that
	\begin{eqnarray}\label{eq:ZeroDivisorComplexValuedFunctionsPsiPrime}
		\int_U\frac{df}{f}\wedge d\psi'=\int_U\frac{dz}{z}\wedge d\overline{z} \wedge d'\alpha_2=\int_{D}\left(\int_{U_z}\ell^*(d'\alpha_2)\right)\frac{dz\wedge d\overline{z}}{z}=0.
	\end{eqnarray}
	We have \(dz\wedge d\alpha_0=\frac{\partial \nu}{\partial \overline{z}} dz\wedge d\overline{z} \wedge dx_3\wedge \ldots\wedge dx_n\). Again by the Fubini-Tonelli theorem we find
	\[\int_U\frac{df}{f}\wedge d\alpha_0=\int_U \frac{\partial \nu}{\partial \overline{z}} \frac{dz}{z}\wedge d\overline{z}  \wedge dx_3\wedge \ldots\wedge dx_n=\int_B\left(\int_D\frac{\partial \nu}{\partial \overline{z}}\frac{dz\wedge d\overline{z}}{z}\right) dx_3\ldots dx_n. \]	
	Note that for any \(x'\in B \) we have that \(z\mapsto \nu(z,x')\) has compact support in \(D\). Then, using the Cauchy-Pompeiu formula, we obtain
	\begin{eqnarray}\label{eq:ZeroDivisorComplexValuedFunctionsAlphaZero}
		\int_U\frac{df}{f}\wedge d\alpha_0=2\pi i \int_B \nu(0,x')dx_3\ldots dx_n.
	\end{eqnarray}
	Since \(\psi=\alpha_0+\psi'\) we conclude from  \eqref{eq:ZeroDivisorComplexValuedFunctionsLeftSide}, \eqref{eq:ZeroDivisorComplexValuedFunctionsPsiPrime} and~\eqref{eq:ZeroDivisorComplexValuedFunctionsAlphaZero} that
	\[\left(\divisor{f},\psi\right)=  \int_B \nu(0,x')dx_3\ldots dx_n=\frac{1}{2\pi i} \int_U \frac{df}{f}\wedge d\psi.\]
\end{proof}

\section{Equidistribution on CR manifolds}\label{sec:EquidistributionCR}
In this section we are going to prove Theorem~\ref{thm:ExpectationValueCRDistributionIntro} and Theorem~\ref{thm:ConvergenceStrongCFIntro} stated in the introduction. We will first understand \(\mathcal{C}_f\) as a distribution valued random variable defined on a sequence of probability spaces \((A_k)_{k>0}\) contained in the space of smooth CR functions (see Definition~\ref{def:CF} and Lemma~\ref{lem:UnregularFunctionsAreZeroSetCR}). A direct calculation leads to the proof of Theorem~\ref{thm:ExpectationValueCRDistributionIntro} (see Theorem~\ref{thm:ExpectationValueCRDistribution} and Corollary~\ref{cor:ZerodistributionCRMfd}). In order to prove Theorem~\ref{thm:ConvergenceStrongCFIntro} (see Corollary~\ref{cor:ZerodistributionCRMfdSequence}) we need  variance estimates for \(\mathcal{C}_f\) (see Theorem~\ref{thm:VarianceEstimateCRCase}). It turns out that those estimates can be obtained from the understanding of a sequence of bilinear forms \((\theta_k)_{k>0}\) (see Definition~\ref{def:ZkThetak}) when \(k\) becomes large (see Lemma~\ref{lem:EstimateIntegralthetak} and Lemma~\ref{lem:EstimateIntegralthetakPart2}).

Let \((X,T^{1,0}X)\) be a compact orientable strictly pseudoconvex CR manifold such that the Kohn-Laplacian has closed range. Denote by \(\xi\) a contact form on \(X\) such that the respective Levi form is positive definite and let \(\mathcal{T}\) be the respective Reeb vector field uniquely determined by \(\iota_\mathcal{T}\xi\equiv1\) and \(\iota_\mathcal{T}d\xi\equiv0\). Let  $P\in L^1_{\mathrm{cl}}(X)$ be a  first order formally self-adjoint classical pseudodifferential operator with \(\sigma_P(\xi)>0\) where   $\sigma_P$ denotes the principal symbol of $P$. Denote by \(T_P=\Pi P\Pi\) the corresponding Toeplitz operator. Let \(0<\lambda_1\leq\lambda_2\leq\ldots\) be the positive eigenvalues of \(T_P\) counting multiplicity and let \(f_1,f_2\ldots\in H_{b}^0(X)\cap\mathscr{C}^\infty(X)\) be a respective orthonormal system of eigenvectors. In particular, we have \(T_Pf_j=\lambda_jf_j\), \((f_j,f_j)=1\) and \((f_j,f_\ell)=0\) for all \(j,\ell\in \N\), \(j\neq \ell\).  Let \(\chi\in \mathscr{C}^\infty((\delta_1,\delta_2))\), \(\chi\not\equiv 0\), be a cut-off function for some \(0<\delta_1<\delta_2<1\) and put \(\eta(t):=|\chi(t)|^2\), \(t\in \R\), with \(\tau_j:=\int_\R t^{n+j}\eta(t)dt\) for \(j\in\N_0\). Put \(N_k=|\{j\in \N\mid\lambda_j\leq k\delta_2\}|\).  Let \(\kappa\colon \R\to \C\) be a  function such that \(|\kappa(k)|^2/ (1+|k|^{n})\) is bounded in \(k\). For \(k>0\) put \(A_k=\text{span}\left(\{1\}\cup\{f_j\mid \lambda_j\leq \delta_2k\}\right)\). On \(A_k\) we consider the probability measure \(\mu_k\) induced by the standard complex Gaussian measure \(\mu^G\) on \(\C^{N_k+1}\) and the map \[\C^{N_k+1} \ni a \mapsto a_0\kappa(k)+\sum_{j=1}^{N_k}a_j\chi_k(\lambda_j)f_j \in A_k\]
where \(\chi_k(t)=\chi(k^{-1}t)\).
\begin{definition}\label{def:CF}
	Let \(f\in H^0_b(X)\cap\mathscr{C}^\infty(X)\) such that \(df_p\neq 0\) for all \(p\in X\) with \(f(p)=0\). We define
	\[\mathcal{C}_f\colon \Omega^{2n}(X)\to\C,\,\,\,\,\mathcal{C}_f(\psi)=\frac{1}{2\pi i}\int_{X}\frac{df}{f}\wedge\psi.\]
\end{definition}
We note that under the assumptions on \(f\) in Definition~\ref{def:CF} we have that \(\mathcal{C}_f\) is well defined by Theorem~\ref{thm:OneOverFIntegrableCR}. Furthermore, given \(f\in H^0_b(X)\cap\mathscr{C}^\infty(X)\) such that zero is a regular value of the map \(f\colon X\to \C\) it follows that \(df_p\neq 0\) for all \(p\in X\) with \(f(p)=0\) (cf.\ Remark~\ref{rmk:dfneq0notregularvalue}).

 Since we wish to understand \(\mathcal{C}_f\) as a random variable on \(A_k\), we have to make sure that it is defined on a set \(\tilde{A}_k\subset A_k\) with \(\mu_k(\tilde{A_k})=1\). 
We have the following.
\begin{lemma}\label{lem:UnregularFunctionsAreZeroSetCR}
	Put \(\tilde{A}_k:=\{f\in A_k\mid \text{0 is a regular value of } f\colon X\to \C\}\). For any \(k>0\) with
	\[|\kappa(k)|^2+\sum_{j=1}^{N_k}|\chi_k(\lambda_j)|^2|f_j(x)|^2>0,\,\,\,\forall x\in X\]
	 we have that \(\tilde{A}_k\) is  \(\mu_k\)-measurable with \(\mu_k(A_k\setminus\tilde{A}_k)=0\).
\end{lemma}
Lemma~\ref{lem:UnregularFunctionsAreZeroSetCR} is a direct consequence of the following.
\begin{lemma}\label{lem:singularValueFunctionsMeasureZeroGeneral}
	Let \(M\) be a compact manifold and let \(g_0,g_1,\ldots,g_N\in \mathscr{C}^\infty(M)\) such that \(|g_0(x)|^2+\ldots+|g_N(x)|^2>0\) for all \(x\in M\). Define \(\tau: \C^{N}\to \mathscr{C}^\infty(M)\), \(\tau(a)=\sum_{j=0}^{N}a_jg_j\) and put \[Q:=\{a\in\C^{N+1}\mid \text{0 is a regular vlaue of the map } \tau(a)\colon X\to \C\}. \]
	We have that \(Q\) is open such that \(\C^{N+1}\setminus Q\) has Lebesgue measure zero.
\end{lemma}
\begin{proof}
	Given an open set \(U\subset M\) we define \(Q(U)\subset \C^{N+1} \) by saying \(a\in Q(U)\) if and only if there exists an open neighborhood \(V\) around \(\overline{U}\) such that 0 is a regular value of the map \(\tau(a)|_V\colon V\to \C\).
	The compactness of  \(M\) implies that \(Q(U)\) is open for any open set \(U\subset M\). Since \(Q=Q(M)\) we conclude that \(Q\) is open. Let \(p\in M\) be a point. Since \(|g_0(x)|^2+\ldots+|g_N(x)|^2>0\) there exist an open neighborhood \(U\subset M\) around \(p\), \(j\leq j\leq N\) and a constant \(c>0\) such that \(|g_j(x)|\geq c\) for all \(x\in U\). Without loss of generality assume \(j=0\). We define \(\tau': \C^{N}\to \mathscr{C}^\infty(M)\), \(\tau'(a)=\sum_{j=1}^{N}a_jg_j\). Given \(a'\in\C^N\) we define \(W(a',U)\subset \C\) by saying \(\lambda\in W(a',U)\) if and only if there exists an open neighborhood \(V\) around \(\overline{U}\) such that \(-\lambda\) is a regular value of the map \(\tau'(a)/g_0|_V\colon V\to \C\). Let \(x\in M\) be a point with \(g_0(x)\neq0\) and \(a=(a_0,a')\in \C\times \C^{N}\) with \(\tau(a)(x)=0\). It follows that \(\tau'(a')(x)=-a_0g_0(x)\) and we find 
	\[\left(d\frac{\tau'(a')}{g_0}\right)_x=\frac{(d\tau'(a'))_x}{g_0(x)}-\frac{\tau'(a')}{g_0(x)}\frac{(dg_0)_x}{g_0(x)}=\frac{(d\tau(a))_x}{g_0(x)}.\]
	 From this observation it follows that for any \(a=(a_0,a')\in \C\times \C^{N}\) we have \(a\in Q(U)\) if and only if \(a_0\in W(a',U)\).  By the Theorem of Sard we have that \(\C\setminus W(a',U)\) has Lebesgue measure zero for any \(a'\in \C^N\). Furthermore, since \(Q(U)\) is open we have that \(\C^{N+1}\setminus Q(U)\) is Lebesgue measurable. Since \(\C^{N+1}\setminus Q(U)=\{(a_0,a')\in \C^{N+1}\mid a_0\in \C\setminus W(a',U)\}\) it follows from the Fubini theorem that \(\C^{N+1}\setminus Q(U)\) has Lebesgue measure zero.\\
	 Using the construction above, by the compactness of \(M\) we can find finitely many open sets \(U_1,\ldots, U_m\subset M\) with \(\bigcup_{j=1}^mU_j=M\) such that \(\C^{N+1}\setminus Q(U_j)\) has Lebesgue measure zero for all \(1\leq j\leq m\). Since \(\C^{N+1}\setminus Q=\bigcup_{j=1}^m \C^{N+1}\setminus Q(U_j) \) we conclude that \(\C^{N+1}\setminus Q\) has  Lebesgue measure zero.     
\end{proof}
\subsection{Proof of Theorem~\ref{thm:ExpectationValueCRDistributionIntro}}
We are now going to prove Theorem~\ref{thm:ExpectationValueCRDistributionIntro}. We need to introduce some useful notation before. For \(d\in\N\) we denote by \(\langle\cdot,\cdot\rangle\) the standard inner product on \(\C^{d+1}\) that is
\(\langle a,b\rangle= \sum_{j=0}^da_j\overline{b_j} \) 
for \(a=(a_0,\ldots,a_d), b=(b_0,\ldots,b_d)\in\C^{d+1}\). Furthermore, given \(a\in \C^{d+1}\) we put \(|a|=\sqrt{\langle a, a\rangle}\). For \(k>0\) consider the map
\begin{eqnarray}
F_k\colon X\to \C^{N_k+1},\,\,\,\,
F_k=(\kappa(k),\chi_k(\lambda_1)f_1,\ldots,\chi_k(\lambda_{N_k})f_{N_k}). \end{eqnarray}
Recall that \(\eta_k=|\chi_k|^2\). We put \(S_k(x,y):=\eta_k(T_P)(x,y)=\sum_{j=1}^{N_k}\eta_k(\lambda_j)f_j(x)\overline{f_j(y)}\)   and \(B_k(x):=S_k(x,x)\) for \(x,y\in X\) and \(k>0\). It turns out that  \(|F_k|^2=|\kappa(k)|^2+B_k\) and \(\langle F_k(x),F_k(y)\rangle=|\kappa(k)|^2+S_k(x,y)\) for all \(x,y\in X\) and \(k>0\). We further consider \(dF_k\) component wise, that is,
\[dF_k=(0,\chi_k(\lambda_1)df_1,\ldots,\chi_k(\lambda_{N_k})df_{N_k}).\]
Then \(\langle a, \overline{dF_k}\rangle\) is a one form on \(X\) for any \(a\in \C^{N_k+1}\). We have for example \(\langle F_k(x), (dF_k)_y\rangle=d_yS_k(x,y)\) and for \(f=\langle a,\overline{F_k}\rangle\) we obtain \(df=\langle a,\overline{dF_k}\rangle\).
\begin{lemma}\label{lem:AprioriEstimateForModulusExpactation}
	There exists \(k_0>0\) such that for any \(k\geq k_0\) we find a constant \(C_k>0\) with
	\begin{eqnarray}\label{eq:FirstEstimateEquidistributionCR}
		\int_{\C^{N_k+1}}\frac{|\langle a,\overline{dF_k} \rangle(V) |}{|\langle a, \overline{F_k(x)}\rangle|}d\mu^G(a)\leq C_k|V|.
	\end{eqnarray} 
	 for all \(x\in X\) and \(V\in T_xX\).
\end{lemma}
\begin{proof}
	By Theorem~\ref{thm:ExpansionMain} we can choose \(k_0>0\) such that there exists a constant \(C>0\) with \(|F_k|^2\geq Ck^{n+1}\) for all \(k\geq k_0\). Furthermore, we have that \(\int_{\C^{N_k+1}} |a_0|^{-1}\|a\|d\mu^G(a)<\infty\). Since \(F_k(x)/\|F_k(x)\|\), \(x\in X\), has unit length we obtain from the unitary invariance of \(d\mu^G\) that
	\[\int_{\C^{N_k+1}} \frac{|a|}{|a_0|}d\mu^G(a)=\int_{\C^{N_k+1}} \frac{|a|}{|\langle a, \overline{F_k(x)}/|F_k(x)|\rangle|}d\mu^G(a).\]
	Given \(k\geq k_0\), \(x\in X\) and \(V\in T_xX\) we find 
	\begin{eqnarray}
		\left|\frac{\langle a, \overline{dF_k} \rangle(V)}{\langle a, \overline{F_k(x)}\rangle}\right|&\leq& \frac{|a||VF_k/|F_k(x)||}{|\langle a, \overline{F_k(x)}/|F_k(x)|\rangle|}
	\end{eqnarray}
	It follows that
	\[\int_{\C^{N_k+1}}\left|\frac{\langle a, \overline{dF_k} \rangle(V) }{\langle a, \overline{F_k(x)}\rangle}\right| d\mu^G(a)\leq \int_{\C^{N_k+1}} \frac{|a|}{|a_0|}d\mu^G(a) \frac{|VF_k|}{|F_k(z)|}<\infty.\]
	Since \(|VF_k|^2=\left((d_x\otimes d_yS_k(z,w))\right)|_{x=y}(V\otimes V)\) we conclude from Lemma~\ref{lem:firstDifferentialMainThm} that there is a constant \(\tilde{C}_k>0\) with \(|VF_k|\leq \tilde{C}_k|V|\) for all \(x\in X\) and \(V\in T_xX\).  Hence we can choose a constant \(C_k>0\) independent of \(x\) and \(V\) such that \eqref{eq:FirstEstimateEquidistributionCR} holds.
\end{proof}
Theorem~\ref{thm:ExpectationValueCRDistributionIntro} follows from the next theorem by putting \(\kappa(k)=0\) and replacing \(\psi\) by \(d\psi\) (see Corollary~\ref{cor:ZerodistributionCRMfd}).
\begin{theorem}\label{thm:ExpectationValueCRDistribution}
	Under the same assumptions as in  Theorem~\ref{thm:ProjectiveEmbeddingIntro} and with the notations above there exists \(k_0>0\)
	such that for each \(k\geq k_0\)  we have that \(\mathcal{C}_f\) is a well-defined  current for all \(f\) outside a subset in \(A_k\) of \(\mu_k\)-measure zero, and  given \(\psi\in \Omega^{2n}(X)\) the expectation value
	\[\mathbb{E}_k(\mathcal{C}_f)(\psi):=\int_{A_k} \mathcal{C}_{f}(\psi)d\mu_k(f)\] exists 
	with \(\mathbb{E}_k(\mathcal{C}_f)(\psi)=\int_{X}\beta_k\wedge \psi\) where \(\beta_k \) is the smooth one form given by 
	\begin{eqnarray}\label{eq:DefBetak}
		\beta_k(x)=\frac{d_x\eta_k(T_P)(x,y)|_{x=y}}{2\pi i(|\kappa(k)|^2+\eta_k(T_P)(x,x))}
	\end{eqnarray}
	for all \(x\in X\) and \(k\geq k_0\) where \(\eta=|\chi|^2\).
\end{theorem}
\begin{proof}
	Choose \(k_0>0\) such that \(|F_k(x)|>0\) holds for all \(k\geq k_0\) and \(x\in X\). From Lemma~\ref{lem:UnregularFunctionsAreZeroSetCR} it follows for \(k\geq k_0\) that \(\mathcal{C}_f\) is a well-defined  distribution for all \(f\) outside a subset in \(A_k\) of \(\mu_k\)-measure zero.   For \(x\in X\) let \(U^x_k\) be a unitary transformation of \(\C^{N_k+1}\) with \(U^x_k\overline{F_k(x)}/|F_k(z)|=(1,0,\ldots,0)\) and given \(b\in \C^{N_k+1}\) let us denote by \((U^x_k b)_j\)  the \(j\)-th component of the vector \(U^x_k b\).  We have \((U^x_k b)_0=\langle F_k(x),\overline{b}\rangle/|F_k(x)|\) since \(U^x_k\) is a unitary transformation and basic properties of the inner product in \(\C^{N_k+1}\). We find \(\int_{\C^{N_k+1}}a_j/a_0d\mu^G(a)=0\) for all \(j\geq 1\) . With \(\int_{\C^{N_k+1}}d\mu^G(a)=1\) it follows that
	\begin{eqnarray*}
		\int_{\C^{N_k+1}}\frac{\langle a, \overline{(dF_k)_x}\rangle(V) }{\langle a, \overline{F_k(x)}\rangle} d\mu^G(a)&=&\int_{\C^{N_k+1}}\frac{\langle a, U^{z}_k \overline{VF_k} \rangle }{a_0} d\mu^G(a)\\
		&=&\overline{(U^x_k \overline{(VF_k)(z)}/|F_k(x)|)_0}\\ 
		&=& \frac{\langle VF_k,F_k(x)\rangle}{|F_k(x)|^2}\\
	\end{eqnarray*}
	for all \(k\geq k_0\), \(x\in X\) and \(V\in T_x X\). 
	We have \(|F_k(x)|^2=|\kappa(k)|^2+\eta_k(T_P)(x,x)\) and \(\langle VF_k,F_k(x)\rangle=(d_x \eta_k(T_P)(x,y)|_{x=y})(V)\) for any \(V\in T_xX \). Hence given \(\psi\in \Omega^{2n}(X)\) we conclude
	\begin{eqnarray*}
		\int_{X} 2\pi i \beta_k\wedge \psi&=& \int_{X} 	\left(\int_{\C^{N_k+1}}\frac{\langle a, \overline{dF_k}\rangle}{\langle a, \overline{F_k}\rangle} d\mu^G(a)\right) \wedge \psi\\
		&=& \int_{X} \int_{\C^{N_k+1}}\frac{\langle a, \overline{dF_k}\rangle}{\langle a, \overline{F_k}\rangle} \wedge \psi d\mu^G(a).
	\end{eqnarray*}
By Lemma~\ref{lem:UnregularFunctionsAreZeroSetCR} we have the \(\mathcal{C}_f\) is well defined for all \(f\) in a subset \(\tilde{A}_k\) of \(A_k\) with \(\mu_k(\tilde{A}_k)=1\). Using Lemma~\ref{lem:AprioriEstimateForModulusExpactation} we can apply the Fubini-Tonelli theorem and get
\begin{eqnarray*}
	\int_{X}  \beta_k\wedge \psi &=& \int_{\C^{N_k+1}}\int_X \frac{1}{2\pi i}\frac{\langle a, \overline{dF_k}\rangle}{\langle a, \overline{F_k}\rangle} \wedge \psi d\mu^G(a)\\
		&=&  \int_{A_k}\int_{X} \frac{1}{2\pi i}\frac{df}{f} \wedge \psi d\mu_k(f)\\
		&=& \int_{\tilde{A}_k}\mathcal{C}_f(\psi) d\mu_k(f)= \mathbb{E}_k(\mathcal{C}_f) 
\end{eqnarray*}
	 for all \(\psi\in \Omega^{2n}(X)\). 
\end{proof}
\begin{corollary}\label{cor:ZerodistributionCRMfd}
	Under the same assumptions as in  Theorem~\ref{thm:ProjectiveEmbeddingIntro} and with the notations above there exist \(k_0>0\)
	such that for each \(k\geq k_0\) we have that \(\divisor{f}\) is a well-defined  current for all \(f\) outside a subset in \(A_k\) of \(\mu_k\)-measure zero and the expectation value \(\mathbb{E}_k(\divisor{f})\) defined by   
	\[\left(\mathbb{E}_k(\divisor{f},\psi\right)=\int_{A_k}\left(\divisor{f}, \psi\right) d\mu_k(f),\,\,\,\,\psi\in \Omega^{2n-1}(X)\] exists 
	with \(\mathbb{E}_k(\divisor{f})=d\beta_k\) where \(\beta_k \) is given by~\eqref{eq:DefBetak}. 
\end{corollary}
\begin{proof}
	From Theorem~\ref{thm:PoincareLelongForCR} we obtain \(\divisor{f}=d\mathcal{C}_f\) in the sense of distributions for any \(f\in H^{0}_b(X)\cap C^\infty(X)\) such that zero is a regular value of the map \(f\colon X\to \C\). Hence, using Lemma~\ref{lem:UnregularFunctionsAreZeroSetCR} and integration by parts,   we obtain from Theorem~\ref{thm:ExpectationValueCRDistribution} with \(\kappa(k)=0\) that
	\[\left(\mathbb{E}_k(\divisor{f}),\psi\right)=\mathbb{E}_k(\mathcal{C}_f)(d\psi)=\int_{X}\beta_k\wedge d\psi=\int_{X} d\beta_k\wedge \psi\]
	holds for all \(\psi\in \Omega^{2n-1}(X)\) when \(k\) is large enough.
\end{proof}
\begin{lemma}\label{lem:ExpansionOfBk}
	Under the same assumptions as in  Theorem~\ref{thm:ProjectiveEmbeddingIntro} and with \(\beta_k\) given by~\eqref{eq:DefBetak} we have
	\[k^{-1}\beta_k= \frac{\operatorname{mv}(|\chi|^2)}{2\pi}\frac{\xi}{\sigma_P(\xi)}+O(k^{-1})\]
	in \(\mathscr{C}^\infty\)-topology for \(k\to+\infty\) where
	\begin{eqnarray}\label{eq:defMVchi}
		\operatorname{mv}(|\chi|^2)=\frac{\int_0^{+\infty} t^{n+1}|\chi(t)|^2dt}{\int_0^{+\infty} t^{n}|\chi(t)|^2dt}.
	\end{eqnarray}
\end{lemma}
\begin{proof}
	The statement follows immediately from Theorem~\ref{thm:ExpansionMain} and Lemma~\ref{lem:firstDifferentialMainThm}.
\end{proof}
\begin{corollary}\label{cor:ConvergenceExpactationsCR}
	In the situation of Theorem~\ref{thm:ExpectationValueCRDistribution} there exists \(k_0,C>0\) such that
	\[\left|k^{-1}\mathbb{E}_k(\mathcal{C}_f)(\psi)-\frac{\operatorname{mv}(|\chi|^2)}{2\pi}\int_X\frac{\xi}{\sigma_P(\xi)}\wedge\psi\right|\leq k^{-1}C\|\psi\|_{\mathscr{C}^0(X,\Lambda_\C^{2n}T^*X)}\]
	holds for all \(\psi\in \Omega^{2n}(X)\) and \(k\geq k_0\) where \(\operatorname{mv}(|\chi|^2)\) is given by~\eqref{eq:defMVchi}.
	As a consequence one has
	\[\lim_{k\to\infty} k^{-1}\mathbb{E}_k(\mathcal{C}_f)=\frac{\operatorname{mv}(|\chi|^2)}{2\pi}\frac{\xi}{\sigma_P(\xi)}.\]
\end{corollary}
\begin{proof}
	The result is a direct consequence of Theorem~\ref{thm:ExpectationValueCRDistribution} and Lemma~\ref{lem:ExpansionOfBk}.
\end{proof}
\subsection{Proof of Theorem~\ref{thm:ConvergenceStrongCFIntro}}
We will now prove Theorem~\ref{thm:ConvergenceStrongCFIntro}. We start with the following lemma.
\begin{lemma}\label{lem:Estimateintegral1overa2}
	There is a constant \(C>0\) such that for all \(w\in \C\setminus\{0\}\) we have 
	\[\int_{\C} \frac{d\mu^G(a)}{|a||a-w|}\leq \frac{C}{|w|}.\]
\end{lemma}
\begin{proof}
	We have
	\[\int_{|a|\leq|w|/2}\frac{d\mu^G(a)}{|a||a-w|}\leq \frac{1}{|w|}\int_{|a|\leq|w|/2}\frac{d\mu^G(a)}{|a||a/w-1|}\leq \frac{2}{|w|}\int_{\C}\frac{d\mu^G(a)}{|a|} \]
	and 
	\[\int_{|a|\geq|w|/2}\frac{d\mu^G(a)}{|a||a-w|}\leq \frac{2}{|w|}\int_{|a|\geq|w|/2}\frac{d\mu^G(a)}{|a-w|}\leq \frac{2}{|w|}\int_{\C}\frac{d\mu^G(a)}{|a-w|}.\]
	We have that the function \(w\mapsto \int_{\C}\frac{d\mu^G(a)}{|a-w|}\) is bounded. The claim follows. 
\end{proof}
It turns out that it in order to understand the variance of \(\mathcal{C}_f\), it is useful to consider the following objects.
\begin{definition}\label{def:ZkThetak}
		With the notation and assumptions above for any \(k>0\) such that \(|F_k|^2>0\) put 
		\[(Z_k)_x:=\overline{(dF_k)_x}-\frac{\langle F_k(x),(dF_k)_x\rangle}{|F_k(x)|^2}\overline{F_k(x)}\]
		and \(\theta_k(a,x,y)\colon T_xX\times T_yX\to \C\),
		\begin{eqnarray*}
			\theta_k(a,x,y)(V,W)&:=&	\frac{\langle a,Z_k(V)\rangle \overline{\langle a,Z_k(W)\rangle}}{\langle a,\overline{F_k(x)}\rangle\overline{\langle a,\overline{F_k(y)}\rangle}}	
		\end{eqnarray*}
		for \(x,y\in X\) and \(a\in \C^{N_k+1}\) such that \(\langle a,\overline{F_k(x)}\rangle\overline{\langle a,\overline{F_k(y)}\rangle}\neq 0\).
\end{definition}
By the assumption \(|F_k|^2>0\), for \(x,y\in X\) fixed, it follows that  \(\theta_k(a,x,y)\) is well defined for \(\mu^G\)-almost every \(a\in \C^{N_k+1}\). Recall the definition of \(h^{F_k}\) (see Definition~\ref{def:hfkHf}) that is
\[h^{F_k}(x,y)=\frac{|\langle F_k(x),F_k(y) \rangle|^2}{|F_k(x)|^2|F_k(y)|^2},\,\,\,x,y\in X\] 
for all \(k>0\) with \(|F_k|^2>0\). 
In order to obtain estimates for the variance of \(\mathcal{C}_f\) we need to understand the behavior of \(\theta_k(a,x,y)\) and its integrals with respect to \(a,x,y\) when \(k\) becomes large. In the following two lemmata we consider \(\theta_k\) for sufficiently large \(k>0\). 
\begin{lemma}\label{lem:EstimateIntegralthetak}
	With the notation and assumptions above there exist \(k_0>0\) and constants \(c_k>0\) such that 
	\begin{eqnarray}\label{eq:estimateIntegralModulusThetak}
		\int_{\C^{N_k+1}}|\theta_k(a,x,y)(V,W)|d\mu^G(a)\leq \frac{c_k|V||W|}{\sqrt{1-h^{F_k}(z,w)}}
	\end{eqnarray}
	for all \(k\geq k_0\), \(x,y\in X\) with \(x\neq y\) and all \(V\in T_xX\), \(W\in T_yX\). Furthermore we can choose \(k_0>0\) such that \(\gamma_k(x,y)\colon T_xX\times T_yX\to \C  \), 
	\begin{eqnarray}\label{eq:DefinitionGammak}
		\gamma_k(x,y)(V,W):=\sqrt{1-h^{F_k}(x,y)}\int_{\C^{N_k+1}}\theta_k(a,x,y)(V,W)d\mu^G(a)
	\end{eqnarray}
	 is well defined for all \(k\geq k_0\) and \(x,y\in X\), \(x\neq y\), and that there exist a constant \(C>0\) with  
	\begin{eqnarray}\label{eq:EstimateGammaK}
		&&|\gamma_k(x,y)(V,W)|\\ &&\leq
		\left(\left|\left\langle \frac{Z_k(V)}{|F_k(x)|},\frac{Z_k(W)}{|F_k(y)|}\right\rangle\right|+\frac{1}{1-h^{F_k}(x,y)}\left|\left\langle \frac{Z_k(V)}{|F_k(x)|},  \frac{\overline{F_k(y)}}{{|F_k(y)|}}\right\rangle\right|\left|\left\langle \frac{Z_k(W)}{|F_k(y)|}, \frac{\overline{F_k(x)}}{|F_k(x)|}\right\rangle\right|\right).\nonumber
	\end{eqnarray}
	for all \(k\geq k_0\), \(x,y\in X\), \(x\neq y\) and \(V\in T_xX\), \(W\in T_yX\).
\end{lemma}
\begin{proof}
	Recalling that \(|F_k|^2=|\kappa(k)|^2+B_k\) by Theorem~\ref{thm:ExpansionMain}  we can choose \(k_1>0\) and a constant \(C_0>0\) such that \(C_0^{-1}k^{n+1}\leq|F_k|^2\leq C_0k^{n+1}\) for all \(k\geq k_1\). In addition, we choose \(k_1\) large enough such that for all \(k\geq k_1\) and \(x,y\in X\) we have \(h^{F_k}(x,y)=1\) if and only if \(x=y\) (see Lemma~\ref{lem:PropertiesOfhkForX}). Furthermore, since 
	\[|VF_k|^2=(d_x\otimes d_y \eta_k(T_P)(x,y))|_{x=y} V\otimes V\]
	for \(V\in T_xX\) we find by Lemma~\ref{lem:firstDifferentialMainThm} that there is a constant \(C_1>0\) such that
	\begin{eqnarray}\label{eq:estimateNormTFk}
		|VF_k|\leq C_1k|F_k(x)||V|
	\end{eqnarray}
	for all \(k\geq k_1\), \(x\in X\) and \(V\in T_xX\).
	For \(x,y\in X\), \(V\in T_xX\) and \(W\in T_yX\) put 
	\[I_k(x,y)(V,W):=	\int_{\C^{N_k}+1}|\theta_k(a,x,y)(V, W)|d\mu^G(a).\]
	For \(x,y\in X\) we can choose a unitary transformation \(U^{x,y}_k\) of \(\C^{N_k+1}\) such \(U^{x,y}_k\frac{\overline{F_k(x)}}{|F_k(x)|}=(1,0,\ldots,0)\) and \(U^{x,y}_k\frac{\overline{F_k(y)}}{|F_k(y)|}=(\alpha,\beta,0\ldots,0)\) with \(\alpha,\beta\in \C\), \(|\alpha|^2+|\beta|^2=1\). One observes that \(\alpha|F_k(x)||F_k(y)|=\langle F_k(x),F_k(y)\rangle\) and hence \(|\alpha|^2=h^{F_k}(x,y)\).  Then in addition we can choose \(U^{x,y}_k\) such that \(\beta=\sqrt{1-h^{F_k}(x,y)}\).
	Since \(\|Z_k(V)\|\leq 2 |VF_k|\) for all \(V\in T_xX\) we find with~\eqref{eq:estimateNormTFk} for \(x\neq y\) that
	\begin{eqnarray*}
		I_k(x,y)(V,W)&=&\int_{\C^{N_k}+1}\left|\frac{\langle a,U_k^{x,y}Z_k(V)\rangle\overline{\langle a,U^{x,y}_kZ_k(W)\rangle}}{|F_k(x)||F_k(y)||a_0||\alpha a_0+\beta a_1|}\right|d\mu^G(a)\\
		&\leq& \int_{\C^{N_k}+1}\frac{|a|^2 |Z_k(V)||Z_k(W)|}{|F_k(x)||F_k(y)||a_0||\alpha a_0+\beta a_1|}d\mu^G(a)\\
		&\leq& (2kC_1)^2|V||W| \left(\int_{\C^2} \frac{|a_0|}{|\alpha a_0+\beta a_1|}d\mu^G(a)+\int_{\C^{N_k}+1} \frac{\sum_{j=1}^{N_k+1}|a_j|^2}{|a_0||\alpha a_0+\beta a_1|}d\mu^G(a)\right)
	\end{eqnarray*}
	By some unitary transformation argument we find that \(\int_{\C^2}\frac{|a_0|}{|\alpha a_0+\beta a_1|}d\mu^G(a)\) is uniformly bounded in \(\alpha,\beta\) when \(|\alpha|^2+|\beta|^2=1\). Using Lemma~\ref{lem:Estimateintegral1overa2} we find a constant \(c>0\) independent of \(\beta\) and \(a_1\in \C\) such that 
	\begin{eqnarray}\label{eq:estimateIntegral1overa2}
		\int_{\C}\frac{1}{|a_0||\alpha a_0+\beta a_1|}d\mu^G(a_0)\leq \frac{c}{\beta |a_1|}.
	\end{eqnarray}
	With 
	\[\int_{\C^{N_k}+1} \frac{\sum_{j=2}^{N_k+1}|a_j|^2}{|a_0||\alpha a_0+\beta a_1|}d\mu^G(a)\leq \frac{c}{\beta}\left(\int_{\C}\frac{1}{|a_1|}d\mu^G(a_1)\right)\left(\int_{\C^{N_k-1}}|a'|^2d\mu^G(a')\right)\] we  conclude that for any \(k\geq k_1\) there  is a constant \(c_k>0\) depending on \(k\) such that
	\(I_k(x,y)(V,W)\leq c_k(1+\frac{1}{\beta})|V||W|\) holds  for all \(x,y\in X\), \(x\neq y\), and  \(V\in T_xX\), \(W\in T_yX\). Since \(\beta=\sqrt{1-h^{F_k}(z,w)}\leq 1\) we obtain~\eqref{eq:estimateIntegralModulusThetak}. 
	
	Now we can prove the second part of the statement. For \(x,y\in X\), \(x\neq y\), define \(II_k(x,y)\colon T_xX\times T_yY\to \C\),
	\[II_k(z,w)(V, W):=	\int_{\C^{N_k+1}}\theta_k(a,z,w)(V,W)d\mu^G(a).\] 
	First, from~\eqref{eq:estimateIntegralModulusThetak}  we obtain by the choice of \(k_1\) at the beginning of the proof that \(II_k(x,y)\) is well defined for all \(x,y\in X\), \(x\neq y\), and all \(k\geq k_1\). In the following we will calculate \(II_k(x,y)\). 
	
	Fix \(V\in T_xX \), \(W\in T_yX\) where \(x\neq y\). Put \(\kappa_j=\left(U^{x,y}_kZ_k(V)\right)_j\) and \(\lambda_j=\left(U^{x,y}_kZ_k(W)\right)_j\) where we denote by \((b)_j\) the \(j\)-th component of a vector \(b\in \C^{N_k+1}\). We find 
	\[\kappa_0= \left\langle U^{x,y}_kZ_k(V), U^{x,y}_k\frac{\overline{F_k(x)}}{|F_k(x)|}\right\rangle =\left\langle Z_k(V), \frac{\overline{F_k(x)}}{|F_k(x)|}\right\rangle=0\] 
	and 
	\[	\beta \kappa_1=\alpha \kappa_0+\beta \kappa_1=\left\langle  U^{x,y}_kZ_k(V),  U^{x,y}_k\frac{\overline{F_k(y)}}{|F_k(y)|}\right\rangle=\left\langle Z_k(V),  \frac{\overline{F_k(y)}}{|F_k(y)|}\right\rangle.\]
	Furthermore, we find
	\[\lambda_0=\left\langle U^{x,y}_kZ_k(W), U^{x,y}_k\frac{\overline{F_k(x)}}{|F_k(x)|}\right\rangle=\left\langle Z_k(W), \frac{\overline{F_k(x)}}{|F_k(x)|}\right\rangle\]
	and 
	\[\alpha\lambda_0+\beta\lambda_1=\left\langle U^{x,y}_kZ_k(W), U^{x,y}_k\frac{\overline{F_k(y)}}{|F_k(y)|}\right\rangle= \left\langle U^{x,y}_kZ_k(W), U^{x,y}_k\frac{\overline{F_k(y)}}{|F_k(y)|}\right\rangle=0. \]
	Summing up we have 
	\begin{eqnarray*}
		\kappa_0&=&0,\\
		\kappa_1&=&\frac{1}{\beta}\left\langle Z_k(V),  \frac{\overline{F_k(y)}}{|F_k(y)|}\right\rangle,\\
		\lambda_0&=& \left\langle Z_k(W), \frac{\overline{F_k(x)}}{|F_k(x)|}\right\rangle, \\
		\lambda_1&=&-\frac{\alpha}{\beta} \left\langle Z_k(W), \frac{\overline{F_k(x)}}{|F_k(x)|}\right\rangle.
	\end{eqnarray*}
	Furthermore, we observe that 
	\[\sum_{j=0}^{N_k}\kappa_j\overline{\lambda_j}=\langle U^{z,w}_kZ_k(V),U^{z,w}_kZ_k(W)\rangle=\langle Z_k(V),Z_k(W)\rangle.\]
	Put \(c'=\int_{\C}|a|^2d\mu^G(a)\). Since \(x\neq y\) we have \(\beta\neq 0\). Then \(\int_{\C^{N_k+1}}\frac{a_j\overline{a_l}}{a_0\overline{(\alpha a_0+\beta a_1)}}d\mu^G(a)=0\) for all \(0\leq j,l\leq N_k\) with \(j\neq l\) and \((j,l)\notin\{(1,0),(0,1)\}\). We check that
	\begin{eqnarray*}
		II_k(x,y)(V,W)&=&\int_{\C^{N_k}+1}\frac{\langle a,U^{x,y}_kZ_k(V)\rangle\overline{\langle a,U^{x,y}Z_k(W)\rangle}}{|F_k(x)||F_k(y)|a_0\overline{(\alpha a_0+\beta a_1)}}d\mu^G(a)\\
		&=& \left(\int_{\C^2}\frac{c'd\mu^G(a)}{a_0\overline{(\alpha a_0+\beta a_1)}}\right)\sum_{j=2}^{N_k}\frac{\kappa_j\overline{\lambda_j}}{|F_k(x)||F_k(y)|}\\ &&+ \int_{\C^2}\frac{|a_0|^2\kappa_0\overline{\lambda_0}+|a_1|^2\kappa_1\overline{\lambda_1} +a_0\overline{a_1}\kappa_0\overline{\lambda_1}+a_1\overline{a_0}\kappa_1\overline{\lambda_0} }{|F_k(x)||F_k(y)|a_0\overline{(\alpha a_0+\beta a_1)}}d\mu^G(a)\\
		&=& \left(\int_{\C^2}\frac{c'd\mu^G(a)}{a_0\overline{(\alpha a_0+\beta a_1)}}\right)\left\langle \frac{Z_k(V)}{|F_k(x)|},\frac{Z_k(W)}{|F_k(y)|}\right\rangle\\
		&&+\left(\int_{\C^2}\frac{c'd\mu^G(a)}{a_0\overline{(\alpha a_0+\beta a_1)}}\right)\frac{\overline{\alpha}}{\beta^2}\left\langle \frac{Z_k(V)}{|F_k(x)|}, \frac{\overline{F_k(y)}}{|F_k(y)|}\right\rangle\overline{ \left\langle \frac{Z_k(W)}{|F_k(y)|}, \frac{\overline{F_k(x)}}{|F_k(x)|}\right\rangle}\\
		&&+\left(\int_{\C^2}\frac{a_1\overline{(\beta a_0-\alpha a_1)}} {a_0\overline{(\alpha a_0+\beta a_1)}}d\mu^G(a)\right)\frac{1}{\beta^2}\left\langle \frac{Z_k(V)}{|F_k(x)|},  \frac{\overline{F_k(y)}}{{|F_k(y)|}}\right\rangle\overline{\left\langle \frac{Z_k(W)}{|F_k(y)|}, \frac{\overline{F_k(x)}}{|F_k(x)|}\right\rangle}.
	\end{eqnarray*}
	With~\eqref{eq:estimateIntegral1overa2} and since
	\[\left|\int_{\C^2}\frac{a_1\overline{(\beta a_0-\alpha a_1)}} {a_0\overline{(\alpha a_0+\beta a_1)}}d\mu^G(a)\right|\leq \int_{\C^2}\frac{|a_1|^2+|a_0||a_1|} {|a_0||\alpha a_0+\beta a_1|}d\mu^G(a)\leq \frac{c}{\beta}c'+\int_{\C^2}\frac{|a_1|}{|\alpha a_0+\beta a_1|}d\mu^G(a)\]
	we conclude with \(\beta=\sqrt{1-h^{F_k}(x,y)}\) that there is a constant \(C_2>0\) independent of \(k\), \(x\), \(y\) and \(V\in T_xX\), \(W\in T_yX\) such that
	\(\gamma_k(x,y)\colon T_xX\times T_yX\to \C\) defined by \(\gamma_k(x,y)=\sqrt{1-h^{F_k}(x,y)}II_k(x,y)\) satisfies
	\begin{eqnarray*}
		&&|\gamma_k(x,y)(V,W)|\\ &&\leq
		C_2\left(\left|\left\langle \frac{Z_k(V)}{|F_k(x)|},\frac{Z_k(W)}{|F_k(y)|}\right\rangle\right|+\frac{1}{1-h^{F_k}(x,y)}\left|\left\langle \frac{Z_k(V)}{|F_k(x)|},  \frac{\overline{F_k(y)}}{{|F_k(y)|}}\right\rangle\right|\left|\left\langle \frac{Z_k(W)}{|F_k(y)|}, \frac{\overline{F_k(x)}}{|F_k(x)|}\right\rangle\right|\right).\nonumber
	\end{eqnarray*}
\end{proof}

\begin{lemma}\label{lem:EstimateIntegralthetakPart2}
	With the notation and assumptions above 
		given \(0<\varepsilon<1\) and a volume form \(dV\) on \(X\times X\) there exist \(C,k_0>0\) such that
	\begin{eqnarray}\label{eq:estimateIntegralModulusIntegralThetak}
		\int_{X\times X}\left|\int_{\C^{N_k+1}}\theta_kd\mu^G\right|dV\leq C k^{2-n-1}(|\kappa(k)|^2 + k^{\varepsilon})
	\end{eqnarray}
	for all \(k\geq k_0\).
\end{lemma}
\begin{proof}
	As in the proof of Lemma~\ref{lem:EstimateIntegralthetak} we can find  \(k_1,C_0,C_1>0\)  such that
	\begin{eqnarray}
		C_0^{-1}k^{n+1}\leq|F_k|^2\leq C_0k^{n+1}
	\end{eqnarray}
	  and
		\begin{eqnarray}\label{eq:estimateNormTFkPart2}
		|VF_k|\leq C_1k|F_k(x)||V|
	\end{eqnarray}
	for all \(k\geq k_1\), \(x\in X\) and \(V\in T_xX\).
	In addition, we choose \(k_1\) large enough such that for all \(k\geq k_1\) and \(x,y\in X\) we have \(h^{F_k}(x,y)=1\) if and only if \(x=y\)  (see Lemma~\ref{lem:PropertiesOfhkForX}). Recall that given \(x\neq y\), \(V\in T_xX\), \(W\in T_yX\) and \(k\geq k_1\) we have
	\[\int_{\C^{N_k+1}}\theta_k(a,x,y)(V,W)d\mu^G(a)=\frac{\gamma_k(x,y)(V,W)}{\sqrt{1-h^{F_k}(x,y)}}\]
	where \(\gamma_k(x,y)(V,W)\) defined in~\eqref{eq:DefinitionGammak} satisfies the estimate~\eqref{eq:EstimateGammaK}.
	 Using~\eqref{eq:EstimateGammaK} we would like to estimate \(\left|\int_{\C^{N_k+1}}\theta_k(a,x,y)d\mu^G(a)\right|\) and its integral over \(X\times X\). First we will show that \(k^{-2}|\gamma_k(x,y)|\) is uniformly bounded in \(x\), \(y\) and \(k\). Then we will estimate \(|\gamma_k(x,y)|\) outside the diagonal in \(X\times X\) interms of \(\eta(T_P)(x,y)\).  As consequence we obtain~\eqref{eq:estimateIntegralModulusIntegralThetak} using Lemma~\ref{lem:EstimateIntegralhk} and Lemma~\ref{lem:IntegralSzegoDerivativeXtimesX}.

	One can check that 
	\begin{eqnarray*}
		\left\langle \frac{Z_k(V)}{|F_k(x)|},  \frac{\overline{F_k(y)}}{{|F_k(y)|}}\right\rangle&=&\left\langle \frac{\overline{VF_k}}{|F_k(x)|},  \frac{\overline{F_k(y)}}{{|F_k(y)|}}\right\rangle -\left\langle \frac{\overline{VF_k}}{|F_k(x)|},  \frac{\overline{F_k(x)}}{{|F_k(x)|}}\right\rangle\left\langle \frac{\overline{F_k(x)}}{|F_k(x)|},  \frac{\overline{F_k(y)}}{|F_k(y)|}\right\rangle \\
		&=&\left\langle \frac{\overline{VF_k}}{|F_k(x)|},  \frac{\overline{F_k(y)}}{{|F_k(y)|}}-\left\langle \frac{\overline{F_k(x)}}{|F_k(x)|},  \frac{\overline{F_k(y)}}{|F_k(y)|}\right\rangle\frac{\overline{F_k(x)}}{|F_k(x)|}\right\rangle
	\end{eqnarray*}
	and
	\begin{eqnarray*}
		\left|\frac{\overline{F_k(y)}}{{|F_k(y)|}}-\left\langle \frac{\overline{F_k(x)}}{|F_k(x)|},  \frac{\overline{F_k(y)}}{|F_k(y)|}\right\rangle\frac{\overline{F_k(x)}}{|F_k(x)|}\right|^2=1-h^{F_k}(x,y).
	\end{eqnarray*}
	Hence it follows
	\[\left|\left\langle \frac{Z_k(V)}{|F_k(x)|},  \frac{\overline{F_k(y)}}{{|F_k(y)|}}\right\rangle\right|\leq \sqrt{1-h_k(x,y)}\frac{|VF_k|}{|F_k(x)|}\leq \sqrt{1-h_k(z,w)}C_1k|V| \]
	and in the same way we obtain
	\[\left|\left\langle \frac{Z_k(W)}{|F_k(y)|},  \frac{\overline{F_k(x)}}{{|F_k(x)|}}\right\rangle\right|\leq \sqrt{1-h_k(z,w)}C_1k|W|.\]
	With
	\[\left|\left\langle \frac{Z_k(V)}{|F_k(x)|},\frac{Z_k(W)}{|F_k(y)|}\right\rangle\right|\leq \frac{4|VF_k||WF_k|}{|F_k(x)||F_k(y)|}\leq 4(C_1k)^2|V||W|\]
	we conclude that there exists a constant \(C_3>0\) such that \(|\gamma_k(x,y)(V,W)|\leq C_3k^2|V||W|\)
	for all \(k\geq k_1\), \(x,y\in X\), \(x\neq y\), and \(V\in T_xX\), \(W\in T_xY\). The metric on TX induces a metric on the bilinear forms from  \(T_xX\times T_yX\) to \(\C\). It follows that we can choose the constant \(C_3>0\) such that
		\begin{eqnarray}\label{eq:suphkgeq1-delta}
		|\gamma_k(x,y)|\leq C_3k^2
	\end{eqnarray}	
for all \(k\geq k_1\), \(x,y\in X\), \(x\neq y\).

In order to obtain estimates for the integral of \(\left|\int_{\C^{N_k+1}}\theta_k(a,x,y)d\mu^G(a)\right|\) over \(X\times X\) using Lemma~\ref{lem:IntegralSzegoDerivativeXtimesX} we need to estimate \(|\gamma_k(x,y)|\) in terms of \(\eta_k(T_P)(x,y)\).

With \(|\langle VF_k,F_k(x)\rangle|\leq kC_1|F_k(x)|^2|V|\) (see~\ref{eq:estimateNormTFkPart2}) we find
	\begin{eqnarray}\label{eq:EstimateZkVZkW}
		|\langle Z_k(V),Z_k(W)\rangle|&\leq&|\langle VF_k,WF_k\rangle|+(C_1k)^2|\langle F_k(x),F_k(y)\rangle|V||W|\\
		&&+C_1k|\langle VF_k,F_k(y)\rangle||W| +C_1k|V||\langle F_k(x),WF_k\rangle|\nonumber
	\end{eqnarray}
	and  
	\begin{eqnarray}\label{eq:EstimateZkVFk}
	|\langle Z_k(V),\overline{F_k(y)}\rangle|\leq |\langle VF_k,F_k(y)\rangle|+k|\langle F_k(x),F_k(y)\rangle|.
	\end{eqnarray}
	We note that
	\begin{eqnarray*}
		\langle VF_k,WF_k\rangle= d_x\otimes d_y \eta_k(T_p)(x,y). 
	\end{eqnarray*}
	Then, by the compactness of \(X\), we can find finitely many partial differential operators \(Q_j\colon \mathscr{C}^\infty(X)\to  \mathscr{C}^\infty(X)\), \(1\leq j\leq r\) of order at most two such that 
	\begin{eqnarray}\label{eq:EstimateVFkWFkByPDO}
	|\langle VF_k,WF_k\rangle|\leq \sum_{j=1}^r|Q_j\eta_k(T_P)(x,y)||V||W|
	\end{eqnarray}
	 for all \(k\geq k_1\), \(x,y\in X\), \(V\in T_xX\) and \(W\in T_yX\). Since
	\begin{eqnarray*}
		\langle VF_k,F_k(y)\rangle = d_x \eta_k(T_p)(x,y) \text{ and } \langle F_k(x),WF_k\rangle= d_y \eta_k(T_p)(x,y)
	\end{eqnarray*} 
	we can find finitely many partial differential operators \(L_j\colon \mathscr{C}^\infty(X)\to  \mathscr{C}^\infty(X)\), \(1\leq j\leq s\) of order at most one such that
	\begin{eqnarray}\label{eq:EstimateVFkFkByPDO}
		|\langle VF_k,F_k(x)\rangle||W|+|\langle F_k(x),WF_k\rangle||V|\leq \sum_{j=1}^s|L_j\eta_k(T_P)(x,y)||V||W|
	\end{eqnarray}
	for all \(k\geq k_1\), \(x,y\in X\), \(V\in T_xX\) and \(W\in T_yX\).  Put
	\[\gamma'_{k}(x,y):= k^2|\kappa(k)|^2+ k^2|\eta_k(T_P)(x,y)|+k\sum_{j=1}^s|L_j\eta_k(T_P)(x,y)| +\sum_{j=1}^r|Q_j\eta_k(T_P)(x,y)|.\]
	Recall that \(\langle F_k(x),F_k(y)\rangle=|\kappa(k)|^2+\eta_k(T_P)(x,y)\), \(\langle F_k(x),F_k(y)\rangle\leq |F_k(x)||F_k(y)|\), \(|\langle VF_k,F_k(y)\rangle|\leq C_1k|F_k(x)||F_k(y)|\) and \(|F_k|^2\geq C_0^{-1}k^{n+1}\) for all \(k\geq k_1\) and \(x,y\in X\).  
	Then from~\eqref{eq:EstimateGammaK},  \eqref{eq:EstimateZkVZkW}, \eqref{eq:EstimateZkVFk}, \eqref{eq:EstimateVFkWFkByPDO} and \eqref{eq:EstimateVFkFkByPDO}  we have that given any \(\delta>0\) we can find a constant \(C_4>0\) such that
	\[|\gamma_k(x,y)|\leq C_4k^{-n-1}\gamma'_k(x,y).\]
	for all \(x,y\in X\) with \(h^{F_k}(x,y)\leq 1-\delta\) and all \(k\geq k_1\).
	Given \(0< \varepsilon <1\) we can choose \(k_2\geq k_1\) to conclude from Lemma~\ref{lem:IntegralSzegoDerivativeXtimesX} that there exists \(C_5>0\) with
	\begin{eqnarray}\label{eq:Integralhkleq1-delta}
		\int_{h^{F_k}<1-\delta} |\gamma_k| dV\leq C_4k^{-n-1}\int_{X\times X}\gamma'_k dV\leq C_5(|\kappa(k)|^2k^{2-n-1}+k^{2+\varepsilon-n-1})
	\end{eqnarray}	
	for all \(k\geq k_2\). 
	From~\eqref{eq:suphkgeq1-delta},~\eqref{eq:Integralhkleq1-delta} and Lemma~\ref{lem:EstimateIntegralhk} we conclude that there exist a constant \(C>0\) and \(k_0>0\) such that
	\[\int_{X\times X}\left|\int_{\C^{N_k+1}}\theta_k d\mu^G\right|dV= \int_{X\times X} \frac{|\gamma_k|}{\sqrt{1-h^{F_k}}} dV\leq C k^{2-n-1}(|\kappa(k)|^2 + k^{\varepsilon})\]
	for all \(k\geq k_0\). 
\end{proof}
Using Lemma~\ref{lem:EstimateIntegralthetak} and Lemma~\ref{lem:EstimateIntegralthetakPart2} we obtain variance estimate for \(\mathcal{C}_f\).
\begin{theorem}\label{thm:VarianceEstimateCRCase}
	Under the assumptions of Theorem~\ref{thm:ProjectiveEmbeddingIntro} and with the notations above 
	put
	\[\mathbb{V}_k(\mathcal{C}_f(\psi)):=\int_{A_k}\left|\mathcal{C}_f(\psi)-\mathbb{E}_k(\mathcal{C}_f)(\psi)\right|^2d\mu_k(f),\,\,\, \psi\in \Omega^{2n}(X).\]
	Given \(0<\varepsilon<1\) there exist \(k_0>0\) and \(C>0\) such that \(\mathbb{V}_k(\mathcal{C}_f(\psi))\) is well-defined with
	\[\mathbb{V}_k(\mathcal{C}_f(\psi))\leq C k^{2-n-1}(|\kappa(k)|^2 + k^{\varepsilon})\|\psi\|^2_{\mathscr{C}^0(X,\Lambda_\C^{2n}T^*X)}\]  for all \(k\geq k_0\) and all \(\psi\in \Omega^{2n}(X)\). 
\end{theorem}
\begin{proof}
	Let \(dV\) be a volume form on \(X\times X\).
	Choose \(k_1>0\) such that \(|F_k|>0\) for all \(k\geq k_1\) and that the statements in Lemma~\ref{lem:EstimateIntegralthetak} hold true. We have that \(\theta_k\) as in Definition~\ref{def:ZkThetak} induces  an alternating multiliniear form \(\hat{\theta}_k(a,x,y)\) on \(T_{(x,y)}X\times X\) for all \(x,y\in X\) and \(\mu^G\)-almost every  \(a\in\C^{N_k+1}\). Define \(p_1,p_2\colon X\times X\to X\), \(p_1(x,y)=x\), \(p_2(x,y)=y\)  and put
	\begin{eqnarray}
		R_k(\psi):=\int_{\C^{N_k+1}}\int_{X\times X}\hat{\theta}_k\wedge  p_1^*\psi \wedge \overline{p_2^*\psi} d\mu^G(a), \,\,\,\,\psi\in \Omega^{2n}(X) .
	\end{eqnarray} 
	By Lemma~\ref{lem:EstimateIntegralthetak} and the Fubini-Tonelli theorem we have that \(R_k(\psi)\) exists for any \(\psi\in \Omega^{2n}(X)\) and \(k\geq k_1\) with
	\begin{eqnarray*}
		|R_k(\psi)|=\left|\int_{X\times X} \int_{\C^{N_k+1}} \hat{\theta}_k\wedge  p_1^*\psi \wedge \overline{p_2^*\psi} \right| \leq C_1 \left(\int_{X\times X} \left|\int_{\C^{N_k+1}} \theta_k d\mu^G\right| dV\right )\|\psi\|^2_{\mathscr{C}^0(X,\Lambda_\C^{2n}T^*X)},
	\end{eqnarray*} 
where \(C_1>0\) is a constant independent of \(k\) and \(\psi\).
	Hence, using Lemma~\ref{lem:EstimateIntegralthetakPart2}, we conclude that for any \(0<\varepsilon<1\) there exist \(k_0\geq k_1\) and \(C>0\) such that
	\begin{eqnarray}\label{eq:EstimateRkCRCase}
		|R_k(\psi)|\leq  C k^{2-n-1}(|\kappa(k)|^2 + k^{\varepsilon})\|\psi\|^2_{\mathscr{C}^0(X,\Lambda_\C^{2n}T^*X)}
	\end{eqnarray} 
	for all \(\psi\in\Omega^{2n}(X)\) and \(k\geq k_0\).
	Recall that for any \(k\geq k_0\) there exists \(\tilde{A}_k\subset A_k\) with \(\mu_k(\tilde{A}_k)=1\) such that for any \(f\in \tilde{A}_k\) we have
	\begin{eqnarray}
		\mathcal{C}_f(\psi)=\frac{1}{2\pi i}\int_X\frac{df}{f}\wedge \psi,\,\,\,\,\,\psi\in \Omega^{2n}(X). 
	\end{eqnarray}
	 By Theorem~\ref{thm:ExpectationValueCRDistribution} we have \(\mathbb{E}_k(\mathcal{C}_f)(\psi)=\int_X \beta_k\wedge \psi\) for all \(\psi\in \Omega^{2n}(X)\)  with
	\begin{eqnarray}
		\beta_k=\frac{\langle dF_k,F_k\rangle}{2\pi i|F_k|^2}.
	\end{eqnarray}
Writing \(f=\langle a,\overline{F_k}\rangle\), \(a\in\C^{N_k+1}\), we find
\begin{eqnarray}\label{eq:RndVariableMinusExpactationInZk}
	\frac{df}{f}-2\pi i\beta_k = \frac{\langle a,\overline{dF_k}\rangle-\frac{\langle dF_k,F_k\rangle}{|F_k|^2}\langle a,\overline{F_k}\rangle }{\langle a,\overline{F_k}\rangle}=\frac{\langle a,Z_k\rangle}{\langle a,\overline{F_k}\rangle}
\end{eqnarray}
for all \(k\geq k_0\) and \(\mu^G\)-almost every  \(a\in\C^{N_k+1}\).
 Given \(\psi\in \Omega^{2n}(X)\) and \(k\geq k_0\) we obtain from the definition of \(\theta_k\) and~\eqref{eq:RndVariableMinusExpactationInZk} that
 \begin{eqnarray*}
 	R_k(\psi)&=& \int_{\C^{N_k+1}}\int_{X\times X}  p_1^*\left(\frac{\langle a,Z_k\rangle}{\langle a,\overline{F_k}\rangle}\wedge \psi\right)\wedge\overline{ p_2^*\left(\frac{\langle a,Z_k\rangle}{\langle a,\overline{F_k}\rangle}\wedge \psi\right)}d\mu^G(a)\\
 	&=& \int_{\C^{N_k+1}}\left|\int_{X} \frac{\langle a,Z_k\rangle}{\langle a,\overline{F_k}\rangle}\wedge \psi\right|^2d\mu^G(a)\\
 	&=& \int_{A_k}\left|\int_{X} \left(\frac{df}{f}-2\pi i\beta_k\right)\wedge\psi \right|^2d\mu_k(f)\\
 	&=& 4\pi^2 \mathbb{V}_k(\mathcal{C}_f(\psi)).
 \end{eqnarray*}
Then the claim follows from~\eqref{eq:EstimateRkCRCase}. 
\end{proof}
Theorem~\ref{thm:ConvergenceStrongCFIntro} follows from the next result by taking \(\kappa(k)\equiv 0\) and replacing \(\psi\) by \(d\psi\) (see Corollary~\ref{cor:ZerodistributionCRMfdSequence}). Consider \(k\in\N\) and put 
\(A_\infty=\Pi_{k=1}^\infty A_k\), \(d\mu_{\infty}=\Pi_{k=1}^\infty d\mu_k\). We have the following.
\begin{theorem}\label{thm:ConvergenceStrongCF}
	 Under the same assumptions as in Theorem~\ref{thm:ProjectiveEmbeddingIntro} and with the notations above, assuming that \(|\kappa(k)|^2/(1+|k|^{n-1})|\) is bounded, there exists \(k_0>0\) such that for  \(\mu_\infty\)-almost every \(f=(f_k)_{k\in \N}\in A_{\infty}\) we have
	that \(\mathcal{C}_{f_k}(\psi)\) is well defined for all \(\psi\in \Omega^{2n}(X)\)  and \(k\geq k_0\). Furthermore, given \(\psi\in \Omega^{2n}(X)\),  we have for  \(\mu_\infty\)-almost every \(f=(f_k)_{k\in \N}\in A_{\infty}\) that
	\[\lim_{k\to\infty}  k^{-1}\mathcal{C}_{f_k}(\psi) = \frac{\operatorname{mv}(|\chi|^2)}{2\pi}\int_{X}\frac{\xi}{\sigma_P(\xi)}\wedge\psi\]
   where \(\operatorname{mv}(|\chi|^2)\) is given by~\eqref{eq:defMVchi}. In addition,  there exists \(C>0\) such that for any \(k\geq k_0\) one has
   \begin{eqnarray}\label{eq:ProbabilityEstimateCf}
   	\,\,\,\,\,\,\mu_k\left(\left\{f\in A_k\colon\left|k^{-1}\mathcal{C}_f(\psi)-\frac{\operatorname{mv}(|\chi|^2)}{2\pi}\int_{X}\frac{\xi}{\sigma_P(\xi)}\wedge\psi\right|\geq\frac{\|\psi\|_{\mathscr{C}^0(X,\Lambda^{2n-1}_\C T^*X)}}{\sqrt{k}} \right\}\right)\leq \frac{C}{\sqrt{k}}
   \end{eqnarray}
   for all \(\psi\in \Omega^{2n-1}(X)\).
\end{theorem}
\begin{proof}
	 Given  \(\psi\in \Omega^{2n}(X)\) and \(k>0\) put 
	\[R_k(\psi):=\int_{A_k}\left|k^{-1}\mathcal{C}_f(\psi)-\frac{\operatorname{mv}(|\chi|^2)}{2\pi}\int_{X} \frac{\xi}{\sigma_P(\xi)} \wedge \psi\right|^2d\mu_k(f).\]
	We will show that there exist \(k_0\in\N\) and a constant \(C>0\) such that \(R_k(\psi)\) exists for all \(k\geq k_0\) with \(R_k(\psi)\leq Ck^{-\frac{3}{2}}\|\psi\|^2_{\mathscr{C}^0(X,\Lambda^{2n}_\C T^*X)}\) for all \(k\geq k_0\) and \(\psi\in \Omega^{2n}(X)\). From the Markov inequality, the last part of the claim follows. Furthermore, considering \(k\in\N\), it follows that \(\sum_{k=k_0}^\infty R_k(\psi)\leq C_0 \|\psi\|^2_{\mathscr{C}^0(X,\Lambda^{2n}_\C T^*X)} \) for some constant \(C_0\) (independent of \(\psi\)) which proves the second part of the claim. The first part of the claim follows immediately from Theorem~\ref{thm:ExpectationValueCRDistribution}.
	
	Given  \(\psi\in \Omega^{2n}(X)\) and \(k\in \N\) put 
	\[R'_k(\psi):=\left|\frac{1}{k}\int_X\beta_k\wedge\psi-\frac{\operatorname{mv}(|\chi|^2)}{2\pi}\int_{X} \frac{\xi}{\sigma_P(\xi)} \wedge \psi\right|^2\]
	with \(\beta_k\) as in~\eqref{eq:DefBetak}.
	 From Theorem~\ref{thm:ExpectationValueCRDistribution} and Corollary~\ref{cor:ConvergenceExpactationsCR} it follows that there exist \(k_1\in\N\) and \(C_1>0\) such that \(R'_k(\psi)\leq C_1k^{-2}\|\psi\|^2_{\mathscr{C}^0(X,\Lambda^{2n}_\C T^*X)}  \)
	holds for all \(k\geq k_0\) and  \(\psi\in \Omega^{2n}(X)\).
	With Theorem~\ref{thm:ExpectationValueCRDistribution}  and \(\mathbb{V}_k\) as in Theorem~\ref{thm:VarianceEstimateCRCase} we obtain 
	\[R_k(\psi)\leq 2k^{-2}\mathbb{V}_k(\mathcal{C}_f(\psi))+2R'_k(\psi)\int_{A_k}d\mu_k\] 
	Then it follows from Theorem~\ref{thm:VarianceEstimateCRCase} with \(\varepsilon=\frac{1}{2}\) that there are \(k_0\in \N\) and a constant \(C>0\) such that  \(R_k(\psi)\leq Ck^{-\frac{3}{2}}\|\psi\|^2_{\mathscr{C}^0(X,\Lambda^{2n}_\C T^*X)}\) for all \(k\geq k_0\) and \(\psi\in \Omega^{2n}(X)\).
\end{proof}
\begin{corollary}\label{cor:ZerodistributionCRMfdSequence}
	Under the same assumptions as in Theorem~\ref{thm:ProjectiveEmbeddingIntro}, with the notations above and assuming \(\kappa(k)\equiv 0\) there exists \(k_0>0\) such that for \(\mu_\infty\)-almost every \(f=(f_k)_{k\in \N}\in A_{\infty}\) we have that \((\divisor{f_k}, \psi)\) is well defined for all \(\psi\in \Omega^{2n-1}(X)\) and \(k\geq k_0\). Furthermore, given any  \(\psi\in \Omega^{2n-1}(X)\), we have for \(\mu_\infty\)-almost every \(f=(f_k)_{k\in \N}\in A_{\infty}\) that
	\[\lim_{k\to\infty} \left(k^{-1}\divisor{f_k}, \psi\right) = \frac{\operatorname{mv}(|\chi|^2)}{2\pi}\int_{X}d\alpha_P\wedge\psi\]
	  where \(\alpha_P=(\sigma_P(\xi))^{-1}\xi\).
	  In addition, there exists \(C>0\) such that for any \(k\geq k_0\) one has
	  \begin{eqnarray*}
	  	\mu_k\left(\left\{f\in A_k\colon\left|\left(k^{-1}\divisor{f}, \psi\right)-\frac{\operatorname{mv}(|\chi|^2)}{2\pi}\int_{X}d\alpha_P\wedge\psi\right|\geq\frac{\|\psi\|_{\mathscr{C}^1(X,\Lambda^{2n-1}_\C T^*X)}}{\sqrt{k}} \right\}\right)\leq \frac{C}{\sqrt{k}}
	  \end{eqnarray*}
	  for all \(\psi\in \Omega^{2n-1}(X)\).
\end{corollary}
\begin{proof}
	From Theorem~\ref{thm:PoincareLelongForCR} we obtain \(\divisor{f}=d\mathcal{C}_f\) in the sense of distributions for any \(f\in H^{0}_b(X)\cap C^\infty(X)\) such that zero is a regular value of the map \(f\colon X\to \C\). Then the first part of the claim follows from Theorem~\ref{thm:ExpectationValueCRDistribution}. Fix \(\psi\in \Omega^{2n-1}(X)\).  Using Lemma~\ref{lem:UnregularFunctionsAreZeroSetCR} and integration by parts, we obtain from Theorem~\ref{thm:ConvergenceStrongCF} with \(\kappa(k)\equiv0\) that for \(\mu_\infty\)-almost every \(f=(f_k)_{k\in \N}\in A_{\infty}\) we have that 
	\[\lim_{k\to\infty} \left(k^{-1}\divisor{f_k}, \psi\right)=\lim_{k\to\infty}  k^{-1}\mathcal{C}_{f_k}(d \psi)=\frac{\operatorname{mv}(|\chi|^2)}{2\pi}\int_{X}\alpha_P\wedge d\psi=\frac{\operatorname{mv}(|\chi|^2)}{2\pi}\int_{X}d\alpha_P\wedge \psi.\]
	Furthermore, replacing \(\psi\) by \(d\psi\) in~\eqref{eq:ProbabilityEstimateCf} leads to the last part of the claim.
\end{proof}

\section{Some Remarks on Complex Manifolds with Boundary}\label{sec:RemarksMfdBoundary}
In this section we will recall basic definitions and notations for complex manifolds with boundary and formulate  Theorem~\ref{thm:PoincareLelongFormula} in this set-up.

Let \(M\) be a real smooth \(N\)-dimensional manifold with smooth boundary. We denote by \(\text{int}(M)\) the set of points in \(M\) which have an open neighborhood homeomorphic to \(\R^N\) and put \(bM:=M\setminus \text{int}(M)\) to denote the boundary of \(M\). We note that \(\text{int}(M)\) is a smooth \(N\)-dimensional manifold without boundary and that \(bM\), in case \(bM\neq\emptyset\), is a smooth \((N-1)\)-dimensional smooth submanifold of \(M\). We call \(V\) a complex domain with boundary if there exists  a complex manifold \(Y\) (without boundary) and a domain \(G\subset Y\) with smooth boundary such that \(V=\overline{G}\). Note that \(V\) is automatically a smooth manifold with boundary with a complex structure on \(\text{int}(V)\).
\begin{definition}\label{def:ComplexManfioldwBoundary}
	Let \(M\) be a real smooth connected manifold with boundary. A complex structure on \(M\) is a complex structure on \(\text{int}(M)\) such that for any point \(p\in bM\) there exist an open neighborhood \(U\subset M\) around \(p\) and a diffeomorphism \(F\colon U\to V\) where \(V\) is a complex domain with boundary such that \(F\) restricted to \(U\cap \text{int}(M)\) is a holomorphic map.\\  
	A smooth connected manifold with smooth boundary together with a complex structure is called complex manifold with boundary of (complex) dimension \(\dim_\C M:=\dim_\C\text{int}(M)\). Furthermore,  we denote by \(\mathcal{O}^\infty(M):=\mathcal{O}(\text{int}(M))\cap C^\infty(M)\) the space of  holomorphic functions which are smooth up to the boundary.
\end{definition}
\begin{example}
	Put \(M=\{z\in\C^2\mid 1<|z|\leq 2\}\). Then \(M\) (together with differential and complex structure induced by \(\C^2\)) is a complex manifold with boundary.  We have \(\text{int}(M)=\{z\in\C^2\mid 1<|z|<2\}\) and \(bM=\{z\in\C^2\mid |z|=2\}\). Furthermore, with \(Y=\{z\in\C^2\mid  |z|>1\}\) and \(G=\{z\in Y\mid |z| <2\}\) we find that \(M=\overline{G}\) is a complex domain with boundary. 
\end{example}
For further constructions it is necessary to give an infinitesimal interpretation of complex manifolds with boundary. 
\begin{lemma}
	Let \(M\) be an \(n\)-dimensional complex manifold with boundary. There exists a uniquely defined smooth complex subbundle \(T^{1,0}M\subset \C TM\) of rank \(n\) such that \(T^{1,0}M|_{\operatorname{int}(M)}\) coincides with the complex structure on \(\operatorname{int}(M)\).	Furthermore, \(T^{1,0}M\) is formally integrable and \(T^{1,0}M\oplus\overline{T^{1,0}M}=\C TM\). 
\end{lemma}
\begin{proof}
	Given \(p\in \text{int}(M)\) we choose \(T^{1,0}_pM\) in accordance with the complex structure on \(\text{int}(M)\), that is, \(T^{1,0}_pM=T^{1,0}_p\text{int}(M)\). Given \(p\in bM\) put
	\[T_p^{1,0}M:=\{W_p\mid W\in \mathscr{C}^\infty(M,\C TM),\, W|_{\text{int}(M)}\in \mathscr{C}^\infty(\text{int}(M),T^{1,0}\text{int}(M))\}.\]
	We  find  an open neighborhood \(U\subset M\) around \(p\) and a diffeomorphism \(F\colon U\to V\) where \(V\) is a complex domain with boundary such that \(F\) restricted to \(U\cap \text{int}(M)\) is a holomorphic map. Furthermore, we can write \(V=\overline{G}\) for some smoothly bounded domain \(G\) of some complex manifold \(Y\). Since \(F\) is a diffeomorphism and holomorphic on \(U\cap \text{int}(M)\) we observe by continuity reasons \(T^{1,0}_pM=(dF_p)^{-1}T^{1,0}_{F(p)}Y\). Hence \(T^{1,0}_pM\) is a complex vector space of dimension \(n\) and \(T^{1,0}M\) is a smooth subbundle. With \(\overline{T^{1,0}_pM}=(dF_p)^{-1}\overline{T^{1,0}_{F(p)}Y}\) and \(\C TY_{F(p)}=T^{1,0}Y\oplus \overline{T^{1,0}Y}\) we find \(\C TX=T^{1,0}_pM\oplus\overline{T^{1,0}_pM}\).  Since \(T^{1,0}Y\) is formally integrable it follows that \(T^{1,0}M\) is formally integrable. Moreover, by construction we have that \(T^{1,0}M\) is contained in any smooth complex subbundle which coincides with \(T^{1,0}\text{int}(M)\) on \(\text{int}(M)\). Since \(\text{rank}T^{1,0}M=n\) the uniqueness follows. 
\end{proof} 
Put \(T^{0,1}M:=\overline{T^{1,0}M}\). The decomposition \(\C TM=T^{1,0}M\oplus T^{0,1}M \) immediately leads to the complex vector bundles \(\Lambda^{\bullet,\bullet}\C T^*M\) over \(M\) defined in the usual way. For any \(p,q\in\N_0\) we denote by \[\Omega^{p,q}(M):=\mathscr{C}^\infty(M,\Lambda^{p,q}\C T^*M)\]
the space of smooth \((p,q)\)-forms on \(M\) and let \(\Omega_c^{p,q}(M)\subset \Omega^{p,q}(M) \) be the space of those forms with compact support. The operators \(\partial,\overline{\partial}\), well understood for \((p,q)\) forms on \(\text{int}(M)\), uniquely extend to operators
\[\partial\colon \Omega^{p,q}(M)\to \Omega^{p+1,q}(M),\,\,\,\overline{\partial}\colon \Omega^{p,q}(M)\to \Omega^{p,q+1}(M).\]  
One can easily check that those construction coincides with the definitions in Section~\ref{subsec:PoincareLelong} when \(M\) is a complex domain with boundary.
By definition \(M\) is locally isomorphic to complex domains with boundary.  Then, the notion of \(f\in \mathcal{O}^\infty(M)\) being regular with respect to \(bM\) (see Definition~\ref{def:RegularWithRespectToBoundary}) as in  Section~\ref{subsec:PoincareLelong} can be extended for complex manifold with boundary. We immediately observe the following.
\begin{lemma}\label{lem:BoundaryIntegrableGlobalMfdBoundary}
 Lemma~\ref{lem:BoundaryIntegrableGlobal} and Corollary~\ref{cor:BoundaryIntegrableGlobal} hold when \(G\) and \(\overline{G}\) are replaced by \(\text{int}(M)\) and \(M\) respectively.
\end{lemma}
 In conclusion we obtain the following.
\begin{theorem}\label{thm:PoincareLelongMfdBoundary}
	Let \(M\) be an \((n+1)\)-dimensional complex manifold with boundary, \(f\in \mathcal{O}^\infty(M)\) regular with respect to \(bM\) and \(\psi\in\Omega^{n,n}_c(M)\). We have that
	\[\text{\,\,\,\,\,\,\,}\int_M\partial\overline{\partial}\log(|f|)\wedge \psi:=-\int_{bM}\iota^*(\frac{1}{2f}\partial f\wedge \psi)-\int_{bM}\iota^*(\log(|f|)\wedge\overline{\partial}\psi)+\int_M\log(|f|)\wedge\partial \overline{\partial} \psi\]
	and \(\left( \divisor{f}, \psi\right)\) are well defined with
	\begin{eqnarray}
		\left( \divisor{f}, \psi\right)=\frac{i}{\pi}\int_M \partial\overline{\partial}\log(|f|)\wedge \psi.
	\end{eqnarray}
	Here, \(\iota\colon bM\to M\) denotes the inclusion map.
\end{theorem}
\begin{proof}
	It is enough to show that any point \(p\in M\) has an open neighborhood \(U\subset M\) such that the statement holds for \(\psi\in\Omega^{n,n}_c(U)\). Given \(p\in \text{int}(M)\) we can choose \(U=\text{int}(M)\) and the claim follows. Given \(p\in bM\) we find an open neighborhood  \(U\subset M\) around \(p\) and a diffeomorphism \(F\colon U\to V\) where \(V\) is a complex domain with boundary such that \(F\) restricted to \(U\cap \text{int}(M)\) is a holomorphic map. Then the claim follows from Theorem~\ref{thm:PoincareLelongFormula}.
\end{proof}
The following lemma shows that \(bM\) naturally carries the structure of a CR manifold.
\begin{lemma}
	Let \(M\) be an \((n+1)\)-dimensional complex manifold with boundary such that \(bM\neq \emptyset\) and \(n\geq 1\). Then \(T^{1,0}bM:=\C TbM\cap T^{1,0}M\) defines a  (codimension one) CR structure on \(bM\) such  that the restriction \(f|_{bM}\) of any \(f\in \mathcal{O}^\infty(M)\)  is a smooth CR function.
\end{lemma} 
\begin{proof}
	We first observe that \(bM\) is an orientable smooth manifold of real dimension \(\dim_\R bM=2n+1\). Basic arguments from linear algebra show that \(\text{rank}_\C T^{1,0}bM=n\) and \(T^{1,0}bM\cap T^{0,1}bM=\text{'zero section'}\). Since \(TbM\) and \(T^{1,0}M\) are formally integrable it follows that \(T^{1,0}bM\) defines a (codimension one) CR structure on \(bM\). Given \(f\in \mathcal{O}^\infty(M)\) we have that \(f|_{bM}\) is smooth and we obtain for continuity reasons that \(\overline{Z}f=0\) for all \(Z\in T^{1,0}M\). Hence \(\overline{Z}(f|_{bM})=0\) for all \(Z\in T^{1,0}bM\) which proves that \(f|_{bM}\) is a CR function. 
\end{proof}
\begin{definition}\label{def:ComplexManifoldwSPSCBoundary}
	Let \(M\) be an \((n+1)\)-dimensional complex manifold with boundary. We say that \(p\in bM\) is a strictly pseudoconvex boundary point if there exists a diffeomorphism \(F\colon U\to V\) where \(U\) is an open neighborhood around \(p\) and \(V\) is complex domain with boundary such that \(F\) is holomorphic on \(U\cap \text{int}(M) \) and \(F(p)\) is a strictly pseudoconvex boundary point of \(V\). We call \(M\) a complex manifold with strictly pseudoconvex boundary if any point \(p\in bM\) is a strictly pseudoconvex boundary point.
\end{definition}
We observe that if \(M\) is a complex manifold with strictly pseudoconvex boundary it follows that \((bM,T^{1,0}bM)\) is a strictly pseudoconvex CR manifold. 
There is the following relation between complex manifolds with strictly pseudoconvex boundary and complex domains with boundary due to Heunemann~\cite{Heu86} and Ohsawa~\cite{O84}, see also Catlin \cite{MR1128581}, and Hill-Nacinovich \cite{MR1289628}.
\begin{theorem}[cf.\ {\cite[Theorem 6.3.15]{MM07}}]\label{thm:ThmKragenManifodSPSCBoundary}
	Let \(M\) be a compact complex manifold
	with strictly pseudoconvex boundary. Then \(M\) can be realized as a  domain with
	boundary in a larger complex manifold \(Y\). More preciesely, there exists a strictly pseudoconvex domain with boundary \(V\subset Y\) and a diffeomorphism \(F\colon M\to V\) which is a biholomorphism
	between \(\operatorname{int}(M)\) and \(\operatorname{int}(V)\), \(F(bM)=bV\).
\end{theorem}	  
\begin{definition}
	Let \(M\) be an \((n+1)\)-dimensional complex manifold with boundary. We say that \(M\) has the CR extension property if for any smooth CR function \(f\) on \(bM\) there exists a function \(F\in\mathcal{O}^\infty(M)\) with \(F|_{bM}=f\).
\end{definition}
Let us give some examples for complex manifolds with boundary satisfying the CR extension property.
\begin{example}\label{ex:KohnRossiCRExtensionGeneral}
	Let \(M\) be a complex domain with  boundary.  As already mentioned, it follows that \(M\) is a complex manifold with boundary. Assume \(\dim_\C M>1\) and that each point of \(p\in bM\) is a strictly pseudoconvex boundary point. If \(bM\) is compact we have that \(bM\) is CR embedabble into the complex Euclidean space (see~\cite{HsM17}). If  \(M\) is compact it follows from results due to Grauert~\cite{Gr58} and Kohn--Rossi~\cite{KR65} that \(bM\) is connected and that \(M\) has the CR extension property.  
\end{example}

\begin{example}\label{ex:DiscBundle}
	Let \((L,h)\to N\) be a positive  holomorphic line bundle with Hermitian metric \(h\) over a compact complex  manifold \(N\) (without boundary). Put \(M=\{v\in L^*\mid |v|_{h^*}\leq 1\}\) where \(L^*\) denotes the dual line bundle of \(L\) and \(h^*\) the Hermitian metric on \(L^*\) induced by \(h\). Then \(M\) is a compact complex manifold with strictly pseudoconvex boundary which has the CR extension property.  
\end{example}
\begin{example}\label{ex:GrauertTube}
	Let \((N,g)\) be a compact (without boundary) connected real analytic Riemannian manifold with real analytic metric \(g\). Assume \(\dim_\R N\geq 2\). For any sufficiently small \(r>0\) we have that \(M:=\{v\in T N\mid |v|_g\leq r\}\) is a compact complex manifold with strictly pseudoconvex boundary with respect to the adapted complex structure induced by \(g\) (see Grauert~\cite{Gr58} and Guillemin--Stenzel~\cite{GS91}). Furthermore, \(bM\) is CR embeddable into the complex Euclidean space and \(M\) has  the CR extension property by Example~\ref{ex:KohnRossiCRExtensionGeneral}.
\end{example}
\begin{example}\label{ex:Kragen}
	Let \((X,T^{1,0}X)\) be a compact strictly pseudoconvex CR manifold which is CR embeddable into the complex Euclidean space. There exists \(r>0\) such that \(M:=X\times [0,r)\) carries the structure of a complex manifold with strictly pseudoconvex boundary such that \((X,T^{1,0}X)\) is CR isomorphic to \((bM,T^{1,0}bM)\) (see Harvey--Lawson~\cite{HL75}). Furthermore, choosing \(r>0\) small enough, it follows that \(M\) has  the CR extension property.
\end{example}
\section{Equidistribution on Complex Manifolds with Strictly Pseudoconvex Boundary}\label{sec:EquidistributionOnSPCdomains}
In this section we are going to prove Theorem~\ref{thm:EquidistributionDomainsIntro}, Theorem~\ref{thm:ExpactationRndZerosIntro} and Theorem~\ref{thm:SequenceRndZerosIntro} (see Theorem~\ref{thm:equidistribution}, Theorem~\ref{thm:ExpactationRndZeros} and Theorem~\ref{thm:SequenceRndZeros}, respectively). Let \(M\) be an \((n+1)\)-dimensional complex manifold with  strictly pseudoconvex boundary, \(n\geq 1\), such that \(M\) has the CR extension property (see Section~\ref{sec:RemarksMfdBoundary}). Put \((X,T^{1,0}X):=(b M,T^{1,0}bM)\) and assume that \(X\) is non-empty, compact and CR embeddable. Then \(X\) (seen as the boundary of \(M\)) is a compact oriented strictly pseudoconvex CR manifold which satisfies the assumptions in Theorem~\ref{thm:ProjectiveEmbeddingIntro}. We denote by  \(\iota\colon bM\to M\) the inclusion map. Choose a volume form \(dV\) on \(X\).  Denote by \(\xi\) a contact form on \(X\) such that the respective Levi form is positive definite and let \(\mathcal{T}\) be the respective Reeb vector field uniquely determined by \(\iota_\mathcal{T}\xi\equiv1\) and \(\iota_\mathcal{T}d\xi\equiv0\). Let  $P\in L^1_{\mathrm{cl}}(X)$ be a  first order formally self-adjoint classical pseudodifferential operator with \(\sigma_P(\xi)>0\) where   $\sigma_P$ denotes the principal symbol of $P$. Denote by \(T_P=\Pi P\Pi\) the corresponding Toeplitz operator.  
Let 
\(0<\lambda_1\leq \lambda_2\leq \ldots\) denote the the positive eigenvalues of \(T_P\) and denote by \(\{f_j\}_{j\in\N}\subset H_b^0(X)\cap \mathscr{C}^\infty(X)\) an orthonormal system of eigenfunctions.   
Given a function \(\eta\in \mathscr{C}^\infty_c\left(\R_+,[0,\infty)\right)\) 
we have that the integral kernel of \(\eta_k(T_P)\) can be written as  \(\eta_k(T_P)(x,y)=\sum_{j=1}^\infty\eta_k(\lambda_j)f_j(x)\overline{f_j(y)}\) where \(\eta_k(t)=\eta(k^{-1}t)\), \(t,k>0\). From the CR extension property assumption on \(M\) any smooth CR function \(f\) on \(X\) extends holomorphically to \(\text{int}(M)\). In particular,  since \(M\) is connected, there exists a uniquely determined function \(\mathcal{F}\in\mathcal{O}^\infty(M):=\mathcal{O}(\text{int}(M))\cap \mathscr{C}^\infty(M)\) such that \(f=\mathcal{F}\) on \(X\). Hence we have a linear mapping 
\begin{eqnarray}\label{eq:defLExtensionOperator}
	L\colon H_b^0(X)\cap \mathscr{C}^\infty(X)\to \mathcal{O}^\infty(M),\,\,\,\,L(f)=\mathcal{F}.
\end{eqnarray}
We define 
\begin{eqnarray}\label{eq:defBetak}
	B^\eta_k\colon M\to \R,\,\,\,\,\,B^\eta_k(z)=\sum_{j=1}^\infty\eta_k(\lambda_j)L(f_j)(z)\overline{L(f_j)(z)}.
\end{eqnarray}
Then \(B^\eta_k\) is smooth and for any \(c>0\) we have that \(k^{-1}\partial \overline{\partial}\log(c+B^\eta_k)\) defines a non negative \((1,1)\)-form on \(\text{int}(M)\) which is smooth up to the boundary.
Let \(\chi\in \mathscr{C}^\infty_c(\R_+)\), \(\chi\not \equiv 0\), be a smooth function with \(\text{supp}(\chi)\subset (\delta_1,\delta_2)\) for some \(0<\delta_1<\delta_2\). Put \(N_k=\#\{j\in\N\colon \lambda_j\leq \delta_2k\}\).  For \(k>0\) put \(A_k=\text{span}\left(\{1\}\cup\{L(f_j)\mid \lambda_j\leq \delta_2k\}\right)\). On \(A_k\) we consider the probability measure \(\mu_k\) induced by the standard complex Gaussian measure \(\mu^G\) on \(\C^{N_k+1}\) and the map \[\C^{N_k+1} \ni a \mapsto a_0+\sum_{j=1}^{N_k}a_j\chi_k(\lambda_j)L(f_j) \in A_k\]
where \(\chi_k(t)=\chi(k^{-1}t)\).  We use the same notation \((A_k,\mu_k)\) for the probability spaces as in Section~\ref{sec:EquidistributionCR} since the constructions coincides. In particular, compared to Section~\ref{sec:EquidistributionCR}, we choose \(\kappa(k)\equiv 1\) in this section.
\subsection{Proof of Theorem~\ref{thm:ExpactationRndZerosIntro}}
To simplify the notations in the proofs we consider the following object.
\begin{definition}\label{def:ZF}
	Let \(f\in \mathcal{O}^\infty(M)\) such that \(\partial f_p\neq 0\) for all \(p\in bM\) with \(f(p)=0\). We define
	\[\mathcal{Z}_f\colon \Omega_c^{n,n}(M)\to \C,\,\,\,\mathcal{Z}_f(\psi)=\left(\divisor{f},\psi\right).\]
	where \(\divisor{f}\) denotes the zero divisor of \(f\).
\end{definition}
We note that under the assumptions on \(f\) in Definition~\ref{def:ZF} we have that  \(\mathcal{Z}_f\) is well defined by Lemma~\ref{lem:ZeroDivisorWellDefined} (cf.\ Section~\ref{lem:BoundaryIntegrableGlobalMfdBoundary}). Furthermore, using Theorem~\ref{thm:PoincareLelongMfdBoundary}, \(Z_f(\psi)\) can be computed via the formula
\begin{eqnarray}\label{eq:PoincareLelongInEquidistributionDomain}
	\mathcal{Z}_{f}(\psi)&:=&\left(\divisor{f},\psi\right)\\
	&=&\frac{i}{\pi}\left(-\int_{bM}\iota^*(\frac{1}{2f}\partial f\wedge \psi)-\int_{bM}\iota^*(\log(|f|)\wedge\overline{\partial}\psi)+\int_M\log(|f|)\wedge\partial \overline{\partial} \psi\right)\nonumber
\end{eqnarray}
for all \(\psi\in \Omega_c^{n,n}(M)\). 
Since we wish to understand \(\mathcal{Z}_f\) as a random variable on \(A_k\), we have to make sure that it is defined on a set \(\tilde{A}_k\subset A_k\) with \(\mu_k(\tilde{A_k})=1\). 
We have the following.
\begin{lemma}\label{lem:UnregularFunctionsAreZeroSetDomain}
	Put \(\tilde{A}_k:=\{f\in A_k\mid \text{0 is a regular value of } f|_{bM}\colon bM\to \C\}\). We have that \(\tilde{A}_k\) is  \(\mu_k\)-measurable with \(\mu_k(A_k\setminus\tilde{A}_k)=0\).
\end{lemma}
\begin{proof}
	We have \(1+\sum_{j=1}^{N_k}|\chi_k(\lambda_j)|^2|f_j(x)|^2\geq 1>0\) for all \(k>0\) and all \(x\in bD\). Then the claim follows from Lemma~\ref{lem:singularValueFunctionsMeasureZeroGeneral}.
\end{proof}
We note that given \(f\in\tilde{A}_k\) with \(\tilde{A}_k\) as in Lemma~\ref{lem:UnregularFunctionsAreZeroSetDomain} it follows that \(\partial f_p\neq 0\) for all \(p\in bM\) with \(f(p)=0\). Hence \(\mathcal{Z}_f\) is well defined and we can apply~\eqref{eq:PoincareLelongInEquidistributionDomain} for its computation.
Theorem~\ref{thm:EquidistributionDomainsIntro} follows from the next result by taking \(c=1\) and assuming that \(M\) is compact (see Example~\ref{ex:KohnRossiCRExtensionGeneral}). 
\begin{theorem}\label{thm:ExpactationRndZeros}
	Let \(M\) be an \((n+1)\)-dimensional complex manifold with   strictly pseudoconvex boundary, \(n\geq 1\), such that \(M\) has the CR extension property (see Section~\ref{sec:RemarksMfdBoundary}). Put \(X:=b M\) and assume that \(X\) is non-empty, compact and CR embeddable. Let \(\xi\) be a contact form  for \(X\) such that its induced Levi form is positive definite. Choose a volume form \(dV\) on \(X\) and consider a Toeplitz operator \(T_P=\Pi P \Pi\) on \(X\) as in the set-up of Theorem~\ref{thm:ProjectiveEmbeddingIntro}. With the notations above we have
	for each \(k>0\)  that \(\mathcal{Z}_f\) is a well-defined distribution for all \(f\) outside a subset in \(A_k\) of \(\mu_k\)-measure zero and the expectation value \(\mathbb{E}_k\left(\mathcal{Z}_f\right)\) defined by
	\[\left(\mathbb{E}_k\left(\mathcal{Z}_f\right),\psi\right):=\int_{A_k} \mathcal{Z}_f(\psi)d\mu_k(f),\,\,\,\,\,\psi\in \Omega^{n,n}_c(M)\] exists 
	with \[\mathbb{E}_k\left(\mathcal{Z}_f\right)=\frac{i}{2\pi}\partial\overline{\partial}\log(1+B^{|\chi|^2}_{k}).\]
	Here, \(B^{|\chi|^2}_{k}\) is given by~\eqref{eq:defBetak} where \(\chi\) appears in the definition of \(A_k\).
\end{theorem}
Similar to Section~\ref{sec:EquidistributionCR} we introduce the following notation.
For \(d\in\N\) we denote by \(\langle\cdot,\cdot\rangle\) the standard inner product on \(\C^{d+1}\) that is
\(\langle a,b\rangle= \sum_{j=0}^da_j\overline{b_j} \) 
for \(a=(a_0,\ldots,a_d), b=(b_0,\ldots,b_d)\in\C^{d+1}\). Furthermore, given \(a\in \C^{d+1}\) we put \(|a|=\sqrt{\langle a, a\rangle}\). For \(k>0\) consider the map
\begin{eqnarray}
	F_k\colon M\to \C^{N_k+1},\,\,\,\,
	F_k=(1,\chi_k(\lambda_1)L(f_1),\ldots,\chi_k(\lambda_{N_k})L(f_{N_k})). 
\end{eqnarray}
Define \(S_k(z,w)=\sum_{j=1}^\infty|\chi_k(\lambda_j)|^2L(f_j)(z)\overline{L(f_j)(w)}\), \(z,w\in M\) and
put \(B_k:=B_k^{|\chi|^2}\). It turns out that \(|F_k|^2=1+B_k\), \(S_k(z,z)=B_k(z)\) and \(\langle F_k(w),F_k(w)\rangle=1+S_k(z,w)\) for all \(z,w\in M\) and \(k>0\). With \(\eta_k=|\chi_k|^2\) we find \(S_k(z,w):=\eta_k(T_P)(z,w)=\sum_{j=1}^{N_k}\eta_k(\lambda_j)f_j(z)\overline{f_j(w)}\)    for all \(z,w\in b M\) and \(k>0\). We further consider \(dF_k\) component wise, that is,
\[dF_k=(0,\chi_k(\lambda_1)df_1,\ldots,\chi_k(\lambda_{N_k})df_{N_k}).\]
Then \(\langle a, \overline{dF_k}\rangle\) is a one form on \(M\) for any \(a\in \C^{N_k+1}\). We have for example \(\langle F_k(z), (dF_k)_w\rangle=d_wS_k(z,w)\) and for \(f=\langle a,\overline{F_k}\rangle\) we obtain \(df=\langle a,\overline{dF_k}\rangle\).

We will prove Theorem~\ref{thm:ExpactationRndZeros} using the decomposition of \(\mathcal{Z}_f\) given by~\eqref{eq:PoincareLelongInEquidistributionDomain} and consider each term in the sum separately. The first term in the right-hand side  of~\eqref{eq:PoincareLelongInEquidistributionDomain} can be handled using the following lemma which is a direct consequence of the results in  Section~\ref{sec:EquidistributionCR}.
\begin{lemma}\label{lem:IkexistsAndValue}
	With the notations and assumptions above given \(k>0\) and \(\psi\in \Omega_c^{n,n}(M)\) we have that 
	\[I_k(\psi):=\int_{A_k}\int_{bM}\iota^*\left(\frac{\partial f}{f}\wedge \psi\right)d\mu_k(f)\]
	exists with
	\[I_k(\psi)=\int_{bM}\iota^*\left((\partial \log(1+B_k))\wedge\psi\right).\]
\end{lemma}
\begin{proof}
	Fix \(k>0\) and \(\psi\in \Omega_c^{n,n}(M) \). We note that \(|F_k|^2\geq 1>0\). Given \(f\in\mathcal{O}^\infty(M)\) we have \(\iota^*(\partial f)=\iota^*(d f)=d(f\circ \iota)\) with \(f\circ \iota \in H_b^0(X)\), \(X=bM\). Then it follows from Lemma~\ref{lem:UnregularFunctionsAreZeroSetDomain} and Lemma~\ref{lem:BoundaryIntegrableGlobalMfdBoundary} that
	\[\int_{b M}\iota^*\left(\frac{\partial f}{f}\wedge \psi\right)=\int_{b M}\iota^*\left(\frac{\partial f}{f}\right)\wedge \iota^*\psi=\int_{b M}\frac{d (f\circ\iota )}{f\circ \iota}\wedge \iota^*\psi\]
	holds for \(\mu_k\)-almost every \(f\in A_k\). Since \(\iota^*\psi\in\Omega^{2n}(X)\) and \(|F_k|^2>0\) we conclude from Theorem~\ref{thm:ExpectationValueCRDistribution} that \(I_k(\psi)\) exists with
	\[I_k(\psi)=\int_{b M}2\pi i \beta_k\wedge\iota^*\psi\]
	where
	\[2\pi i \beta_k=\frac{\langle d(F_k\circ\iota),F_k\circ\iota\rangle}{|F_k\circ\iota|^2}.\]
	Since \(d(F_k\circ\iota)=\iota^*dF_k=\iota^*\partial F_k\) and \(\langle\partial F_k,F_k\rangle=\partial |F_k|^2=\partial B_k \) we obtain
	\[2\pi i \beta_k=\iota^*\left(\frac{\langle \partial F_k,F_k\rangle}{|F_k|^2}\right)=\iota^*\left(\frac{\partial B_k}{1+B_k}\right)=\iota^*\left(\partial\log(1+B_k)\right).\]
	The claim follows. 
\end{proof}
In order to handle the log-terms in~\eqref{eq:PoincareLelongInEquidistributionDomain} we need the following.
\begin{lemma}\label{lem:logfpsiexists}
	With the notations and assumptions above we have that for each \(k>0\) and any compact set \(K\subset M\) there exists a constant \(C_{k,K}>0\) such that
	\begin{eqnarray}\label{eq:logaFexistslocal}
		\int_{\C^{N_k+1}}\left|\log(|\langle a, \overline{F_k(z)}\rangle|)\right|  d\mu^G(a)\leq C_{k,K}
	\end{eqnarray}
	holds for all \(z\in K\).
	Furthermore, we have that 
	\[\int_{\C^{N_k+1}}\log|\langle a, \overline{F_k(z)} \rangle| d\mu^G(a)-\frac{1}{2}\log(1+B_k(z))=\int_{\C}\log|a_0|d\mu^G(a_0)\]
	for all \(k>0\), \(z\in M\). 
\end{lemma}
\begin{proof}
	We have that \(c:=\int_{\C^{N_k+1}}|\log|a_0||d\mu^G(a)<\infty\) is independent of \(k>0\) and \(z\in\overline{\domain}\).  Since \(F_k(z)/|F_k(z)|\) has unit length we obtain from the unitary invariance of \(d\mu^G\) that
	\begin{eqnarray*}
		c&=&\int_{\C^{N_k+1}} \left|\log\left(\left\langle a, \frac{\overline{F_k(z)}}{|F_k(z)|}\right\rangle\right)\right| d\mu^G(a)\\
		&\geq& \int_{\C^{N_k+1}} \left|\log(\langle a, \overline{F_k(z)}\rangle)\right| d\mu^G(a) -|\log(|F_k(z)|)|.
	\end{eqnarray*}
	Since \(|F_k(z)|^2=1+B_k(z)\geq 1\) we obtain~\eqref{eq:logaFexistslocal}. Similarly, 
	we find
	\begin{eqnarray*}
		\int_{\C^{N_k+1}} \log(\langle a, \overline{F_k(z)}\rangle)\ d\mu^G(a)-\log(|F_k(z)|)&=&\int_{\C^{N_k+1}} \log\left(\left\langle a, \frac{\overline{F_k(z)}}{|F_k(z)|}\right\rangle\right) d\mu^G(a)\\
		&=& \int_{\C}\log|a_0|d\mu^G(a_0). 
	\end{eqnarray*}
	With \(2\log(|F_k(z)|)=\log(1+B_k(z))\) the claim follows. 
\end{proof}
Now we can consider the last two terms in the right-hand side of~\eqref{eq:PoincareLelongInEquidistributionDomain}. We have the following.
\begin{lemma}\label{lem:IIkexistsAndValue}
	With the notations and assumptions given \(k>0\) and \(\psi\in \Omega_c^{n,n}(M)\) we have that 
	\[II_k(\psi):=\int_{A_k}\left(-\int_{bM}\iota^*(\log(|f|)\wedge\overline{\partial}\psi)+\int_M\log(|f|)\wedge\partial \overline{\partial} \psi\right) d\mu_k(f)\]
	exists with
	\[II_k(\psi)=-\frac{1}{2}\int_{M}(\partial \log(1+B_k))\wedge \overline{\partial}\psi.\] 
\end{lemma}
\begin{proof}
	Fix \(k>0\) and \(\psi\in \Omega_c^{n,n}(M)\). From~\eqref{eq:logaFexistslocal} it follows by the Fubini-Tonelli theorem  that  \(II_k(\psi)\) exists.
	Put \(c:=\int_{\C}\log|a_0|d\mu^{G}(a_0)=\int_{\C^{N_k+1}}\log|a_0|d\mu^{G}(a)\). 
	From Lemma~\ref{lem:logfpsiexists}, the Fubini theorem and the definition of \(\mu_k\) we obtain
	\begin{eqnarray*}
		\int_M	\left(\frac{1}{2}\log(1+B_k)+c\right)\partial\overline{\partial}\psi&=&\int_{\C^{N_k+1}}\int_M\log|\langle a,F_k\rangle|\partial\overline{\partial}\psi d\mu^{G}(a)\\
		&=& \int_{A_k}\int_M\log(|f|)\wedge\partial \overline{\partial} \psi d\mu_k(f).
	\end{eqnarray*}
	With the same arguments we find
	\begin{eqnarray*}
		\int_{b M}\iota^*\left(	\left(\frac{1}{2}\log(1+B_k)+c\right)\overline{\partial}\psi\right)&=&\int_{\C^{N_k+1}}\int_{bM}\iota^*\left(\log|\langle a,F_k\rangle|\overline{\partial}\psi\right) d\mu^{G}(a)\\
		&=& \int_{A_k}\int_{bM}\iota^*\left(\log(|f|)\overline{\partial} \psi\right) d\mu_k(f).
	\end{eqnarray*}
	Using Stokes theorem we obtain \(\int_{b M}\iota^*(\overline{\partial}\psi)=\int_M\partial\overline{\partial}\psi\) and
	\begin{eqnarray*}
		\int_{b M}\iota^*\left(\log(1+B_k)\overline{\partial} \psi\right)&=&\int_M \left(d \log(1+B_k)\right)\wedge \overline{\partial}\psi + \int_M\log(1+B_k)\wedge d \overline{\partial} \psi\\
		&=&\int_M \left(\partial \log(1+B_k)\right)\wedge \overline{\partial}\psi + \int_M\log(1+B_k)\wedge\partial \overline{\partial} \psi.
	\end{eqnarray*}
	We conclude
	\begin{eqnarray*}
		2II_k(\psi)&=& -\int_{bM}\iota^*\left(\log(1+B_k) \overline{\partial} \psi\right) + \int_M\log(1+B_k)\wedge\partial \overline{\partial} \psi + 2c\left(- \int_{b M}\iota^*(\overline{\partial}\psi)+\int_M\partial\overline{\partial}\psi\right)\\
		&=&-\int_M\left( \partial \log(1+B_k)\right)\wedge \overline{\partial}\psi.
	\end{eqnarray*}
\end{proof}
Using~\eqref{eq:PoincareLelongInEquidistributionDomain} we obtain Theorem~\ref{thm:ExpactationRndZeros} as a consequence of Lemma~\ref{lem:IkexistsAndValue} and Lemma~\ref{lem:IIkexistsAndValue}.
\begin{proof}[\textbf{Proof of Theorem~\ref{thm:ExpactationRndZeros}}]
	Fix \(k>0\) and \(\psi\in \Omega_c^{n,n}(M)\). From Lemma~\ref{lem:IkexistsAndValue} and Lemma~\ref{lem:IIkexistsAndValue} it follows that
	\[III_k(\psi):=\frac{i}{\pi}\int_{A_k}\left(-\int_{b M}\iota^*(\frac{1}{2f}\partial f\wedge \psi)-\int_{b M}\iota^*(\log(|f|)\wedge\overline{\partial}\psi)+\int_M\log(|f|)\wedge\partial \overline{\partial} \psi\right)d\mu_k(f)\]
	exists with \(III_k(\psi)<\infty\).  Using Lemma~\ref{lem:UnregularFunctionsAreZeroSetDomain} we find a set \(\tilde{A}_k\subset A_k\) with \(\mu_k(\tilde{A_k})=1\) such that Theorem~\ref{thm:PoincareLelongMfdBoundary} can be applied for all \(f\in \tilde{A}_k\). We find that
	\begin{eqnarray*}
		III_k(\psi)&=&\frac{i}{\pi}\int_{\tilde{A_k}}\left(-\int_{bM}\iota^*(\frac{1}{2f}\partial f\wedge \psi)-\int_{bM}\iota^*(\log(|f|)\wedge\overline{\partial}\psi)+\int_M\log(|f|)\wedge\partial \overline{\partial} \psi\right)d\mu_k(f)\\
		&=& \int_{\tilde{A_k}}\left(\divisor{f},\psi\right) d\mu_k(f)\\
		&=& \int_{A_k}\left(\divisor{f},\psi\right) d\mu_k(f)=\mathbb{E}_k(\mathcal{Z}_f)(\psi). 
	\end{eqnarray*}
	Hence \(\mathbb{E}_k(\mathcal{Z}_f)(\psi)<\infty\) exists. 
	From Stokes theorem we find
	\[\int_{bM} \iota^*\left(\partial \log(1+B_k)\wedge \psi\right)=-\int_M\left(\partial\overline{\partial}\log(1+B_k)\right)\wedge\psi -\int_M \left(\partial \log(1+B_k)\right)\wedge \overline{\partial}\psi. \]
	Since \(III_k(\psi)=\pi^{-1}i(-I_k(\psi)/2+II_k(\psi))\) we conclude from Lemma~\ref{lem:IkexistsAndValue} and Lemma~\ref{lem:IIkexistsAndValue} that
	\begin{eqnarray*}
		III_k(\psi)&=&\frac{i}{2\pi}\left(-\int_{bM}\iota^*((\partial \log(1+B_k))\wedge\psi)-\int_{M}(\partial \log(1+B_k))\wedge \overline{\partial}\psi\right)\\
		&=&\frac{i}{2\pi}\int_M\left(\partial\overline{\partial}\log(1+B_k)\right)\wedge\psi.
	\end{eqnarray*}
\end{proof}
\subsection{A Variance Estimate for \(\mathcal{Z}_f\)}
The proof of Theorem~\ref{thm:SequenceRndZerosIntro} will mainly rely on a variance Estimate for the distribution valued random variable \(\mathcal{Z}_f\) which we are going to explain now. We start with an elementary lemma.
\begin{lemma}\label{lem:EstimateLogRndVarinaceGauss}
	For all \(\alpha,\beta\in \C\) with \(|\alpha|^2+|\beta|^2=1\) we have
	\[\int_{\C^2}|\log(|a_0|)\log(|\alpha a_0+\beta a_1|)|d\mu^G(a)\leq\int_{\C}|\log(|a|)|^2d\mu^G(a).\]
\end{lemma}
\begin{proof}
	Put \(c:=\int_{\C}|\log(|a|)|^2d\mu^G(a)\) by the unitary invariance of \(\mu^G\) we obtain \[\int_{\C^2}|\log(|\alpha a_0+\beta a_1|)|^2d\mu^G(a)=c.\] By the Cauchy-Schwarz inequality we  conclude 
	\begin{eqnarray*}
		\int_{\C^2}|\log(|a_0|)\log(|\alpha a_0+\beta a_1|)|d\mu^G(a)&\leq& \sqrt{\int_{\C^2}|\log(|a_0|)|^2 d\mu^G(a)} \sqrt{\int_{\C^2}|\log(|\alpha a_0+\beta a_1|)|^2 d\mu^G(a)}\\
		&=&c.
	\end{eqnarray*}
\end{proof}
We have the following result for the variance of \(\mathcal{Z}_f\).
\begin{theorem}\label{thm:VarianceZerosEstimate}
	Under the assumptions of Theorem~\ref{thm:ExpactationRndZeros}, with the notations  above put
	\[\mathbb{V}_k(\mathcal{Z}_f(\psi)):=\int_{A_k}\left|\mathcal{Z}_f(\psi)-\mathbb{E}_k(\mathcal{Z}_f)(\psi)\right|^2d\mu_k(f),\,\,\, \psi\in \Omega_c^{n,n}(M).\]
	Given \(0<\varepsilon<1\) there exist \(k_0>0\) such that for any compact set \(K\subset M\) there is a constant \(C_K>0\) with
	\[\mathbb{V}_k(\mathcal{Z}_f(\psi))\leq C_Kk^\varepsilon\|\psi\|^2_{\mathscr{C}^2(M,\Lambda^{n,n}\C T^*M)} \]  for all \(k\geq k_0\) and all \(\psi\in  \Omega_c^{n,n}(M)\), \(\supp\psi \subset K\). 
\end{theorem}
\begin{proof}
	Let \(0<\varepsilon <1\) be arbitrary. For \(k>0 \) and \(\psi\in  \Omega_c^{n,n}(M)\)  put
	\[R_k(\psi):=\int_{A_k}\left|\frac{i}{\pi}\int_{M}\partial\overline{\partial}\log|f|\wedge \psi-\frac{i}{2\pi}\int_M\left(\partial\overline{\partial}\log(1+B_k)\right)\wedge\psi\right|^2d\mu_k(f)\]
	and
	\begin{eqnarray*}
		R^{(1)}_k(\psi)&=&\int_{A_k}\left|\int_{bM}\iota^*\left(\frac{\partial f}{f}-\partial \log(1+B_k)\right)\wedge \iota^*\psi\right|^2d\mu_k(f),\\
		R^{(2)}_k(\psi)&=&\int_{A_k}\left|\int_{bM}\iota^*\left(\log\left(\frac{|f|}{\sqrt{1+B_k}}\right)\overline{\partial} \psi\right)\right|^2d\mu_k(f),\\
		R^{(3)}_k(\psi)&=&\int_{A_k}\left|\int_{M}\log\left(\frac{|f|}{\sqrt{1+B_k}}\right)\partial\overline{\partial} \psi\right|^2d\mu_k(f).\\
	\end{eqnarray*}
	We have \(\frac{i}{\pi}\int_{M}\partial\overline{\partial}\log|f|\wedge \psi= \left(\divisor{f},\psi\right)=:\mathcal{Z}_f(\psi)\) for all \(f\) in a subset \(\tilde{A}_k\) of \(A_k\) with \(\mu_k(\tilde{A}_k)=1\) by Lemma~\ref{lem:UnregularFunctionsAreZeroSetDomain}. Then, using Theorem~\ref{thm:ExpactationRndZeros}, we find
	\begin{eqnarray}\label{eq:RkEqualVkInZerosDomain}
		R_k(\psi)=\mathbb{V}_k(Z_f(\psi)).
	\end{eqnarray}
	Furthermore, from the definition of \(\int_{M}\partial\overline{\partial}\log|f|\wedge \psi\) and Stokes theorem  we find
	\begin{eqnarray}\label{eq:EstimateRkInZerosDomain}
		R_k(\psi)\leq \frac{1}{2\pi^2} R^{(1)}_k(\psi)+\frac{4}{\pi^2} \left(R^{(2)}_k(\psi)+R^{(3)}_k(\psi)\right).
	\end{eqnarray}
	Hence, it is enough to show that there is \(k_0>0\) and, given a compact set \(K\subset M\), a constant \(C_K>0\) such that  \(R^{(j)}_k(\psi)\leq C_Kk^{\varepsilon}\|\psi\|^2_{\mathscr{C}^2(M,\Lambda^{n,n}\C T^*M)}\) for all \(k\geq k_0\), \(\psi\in \Omega^{n,n}_c(M)\) with \(\supp\psi\subset K\) and \(j=1,2,3\). Then the claim in Theorem~\ref{thm:VarianceZerosEstimate} follows immediately. 
	
	We will estimate the term \(R^{(1)}_k(\psi)\) using Theorem~\ref{thm:VarianceEstimateCRCase}. With respect to  Theorem~\ref{thm:ExpectationValueCRDistribution} we have that 
	\[\beta_k(x)=\frac{d_x\eta_k(T_P)(x,y)|_{x=y}}{2\pi i(1+\eta_k(T_P)(x,x))}, \,\,\,x\in b M\]
	defines a smooth one form on \(b M\) for all \(k>0\) where \(\eta=|\chi|^2\).
	Recall that 
	\[\eta_k(T_P)(z,w)=\sum_{j=1}^\infty |\chi_k(\lambda_j)|^2f_j(z)\overline{f_j(w)}=\langle F_k(z), F_k(w)\rangle-1,\,\,\, z,w\in bM.\]
	The components of \(F_k\) are  functions in \(\mathcal{O}^\infty(M)\). Hence we obtain \(\partial F_k=dF_k\) and in addition we find with \(|F_k|^2=1+B_k\) that
	\[2\pi i\beta_k=\frac{\iota^*\partial B_k}{1+B_k\circ \iota}=\iota^*\partial\log\left(1+B_k\right).\]
	Using Theorem~\ref{thm:ExpectationValueCRDistribution} we obtain in conclusion
	\[R^{(1)}_k(\psi)=\int_{A_k}\left|\int_{b M} \left(\iota^*\left(\frac{df}{f}\right)-2\pi i\beta_k\right)\wedge\iota^*\psi \right|^2d\mu_k(f)
	= 4\pi^2 \mathbb{V}_k(\mathcal{C}_f(\iota^*\psi))\]
	for all \(k>0\) and all  \(\psi\in \Omega^{n,n}_c(M)\),  where \(\mathbb{V}_k(\mathcal{C}_f(\iota^*\psi))\) is as in Theorem~\ref{thm:VarianceEstimateCRCase} with \(X=b M\) and \(\kappa(k)\equiv 1\). Then, by Theorem~\ref{thm:VarianceEstimateCRCase}, there exist \(k_0,C_1>0\) such that
	\begin{eqnarray}\label{eq:EstimateR1kInZerosDomain}
		R_k^{(1)}(\psi)\leq C_1k^{\varepsilon}\|\psi\|^2_{\mathscr{C}^0(M,\Lambda^{n,n} \C T^*M)}
	\end{eqnarray}
	holds for all \(k\geq k_0\) and \(\psi\in \Omega^{n,n}_c(M)\). It remains to estimate \(R_k^{(2)}(\psi)\) and \(R_k^{(3)}(\psi)\).
	
	Fix a compact set \(K\subset M\). Recall that \(|F_k|\geq 1 >0\) for all \(k>0\). For \(z,w\in M\) we can choose a unitary transformation \(U^{z,w}_k\) of \(\C^{N_k+1}\) such \(U^{z,w}_k\frac{\overline{F_k(z)}}{|F_k(z)|}=(1,0,\ldots,0)\) and \(U^{z,w}_k\frac{\overline{F_k(w)}}{|F_k(w)|}=(\alpha,\beta,0\ldots,0)\) with \(\alpha,\beta\in \C\), \(|\alpha|^2+|\beta|^2=1\). 
	Denote by \(p_1, p_2\colon M\times M\to M\) the projections on the first and second component. 
	We write
	\[R^{(3)}_k=\int_{A_k}\int_{M\times M}\log\left(\frac{|f\circ p_1|}{\sqrt{1+B_k\circ p_1}}\right)\log\left(\frac{|f\circ p_2|}{\sqrt{1+B_k\circ p_2}}\right)p_1^*(\partial\overline{\partial} \psi)\wedge \overline{p_2^*(\partial\overline{\partial} \psi)}d\mu_k(f).\]
	Using \(U^{z,w}_k\) we find by the invariance of \(\mu^G\) under unitary transformation and Lemma~\ref{lem:EstimateLogRndVarinaceGauss} that
	\begin{eqnarray*}
		\int_{\C^{N_k+1}}\left|\log\left|\left\langle a,\frac{\overline{F_k(z)}}{|F_k(z)|}\right\rangle\right|\log\left|\left\langle a,\frac{\overline{F_k(w)}}{|F_k(w)|}\right\rangle\right|\right|d\mu^G(a)&=&\int_{\C^{N_k+1}}|\log(|a_0|)\log(|\alpha a_0+\beta a_1|)|d\mu^G(a)\\&\leq&\int_{\C}|\log(|a|)|^2d\mu^G(a).
	\end{eqnarray*}
	With \(|F_k|^2=1+B_k\), \(f=\langle a,\overline{F_k}\rangle\) for \(a\in\C^{N_k+1}\), and \(p_1(z,w)=z\), \(p_2(z,w)=w\)  it follows from the Fubini-Tonelli theorem that there is a constant \(C_{3,K}>0\) such that \(R_k^{(3)}(\psi)\leq C_3\|\psi\|^2_{\mathscr{C}^2(M,\Lambda^{n,n}\C T^*M)}\) holds for all \(k>0\) and \(\psi\in \Omega^{n,n}_c(M)\) with \(\supp \psi\subset K\). With the same arguments we obtain \(R_k^{(2)}(\psi) \leq C_{2,K}\|\psi\|^2_{\mathscr{C}^2(M,\Lambda^{n,n}\C T^*M)}\) for some constant \(C_{2,K}>0\) independent of \(k\) and \(\psi\), \(\supp\psi\subset K\). Combining those estimates for \(R_k^{(2)}(\psi)\) and \(R_k^{(3)}(\psi)\) with the estimate for \(R_k^{(1)}(\psi)\) in~\eqref{eq:EstimateR1kInZerosDomain} the claim of the theorem follows from~\eqref{eq:RkEqualVkInZerosDomain} and~\eqref{eq:EstimateRkInZerosDomain}.
\end{proof}
\subsection{Proof of Theorem~\ref{thm:EquidistributionDomainsIntro}}
Theorem~\ref{thm:EquidistributionDomainsIntro} follows from the next result by putting \(c=1\). 
 \begin{theorem}\label{thm:equidistribution}
  	Under the assumptions of Theorem~\ref{thm:ExpactationRndZeros} and with the notation above, assuming that \(M\) is compact, choose  \(\eta\in \mathscr{C}^\infty(\R_+,[0,\infty))\), \(\eta\not \equiv 0\), and \(c>0\). With \(B_k^\eta\) as in~\eqref{eq:defBetak}, for all \(\psi\in \Omega^{n,n}(M)\), we have
  	\begin{equation}
  		\lim_{k\to\infty}\int_M k^{-1}i\partial \overline{\partial}\log(c+B^\eta_k)\wedge \psi =\operatorname{mv}(\eta)\int_{bM}\frac{\xi}{\sigma_P(\xi)}\wedge \iota^*\psi,
  	\end{equation}
  	where 
  	\begin{eqnarray}
  		\operatorname{mv}(\eta)=\frac{\int_0^{+\infty} t^{n+1}\eta(t)dt}{\int_0^{+\infty} t^{n}\eta(t) dt}.
  	\end{eqnarray}
  	Here \(\iota\colon bM \to M\) denotes the inclusion map. More precisely, there exists a constant \(C>0\) such that
  	\[\left| \int_M k^{-1}i\partial \overline{\partial}\log(c+B^\eta_k)\wedge \psi - \operatorname{mv}(\eta)\int_{bM}\frac{\xi}{\sigma_P(\xi)}\wedge \iota^*\psi\right|\leq Ck^{-1}(\log(k)+1)\|\psi\|_{\mathscr{C}^2(M,\Lambda^{n,n}\C T^*M)} \]
  	for all \(k\geq 1\) and  \(\psi\in \Omega^{n,n}(M)\).
  \end{theorem}
 For the proof of Theorem~\ref{thm:equidistribution} we need the following lemma.
 \begin{lemma}\label{lem:logBkbounded}
 In the situation of Theorem~\ref{thm:equidistribution} there exists a constant \(C>0\) such that for all \(k>0\) and \(z\in M\) we have \(|\log(c+B_k^\eta(z))|\leq C(\log(k)+1)\).
 \end{lemma}
 \begin{proof}
 	Given \(z\in bM\) and \(k>0\) we have \(B_k^\eta(z)=\eta_k(T_P)(z,z)\). Furthermore, using Theorem~\ref{thm:ExpansionMain} we obtain a constant \(R>0\) such that \(0\leq \eta_k(T_P)(z,z)\leq Rk^{n+1}\) holds for all \(z\in bM\) and \(k>0\).
 	Since \(B^\eta_k\in \mathscr{C}^\infty(M)\) is plurisubharmonic on \(\text{int}(M)\) it follows from the maximum principle for plurisubharmonic functions that  \(0\leq B^\chi_k(z)\leq Rk^{n+1}\) for all \(z\in M\) and and \(k>0\). The statement follows from \(\log (c+B^\chi_k)\geq \log(c)\) with \(c>0\).
 \end{proof}
 \begin{proof}[\textbf{Proof of Theorem~\ref{thm:equidistribution}}]
 	Recall the definition of \(L\) in~\eqref{eq:defLExtensionOperator}. Given a smooth CR function \(f\) on \(X=bM\) we have \(L(f)\in \mathcal{O}^\infty(M)\) and hence \(\partial (L(f))=d(L(f))\) on \(M\). It follows that
 	\[(\iota^*\overline{\partial}B^\eta_k)(z)=\sum_{j=1}^\infty\eta_k(\lambda_j)f_j(z)d\overline{f_j}(z)=d_w\eta_k(T_P)(z,w)|_{z=w}\]
 	for all \(z\in b M\). 
 	Using Lemma~\ref{lem:firstDifferentialMainThm} and Theorem~\ref{thm:ExpansionMain} we conclude 
 	\begin{equation}\label{eq:equiProofDelBkNEW}
 		k^{-1}\iota^*\overline{\partial}\log(c+B^\eta_k)=k^{-1}\iota^*\left(\frac{\overline{\partial}B^\eta_k}{c+B^\eta_k}\right)=-i\operatorname{mv}(\eta)\frac{\xi}{\sigma_P(\xi)}+O(k^{-1})
 	\end{equation}
 	on \(bM\) in \(\mathscr{C}^\infty\)-norm. It follows that there exists a constant \(C_1>0\) such that
 	\begin{equation}\label{eq:equiProofDelBk2NEW}
 		\left|\int_{bM}k^{-1}\iota^*\left(\frac{\overline{\partial}B^\eta_k}{c+B^\eta_k}\wedge\psi\right)+\int_{bM}i\operatorname{mv}(\eta)\frac{\xi}{\sigma_P(\xi)}\wedge \iota^*\psi\right|\leq C_1k^{-1}\|\psi\|_{\mathscr{C}^0(\overline{\domain},\Lambda^{n,n}T^*\C^{n+1})}
 	\end{equation}
 	for all \(\psi\in \Omega^{n,n}(M)\) and \(k>0\). 
 	Let  \(\psi\in\Omega^{n,n}(M)\) be arbitrary.  By Stokes theorem we find that
 	\begin{eqnarray*}
 		\int_M  \overline{\partial}\log(c+B^\eta_k)\wedge d\psi=&&\int_M \overline{\partial}\log(c+B^\eta_k)\wedge \partial\psi=\int_M d\log(c+B^\eta_k)\wedge \partial\psi\\
 		=&& \int_{bM}  \log(c+B^\eta_k\circ\iota)\wedge \iota^*(\partial\psi)+\int_M \log(c+B^\eta_k)\wedge \partial\overline{\partial}\psi.
 	\end{eqnarray*}
 	Then it follows from Lemma~\ref{lem:logBkbounded} that there exists a constant \(C_2>0\) such that
 	\begin{equation}\label{eq:equiProofSecondIntNEW}
 		\left|\int_M   \overline{\partial}\log(c+B^\eta_k)\wedge d\psi\right|\leq C_2(\log(k)+1)\|\psi\|_{\mathscr{C}^2(M,\Lambda^{n,n}\C T^*M)}
 	\end{equation}
 	holds for all \(\psi\in \Omega^{n,n}(M)\) and \(k>0\).
 	Using Stokes theorem again we find
 	\begin{equation}\label{eq:equiProofStokesNEW}
 		\int_M\partial \overline{\partial}\log(c+B^\eta_k)\wedge \psi= \int_{bM}\iota^*(\overline{\partial}\log(c+B^\eta_k)\wedge \iota^*\psi+\int_M \overline{\partial}\log(c+B^\eta_k)\wedge d\psi.
 	\end{equation}
 	Finally, the claim follows from \eqref{eq:equiProofDelBkNEW}, \eqref{eq:equiProofSecondIntNEW} and \eqref{eq:equiProofStokesNEW}.
 \end{proof}

\subsection{Proof of Theorem~\ref{thm:SequenceRndZerosIntro}} 
Consider \(k\in\N\) and put
\(A_\infty=\Pi_{k=1}^\infty A_k\), \(d\mu_{\infty}=\Pi_{k=1}^\infty d\mu_k\). Theorem~\ref{thm:SequenceRndZerosIntro} follows from the next result.
\begin{theorem}\label{thm:SequenceRndZeros}
		Under the assumptions of Theorem~\ref{thm:ExpactationRndZeros} and with the notation above, assuming that \(M\) is compact, we have for \(\mu_\infty\)-almost every \(u=(u_k)_{k\in \N}\in A_{\infty}\) that \((\divisor{u_k},\psi)\) is well-defined for all \(\psi\in \Omega^{n,n}(M)\) and \(k\in\N\). Furthermore, given \(\psi\in \Omega^{n,n}(M)\), we have 
		\[\lim_{k\to\infty} \left(k^{-1}\divisor{u_k},\psi\right) = \frac{\operatorname{mv}(|\chi|^2)}{2\pi}\int_{b\domain}\frac{\xi}{\sigma_P(\xi)}\wedge \iota^*\psi\]
		for \(\mu_\infty\)-almost every \(u=(u_k)_{k\in \N}\in A_{\infty}\).  
		In addition, there exist \(C,k_0>0\) such that for any \(k\geq k_0\) one has
		\begin{eqnarray*}
			\mu_k\left(\left\{f\in A_k\colon\left|\left(k^{-1}\divisor{f}, \psi\right)-\frac{\operatorname{mv}(|\chi|^2)}{2\pi}\int_{bM}\frac{\xi}{\sigma_P(\xi)}\wedge \iota^*\psi\right|\geq\frac{\|\psi\|_{\mathscr{C}^2(M,\Lambda^{n,n}\C T^*M)}}{\sqrt{k}} \right\}\right)\leq \frac{C}{\sqrt{k}}
		\end{eqnarray*}
		for all \(\psi\in \Omega^{n,n}(M)\).
	\end{theorem}
\begin{proof}
	From Theorem~\ref{thm:PoincareLelongMfdBoundary} and Lemma~\ref{lem:UnregularFunctionsAreZeroSetDomain} it follows that \((\divisor{u_k},\psi)\), \(k\in\N\), is well-defined for all \(\psi\in \Omega^{n,n}(M)\) and  \(\mu_\infty\)-almost every \(u=(u_k)_{k\in \N}\in A_{\infty}\) which proves the first part of the claim. Given  \(\psi\in \Omega^{n,n}(M)\) and \(k\in \N\) put 
	\[R_k(\psi):=\int_{A_k}\left|\left(k^{-1}\divisor{f},\psi\right)-\frac{\operatorname{mv}(|\chi|^2)}{2\pi}\int_{b M} \frac{\xi}{\sigma_P(\xi)} \wedge \iota^*\psi\right|^2d\mu_k(f).\]
	We will show that there exist \(k_0\in\N\) and a constant \(C>0\) such that \(R_k(\psi)\) exists for all \(k\geq k_0\) with \(R_k(\psi)\leq Ck^{-\frac{3}{2}}\|\psi\|^2_{\mathscr{C}^2(M,\Lambda^{n,n}\C T^*M)}\) for all \(k\geq k_0\) and \(\psi\in  \Omega^{n,n}(M)\). From the Markov inequality, the last part of the claim follows. Furthermore, considering \(k\in\N\), it follows that \(\sum_{k=k_0}^\infty R_k(\psi)\leq C_0 \|\psi\|^2_{\mathscr{C}^2(M,\Lambda^{n,n}\C T^*M)} \) for some constant \(C_0\) (independent of \(\psi\)) which leads to the second part of the claim.
	
	Given  \(\psi\in  \Omega^{n,n}(M)\) and \(k\in \N\) put 
	\[R'_k(\psi):=\left|\frac{i}{2\pi k}\int_M\left(\partial\overline{\partial}\log(1+B_k)\right)\wedge\psi-\frac{\operatorname{mv}(|\chi|^2)}{2\pi}\int_{b M} \frac{\xi}{\sigma_P(\xi)} \wedge \iota^*\psi\right|^2.\]
	From Theorem~\ref{thm:equidistribution} it follows that there exists a constant \(C_1>0\) such that 
	\[R'_k(\psi)\leq C_1k^{-\frac{3}{2}}\|\psi\|^2_{\mathscr{C}^2(M,\Lambda^{n,n}\C T^* M)} \]
	holds for all \(k\in \N\) and  \(\psi\in  \Omega^{n,n}(M)\).
	Since \(\mathcal{Z}_f(\psi)=\left(\divisor{f},\psi\right)\) we find with Theorem~\ref{thm:ExpactationRndZeros} and \(\mathbb{V}_k\) as in Theorem~\ref{thm:VarianceZerosEstimate} that
	\[R_k(\psi)\leq 2k^{-2}\mathbb{V}_k(\mathcal{Z}_f(\psi))+2R'_k(\psi)\int_{A_k}d\mu_k.\]
	Then it follows from Theorem~\ref{thm:VarianceZerosEstimate} with \(\varepsilon=\frac{1}{2}\) that there are \(k_0\in \N\) and a constant \(C>0\) such that  \(R_k(\psi)\leq Ck^{-\frac{3}{2}}\|\psi\|^2_{\mathscr{C}^2(M,\Lambda^{n,n}\C T^*M)}\) for all \(k\geq k_0\) and \(\psi\in  \Omega^{n,n}(M)\).
\end{proof}


\bibliographystyle{plain}	

\end{document}